\documentclass{amsart}
\usepackage{amssymb}
\usepackage{amsmath}
\usepackage{subfigure}
\usepackage{epsfig}
\usepackage[utf8]{inputenc}
\usepackage{mathabx}
\usepackage{graphicx}
\usepackage{xcolor}
\usepackage[T1]{fontenc}
\usepackage[all]{xy}
\usepackage{xcolor}
\usepackage{comment}
\usepackage{stmaryrd}
\usepackage{enumitem}
\usepackage{hyperref}
\hypersetup{colorlinks=true,linkcolor=blue,citecolor=red,urlcolor=blue}

\newcommand\ad{\mathbb{A}}

\newcommand\oink{\mathcal O}

\newcommand\Da{\mathfrak{D}}

\newcommand\Ea{\mathfrak{E}}

\newcommand\bmattrix[4]{\left(\begin{array}{cc}#1&#2\\#3&#4\end{array}\right)}
\newcommand\sbmattrix[4]{\textnormal{\scriptsize$\left(\begin{array}{cc}#1&#2\\#3&#4\end{array}\right)$\normalsize}}

\newcommand\matrici{\mathbb{M}}

\newcommand{\rf}{\mathfrak{r}}

\theoremstyle{definition} 
\newtheorem{theorem}{Theorem}[section]
\newtheorem{lemma}[theorem]{Lemma}
\newtheorem{prop}[theorem]{Proposition}
\newtheorem{corollary}[theorem]{Corollary}

\newtheorem{rem}[theorem]{Remark}
\newtheorem{ex}[theorem]{Example}
\newtheorem{defi}[theorem]{Definition}

\numberwithin{equation}{section}

\title[Quotients of trees by Hecke congruence subgroups]
{Quotients of the Bruhat-Tits tree by function field analogs of the Hecke congruence subgroups}

\author{Claudio Bravo}

\begin{document}
\maketitle

\footnotesize
\textbf{MSC Numbers (2020): 20G30-20E08 (primary), 11R58-14H60 (secondary)}

\normalsize

\begin{abstract}
Let $\mathcal{C}$ be a smooth, projective and geometrically integral curve defined over a finite field $\mathbb{F}$. For each closed point $P_{\infty}$ of $\mathcal{C}$, let $R$ be the ring of functions that are regular outside $P_{\infty}$, and let $K$ be the completion at $P_{\infty}$ of the function field of $\mathcal{C}$. In order to study groups of the form $\mathrm{GL}_2(R)$, Serre describes the quotient graph $\mathrm{GL}_2(R) \backslash \mathfrak{t}$, where $\mathfrak{t}$ is the Bruhat-Tits tree defined from $\mathrm{SL}_2(K)$. In particular, Serre shows that $\mathrm{GL}_2(R) \backslash \mathfrak{t}$ is the union of a finite graph and a finite number of ray shaped subgraphs, which are called cusps. It is not hard to see that finite index subgroups inherit this property.

In this work we describe the associated quotient graph $\mathrm{H} \backslash \mathfrak{t}$ for the action on $\mathfrak{t}$ of the group $\mathrm{H}= \left\lbrace \sbmattrix{a}{b}{c}{d} \in \text{GL}_2(R) : c \equiv 0 \, (\text{mod } I)\right\rbrace
\normalsize$, where $I$ is an ideal of $R$. More specifically, we give a explicit formula for the cusp number of $\mathrm{H} \backslash \mathfrak{t}$. Then, by using Bass-Serre Theory, we describe the combinatorial structure of $\mathrm{H}$. These groups play, in the function field context, the same role as the Hecke congruence subgroups of $\mathrm{SL}_2(\mathbb{Z})$.
\end{abstract}


\section{Introduction}\label{Section Intro}

Geometric group theory is the area of mathematics devoted to study group-theoretic problems via topological and geometric methods. A classical example of this approach is the characterization of discrete subgroups of real Lie groups by their action on symmetric spaces. In the latter seventies, Bass and Serre initiated a program for study certain groups by their actions on trees, i.e. on connected graphs without cycles. This theory is currently known as Bass-Serre theory (c.f \cite{S}). Bass-Serre theory can be applied to study discrete subgroups of $\mathrm{GL}_2(K)$, where $K$ is a discrete valuation field. For instance, Serre proves in \cite[Chapter I, \S 3.3]{S} that a discrete subgroup of $\mathrm{SL}_2(K)$ is torsion-free if and only if it acts freely on a tree. This result gives another proof of a theorem due to Ihara, which states that any discrete torsion-free subgroup of $\mathrm{SL}_2(\mathbb{Q}_p)$ is free (cf.~\cite[Theorem 1]{Ihara}). 

In \cite[Chapter II, \S 1.1]{S} Serre introduces the Bruhat-Tits tree $\mathfrak{t}=\mathfrak{t}(K)$ associated to the group $\mathrm{SL}_2$ at a discrete valuated field $K$. Moreover, the same author shows that $\mathrm{GL}_2(K)$ acts on $\mathfrak{t}$. Let
$R$ be the coordinate ring of an affine open set of a smooth, projective, geometrically integral curve $\mathcal{C}$, defined over a field $\mathbb{F}$, with a unique point $P_{\infty}$ at infinity. In \cite[Chapter II, \S 2]{S} Serre uses $\mathfrak{t}$ in order to study groups of the form $\mathrm{GL}_2(R)$. Indeed, the closed point $P_{\infty}$ gives rise to a discrete valuation $\nu$ on the function field $k$ of $\mathcal{C}$ and hence we have an action of $\mathrm{GL}_2(R)$, as a subgroup of  $\mathrm{GL}_2(K)$, on the Bruhat-Tits tree $\mathfrak{t}(K)$ defined from the completion $K$ of $k$ at $P_{\infty}$. In this situation, Serre
gave the following description of the quotient graph $\mathrm{GL}_2(R)  \backslash \mathfrak{t}$. 

\begin{theorem} \cite[Chapter II, Theorem 9]{S}\label{teo serre quot}
Let $\mathfrak{t}$ be the local Bruhat-Tits tree defined by the group $\text{SL}_2$ at the completion $K$ associated to the valuation induced by $P_{\infty}$ (cf.~\S \ref{subsection BTT}). Then, the graph $\mathrm{GL}_2(R) \backslash \mathfrak{t}$ is combinatorially finite, i.e. it is obtained by attaching a finite number of infinite half lines, called cuspidal rays, to a certain finite graph $Y$. The set of such cuspidal rays is indexed by the elements of the Picard group $\mathrm{Pic}(R)=\mathrm{Pic}(\mathcal{C})/\langle \overline{P_{\infty}}\rangle$.
\end{theorem}


Using Bass-Serre Theory, Serre concludes the following structural result for the groups $\mathrm{GL}_2(R)$ defined above.

\begin{theorem}\cite[Chapter II, Theorem 10]{S}\label{teo serre grup}
There exists a finitely generated group $H$, and a family $ \lbrace (I_{\sigma}, \mathcal{P}_{\sigma}, \mathcal{B}_{\sigma}) \rbrace_{\sigma \in \mathrm{Pic}(R)}$, where:
\begin{itemize}
\item[1.] $I_\sigma$ is an $R$-fractional ideal and $\mathcal{P}_\sigma= (\mathbb{F}^{*}\times \mathbb{F}^{*})\ltimes I_\sigma$,
\item[2.] $\mathcal{B}_{\sigma}$ is a group with canonical injections $\mathcal{B}_{\sigma} \rightarrow H$ and $\mathcal{B}_{\sigma} \rightarrow \mathcal{P}_{\sigma}$,
\end{itemize}
such that $\mathrm{GL}_2(R)$ is isomorphic to the sum of $\mathcal{P}_\sigma$, for $\sigma \in \mathrm{Pic}(R)$, and $H$, amalgamated along their common subgroups $\mathcal{B}_\sigma$ according to the above injections.
\end{theorem}

Moreover, Serre describes the structure of $\mathrm{GL}_2(R)$ as an amalgamated sum in certain cases, by explicitly describing the corresponding quotient graphs. This work considers, for example, the cases $\mathcal{C}=\mathbb{P}^1_{\mathbb{F}}$, for $\text{deg}(P_{\infty}) \in \lbrace 1, 2,3,4 \rbrace$, or when $\mathcal{C}$ is a curve of genus $0$ without rational points and $\text{deg}(P_{\infty})=2$. The case $\mathcal{C}=\mathbb{P}^1_{\mathbb{F}}$ and $\deg(P_{\infty})=1$ reduces to a classical result, now called Nagao's Theorem (cf.~\cite{N}), where the corresponding quotient graph is a ray. 

Theorem \ref{teo serre quot} and Theorem \ref{teo serre grup} have some interesting consequences for the involved groups. For instance, we can apply Theorem \ref{teo serre grup} in order to show that $\mathrm{GL}_2(R)$ is never finitely generated. Also, the preceding theorems allow us to study the homology and cohomology groups of $\mathrm{GL}_2(R)$. Indeed, in \cite[Chapther II, \S 2.8]{S} Serre applies this approach in order to understand the homology group of $\mathrm{GL}_2(R)$ with coefficient in certain modules. For instance, he concludes that $H_i(\mathrm{GL}_2(R), \mathbb{Q})=0$, for all $i >1$, and $H_1(\mathrm{GL}_2(R), \mathbb{Q})$ is a finite dimensional $\mathbb{Q}$-vector space. Then, in \cite[Chapther II, \S 2.9]{S} Serre interprets the Euler-Poincar\'e characteristic of certain subgroups $G$ of $\mathrm{GL}_2(R)$ in terms of the relative homology of $G$ modulo a representative system of its conjugacy classes of maximal unipotent subgroups.

In order to prove Theorem \ref{teo serre quot}, Serre makes an extensive use of the theory of vector bundles of rank $2$ over $\mathcal{C}$. On the other hand, Mason in \cite{M} gives a more elementary approach which involves substantially less algebraic geometry. This point of view only requires the Riemann-Roch Theorem and some basic notions about Dedekind rings. The price to pay for this simplicity is that one is not able to prove the finiteness of the diameter of graph $Y$ in Theorem \ref{teo serre quot}. However, Mason applies this result on quotient graphs in order to study the lowest index non-congruence subgroups of $\mathrm{GL}_2(R)$.

In a more general context, let $\mathbf{G}$ be an arbitrary reductive algebraic $k$-group. We can define a poly-simplicial complex $\mathcal{X}(\mathbf{G},K)$ associated to the group $\mathbf{G}$ and the field $K$. This topological space is called the building of $\mathbf{G}(K)$, and this notion generalizes the definition of the Bruhat-Tits tree. When $R = \mathbb{F}[t]$ and $\mathbf{G}$ is split over $k$, there exists a result that generalizes Nagao's theorem, which describes the structure of the quotient space $\overline{\mathcal{X}}=  \textbf{G}(\mathbb{F}[t]) \backslash \mathcal{X}$ associated to the action of $\textbf{G}(\mathbb{F}[t])$ on $\mathcal{X}=\mathcal{X}(\textbf{G},\mathbb{F}((t^{-1})))$. This result is due to Soul\'e and described in \cite[Theorem 1]{So}. Soul\'e shows that $\overline{\mathcal{X}}$ is isomorphic to a sector of
$\mathcal{X}$, which is the analog of a ray in the general building context. Then, in the same article, the author describes $\textbf{G}(\mathbb{F}[t])$ as an amalgamated sum of certain well-known subgroups. This structural result can be extended to the context where $\mathbf{G}$ is an isotrivial $k$-group, i.e.~a reductive $k$-group that splits in the composite field $\ell=\mathbb{L} k$, for a finite extension $\mathbb{L}$ of $\mathbb{F}$. This problem has been developed by Margaux in \cite{Margaux}. Indeed, Margaux manages to prove the same result as Soul\'e, replacing the condition ``split'' by ``isotrivial''. Finally, in the context of quasi-split groups, with L.~Arenas-Carmona, B.~Loisel and G.~Lucchini Arteche, we extend Theorem \ref{teo serre quot} and Theorem \ref{teo serre grup} to the special unitary groups of split rank one, which are the smallest quasi-split non-split reductive groups. This results are described in \cite{ABLL}, and they are not covered by the preceding works.

In the context where $\mathbb{F}$ is a finite field, one of the strongest results about the structure of the preceding quotients that exists in the literature is due to Bux, K\"ohl and Witzel in \cite{B}. In order to present this result, let $\mathcal{S}$ be a set of closed points in $\mathcal{C}$, and denote by $\mathcal{O}_\mathcal{S}$ the ring of functions that are regular outside $\mathcal{S}$. In particular, we have $\mathcal{O}_{P_\infty}=R$. Assume that $\mathbf{G}$ is an isotropic and non-commutative algebraic $k$-group and let $\mathcal{X}_\mathcal{S}$ be the product of buildings $\prod_{P \in \mathcal{S}} \mathcal{X}(\mathbf{G}, k_P)$, where $k_P$ is the completion of $k$ at $P$. Choose a particular realization $\mathbf{G}_{\mathrm{real}}$ of $\mathbf{G}$ as an algebraic set of some affine space defined over $k$. Given this realization, we define $G$ as the group of $\mathcal{O}_{\mathcal{S}}$-points of $\mathbf{G}_{\mathrm{real}}$. Then, \cite[Proposition 13.6]{B} shows that there exists a constant $L$ and finitely many pairwise sectors $Q_1, \cdots, Q_s$ of $\mathcal{X}_\mathcal{S}$ such that the $G$-translates of the $L$-neighborhood of $\bigcup_{i=1}^{s} Q_i$ cover $\mathcal{X}_\mathcal{S}$.

Note that almost all the previous results are specific to groups of points of certain reductive groups. Thus, it is natural to seek for extensions to certain subgroups of them.
Widely speaking, the goal of this work is to study a certain family of congruence subgroups $\mathrm{H}$ of $\mathrm{GL}_2(R)$, which is a direct analog of the Hecke congruence subgroups of $\text{SL}_2(\mathbb{Z})$ in the function field context. We do this through the analysis of the associated group actions on trees. Specifically, in analogy with Theorem \ref{teo serre quot}, we show that $\mathrm{H} \backslash \mathfrak{t}$ is combinatorially finite and we give a explicit formula for the number its cuspidal rays. Then, by using Bass-Serre Theory, we describe the combinatorial structure of $\mathrm{H}$ and its abelianization. Finally, we present some interesting examples in the context where $\mathcal{C}=\mathbb{P}^1_\mathbb{F}$.



\section{Context and Main results}\label{Section on the principal problem}

In order to introduce the main results of this work, we use the following definitions and notations. Let $\mathcal{C}$ be a smooth, projective, geometrically integral $\mathbb{F}$-curve and let $k$ be its function field. Since in the sequel we use the spinor genera theory in some proofs, and this theory is set in the context where \underline{the ground field $\mathbb{F}$ is finite}, we assume this throughout and we denote its cardinality by $q=p^s$. A $\mathcal{C}$-order $\mathfrak{D}$ on the matrix algebra $\mathbb{M}_2(k)$ is a locally free sheaf of $\mathcal{O}_{\mathcal{C}}$-algebras whose generic fiber is $\mathbb{M}_2(k)$ (cf.~Definition \ref{definition of max and eichler orders}). Analogously, an $R$-order is a locally free $R$-algebra. As we explain in \S \ref{Section Spinor}, any $R$-order can be extended to a $\mathcal{C}$-order by choosing an arbitrary local order at the point $P_{\infty} \in \mathcal{C}$. We say that a $\mathcal{C}$-order is maximal if its completions are maximal at all places of $\mathcal{C}$. By definition, a split maximal order is an order that is $\text{GL}_2(k)$-conjugate to the sheaf 
$$
\small
\mathfrak{D}_D= \left( \begin{array}{cc}
\mathcal{O}_{\mathcal{C}} &   \mathfrak{L}^D\\
\mathfrak{L}^{-D}  & \mathcal{O}_{\mathcal{C}} \end{array} \right),
\normalsize
$$
where $D$ is an arbitrary divisor on $\mathcal{C}$, and where $\mathfrak{L}^D$ is the invertible sheaf defined in every open set $U \subseteq \mathcal{C}$ by

$$
\mathfrak{L}^D(U)= \left\lbrace f \in k: \text{div}(f)|_{U}+D|_{U} \geq 0 \right\rbrace .
$$

In general, an Eichler $\mathcal{C}$-order is a sheaf-theoretical intersection of two maximal $\mathcal{C}$-orders. We define a specific family of Eichler $\mathcal{C}$-orders $\mathfrak{E}_D$ by taking
\begin{equation}\label{eq de eD}
\mathfrak{E}_D= \mathfrak{D}_D \cap \mathfrak{D}_0,
\end{equation}
 where $D$ is an effective divisor. Let $U_0$ be the open set in $\mathcal{C}$ defined as the complement of $\lbrace P_{\infty} \rbrace$. We define $\mathrm{H}_D$ as the group of invertible elements in $\mathfrak{E}_D(U_0)$. In other words
\begin{equation}\label{eq eichler}
\small
\text{H}_D= \left\lbrace \left( \begin{array}{cc}
a &   b\\
c & d \end{array} \right) \in \text{GL}_2(R) : c \equiv 0 \, (\text{mod } I_D)\right\rbrace ,
\normalsize
\end{equation}
where $I_D$ is the $R$-ideal defined as $I_D=\mathfrak{L}^{-D}(U_0)$. Then, the family of groups $\mathcal{H}=\lbrace \mathrm{H}_D: D \, \,\text{effective divisor} \rbrace$ plays the same role as the Hecke congruence subgroups of $\mathrm{SL}_2(\mathbb{Z})$ in the function field context.

The objective of this work is to characterize the quotient graph associated to the action of $\mathrm{H}_D$ on the Bruhat-Tits tree $\mathfrak{t}$, to subsequently describe the combinatorial structure of $\mathrm{H}_D$. Note that $\mathrm{H}_D$ naturally contains the kernel of $\mathrm{GL}_2(R) \to \mathrm{GL}_2(R/I_D)$. This implies, as we prove in Corollary \ref{Cor comb fin} (which follows from a lemma by Serre in \cite{Se2}), that the quotient graph $\mathrm{H}_D \backslash \mathfrak{t}$ is combinatorially finite, and the number of cuspidal rays of $\mathfrak{t}_D= \text{H}_D \backslash \mathfrak{t}$ is equal to the number of $\text{H}_D$-orbits in $\mathbb{P}^1(k)$. The previous observation is useful in the context where $D$ has small degree. Unfortunately, this set of $\text{H}_D$-orbits is really hard to characterize in the general case. Another obstruction for a direct computation of $\mathfrak{t}_D$ is that $\mathrm{H}_D$ is not a normal subgroup of $\mathrm{GL}_2(R)$. In particular, $\mathrm{GL}_2(R) \backslash \mathfrak{t}$ is not always a quotient of $\text{H}_D \backslash \mathfrak{t}$.

In order to present our main result we introduce some additional notations. For any divisor $D$ on $\mathcal{C}$, we denote by $\overline{D}$ its linear equivalence class. 
Also, we denote by $\lfloor a \rfloor$ the largest integer not exceeding $a \in \mathbb{R}$. Observe that, when $D=0$, we have $\text{H}_D=\text{GL}_2(R)$. In particular, the next theorem can be considered as a partial generalization of Serre's result on the structure of quotient graphs. 
\begin{theorem}\label{teo cusp}
Let $D$ be an effective divisor, which we write as $D=\sum_{i=1}^r n_i P_i$, where the points $P_1, \dots, P_r, P_{\infty}$ are all different. Then, the graph $\mathfrak{t}_D=\mathrm{H}_D\backslash \mathfrak{t}$ is obtained by attaching a finite number of rays, called cuspidal rays, to a certain finite graph $Y\subset \mathfrak{t}_D$. The cardinality $\mathfrak{c}_D$ of the set of such cuspidal rays satisfies
\small
\begin{equation}\label{numero de patas}
\mathfrak{c}_D \leq c(\mathrm{H}_D)
\footnotesize
:=2^r |g(2)| \left| \frac{ 2\mathrm{Pic}(\mathcal{C})+ \langle \overline{P_{\infty}} \rangle}{ \langle \overline{P_{\infty}} \rangle} \right|
\normalsize
\left(  1+ \frac{1}{q-1} \prod_{i=1}^r \left( q^{\text{deg}(P_i)\lfloor  \frac{n_i}{2} \rfloor}-1\right) \right),
\end{equation}
\normalsize
where $g(2)$ is the maximal exponent-2 subgroup of $\mathrm{Pic}(R)$. Moreover, equality holds when $g(2)$ is trivial and each $n_i$ is odd. 
\end{theorem}
Note that $g(2)$ is trivial in various cases: for example, when $\mathcal{C}=\mathbb{P}^1_{\mathbb{F}}$ and $P_{\infty}$ has odd degree, or when $\mathcal{C}$ is an elliptic curve with no non-trivial $2$-torsion rational points.
The previous results give a more precise description than \cite[\S 3.3, Lemma 8]{Se2}, which in fact only says that the set of cuspidal rays is finite. Indeed, in Theorem \ref{teo cusp}, we have a control on the number of cusps, and, in the case where $g(2)$ is trivial and each $n_i$ is odd, we have an explicit expression to compute it. When $\mathcal{C}=\mathbb{P}^1_\mathbb{F}$, Theorem 3.3 of \cite{A5} characterizes a set of representatives for all but finitely many vertex classes. In particular, this theorem gives a set of representatives for the action of $\mathrm{H}_D$ on the ends of $\mathfrak{t}$. 

Now, by using Bass-Serre theory and Theorem \ref{teo cusp} we can deduce the following general result on the combinatorial structure of $\text{H}_D$. This can be considered as a partial generalization of Theorem \ref{teo serre grup}, and as a more detailed description than \cite[\S 3.3, Lemma 8]{Se2}.

\begin{theorem}\label{teo grup}
In the notation of Theorem \ref{teo cusp}, assume that each $n_i$ is odd and $g(2)= \lbrace e \rbrace$. Then, there exist a finitely generated group $H$, two sets of indices, denoted by $\mathbf{D}$ and $\mathbf{I}$, and a family $\lbrace (I_{\sigma}, \mathcal{P}_{\sigma}, \mathcal{B}_{\sigma}) : \sigma \in  \mathbf{D} \sqcup \mathbf{I} \rbrace$, where
\begin{itemize}
\item[1.] $\mathrm{Card}(\mathbf{D})= 2^r [2\mathrm{Pic}(\mathcal{C})+ \langle \overline{P_{\infty}} \rangle: \langle \overline{P_{\infty}} \rangle],$ and $\mathrm{Card}(\mathbf{I})=c(\mathrm{H}_D)-\mathrm{Card}(\mathbf{D}),$
\item[2.] $ I_\sigma$ is an $R$-ideal contained in $I_D$
\item[3.] $\mathcal{P}_\sigma= (\mathbb{F}^{*}\times \mathbb{F}^{*})\ltimes I_\sigma$, for any $\sigma \in \mathbf{D},$
while $\mathcal{P}_{\sigma} =  \mathbb{F}^{*}\times I_\sigma$, for any $\sigma \in \mathbf{I},$
\item[4.] $\mathcal{B}_{\sigma}$ is a group with canonical injections $\mathcal{B}_{\sigma} \rightarrow H$ and $\mathcal{B}_{\sigma} \rightarrow \mathcal{P}_{\sigma}$, for any $\sigma \in \mathbf{D} \sqcup \mathbf{I}$,
\end{itemize}
such that $\mathrm{H}_D$ is isomorphic to the sum of $\mathcal{P}_\sigma$, for $\sigma \in \mathbf{D}\sqcup \mathbf{I}$, and $H$, amalgamated along their common subgroups $\mathcal{B}_\sigma$ according to the above injections.
\end{theorem}

As we cite in \S \ref{Section Intro}, Serre in \cite[Chapter II, \S 2.8]{S} gives a series of results that relate the homology of congruence subgroups of $\mathrm{GL}_2(R)$ with the structure of its quotient graphs. We can apply them to our context. In particular, the following result, which gives a more precise description of the abelianization of $\mathrm{H}_D$, is a consequence of Theorem \ref{teo cusp}.

\begin{theorem}\label{teo grup ab}
With the same notation and hypotheses of Theorem \ref{teo grup}, the abelianization $\mathrm{H}_D^{\mathrm{ab}}$ of $\mathrm{H}_D$ is
\begin{itemize}
    \item[1.]  finite generated, if each $n_i$ equals one, or
    \item[2.] the direct product of a finitely generated group with a infinite dimensional $\mathbb{F}_p$-vector space, in any other case.
\end{itemize}
\end{theorem}

Finally, we present an interesting example in the context where $\mathcal{C}=\mathbb{P}^1_{\mathbb{F}}$, $D=P$ is a closed point and $\deg(P_{\infty})=\deg(P)=1$. Indeed, by using these hypotheses on $\mathcal{C}$ and $D$ we can show that the graph $\mathfrak{t}_D=\mathrm{H}_D\backslash \mathfrak{t}$ is isomorphic to a double ray (cf.~Proposition \ref{teo quot t}). We subsequently describe $\mathrm{H}_D$ as an explicit amalgamated sum (cf.~Proposition \ref{teo group t}) and compute its abelianization (cf.~Proposition \ref{teo group ab t}).

\section{Preliminaries on the Bruhat-Tits tree}\label{Section BTT}

\subsection{Conventions and notations for graphs}

We recall some basic definitions on graphs and trees. We define a graph $\mathfrak{g}$ as a pair of sets $\text{V}=\text{V}(\mathfrak{g})$ and $\text{E}=\text{E}(\mathfrak{g})$, and three functions $s,t:E\rightarrow V$ and $r:E\rightarrow E$ satisfying the identities $r(a)\neq a$, $r\big(r(a)\big)=a$ and $s\big(r(a)\big)=t(a)$, for every $a \in E$. In all that follows $V$ and $E$ are called vertex and edge set, respectively, and the functions $s,t$ and $r$ are called respectively source, target and reverse.
Our definition is chosen in a way that allows the existence of multiple edges and loops. Two vertices $v,w \in \text{V}$ are neighbors if there exists an edge $e \in \text{E}$ satisfying $s(e)=v$ and $t(e)=w$. The valency of a vertex $v$ is the cardinality of its set of neighboring vertices. A simplicial map $\gamma:\mathfrak{g} \rightarrow \mathfrak{g}'$ between graphs is a pair of functions $\gamma_V:\text{V}(\mathfrak{g})\rightarrow \text{V}(\mathfrak{g}')$ and $\gamma_E: \mathrm{E}(\mathfrak{g})\rightarrow \mathrm{E}(\mathfrak{g}')$ preserving the source, target and reverse functions. We say that a simplicial map $\gamma:\mathfrak{g} \rightarrow \mathfrak{g}'$ is an isomorphism if there exists another simplicial map $\gamma':\mathfrak{g}' \rightarrow \mathfrak{g}$ such that $\gamma_V \circ \gamma'_V= \mathrm{id}_{\text{V}(\mathfrak{g})}$, $\gamma'_V \circ \gamma_V= \mathrm{id}_{\text{V}(\mathfrak{g'})}$, $\gamma_E \circ \gamma'_E= \mathrm{id}_{\text{E}(\mathfrak{g})}$ and $\gamma'_E \circ \gamma_E= \mathrm{id}_{\text{E}(\mathfrak{g'})}$. The group of automorphisms $\mathrm{Aut}(\mathfrak{g})$ of a graph $\mathfrak{g}$ is the set of isomorphism from $\mathfrak{g}$ to $\mathfrak{g}$, with the composition as a group law.

We say that a group $\Gamma$ acts on a graph $\mathfrak{g}$ is there exists an homomorphism from $\Gamma$ to $\mathrm{Aut}(\mathfrak{g})$. A group action of $\Gamma$ on a graph $\mathfrak{g}$ has no inversions if $g \cdot a\neq r(a)$, for every edge $a$ and every element $g\in\Gamma$. An action without inversions defines a quotient graph in a natural sense. Indeed, if $\Gamma$ acts on $\mathfrak{g}$ without inversions, then the vertex set of $\Gamma \backslash \mathfrak{g}$ corresponds to $\Gamma \backslash V$, and the edge set corresponds to $\Gamma \backslash E$.  When the action of $\Gamma$ has inversions, we can replace $\mathfrak{g}$ by its first 
barycentric subdivision in order to obtain a graph where $\Gamma$ acts without inversions. Then, the quotient graph is also defined in this case.

Let $\mathfrak{g}$ be a graph. A finite line in $\mathfrak{g}$, usually denoted by $\mathfrak{p}$, is a subgraph whose vertex and edge sets are $\mathrm{V}(\mathfrak{p})=\lbrace v_i \rbrace_{i=0}^{n}$ and $\mathrm{E}(\mathfrak{p})=\lbrace e_i, r(e_i) \rbrace_{i=0}^{n-1}$, where $s(e_i)=v_i$ and $t(e_{i})=v_{i+1}$, for all index $0 \leq i \leq n-1$. The length of $\mathfrak{p}$ is, by definition, $n=\mathrm{Card} (\mathrm{V}(\mathfrak{p}))-1=\mathrm{Card} (\mathrm{E}(\mathfrak{p}))/2$. We often emphasize the vertices $v_0$, the initial vertex of $\mathfrak{p}$, and $v_r$, the 
final vertex of $\mathfrak{p}$, by saying ``$\mathfrak{p}$ is a path connecting $v_0$ with $v_r$''. A graph $\mathfrak{g}$ 
is connected if, given two vertices $v,w \in \mathrm{V}(\mathfrak{g})$, there exists finite path $\mathfrak{p}$ connecting them. We define a ray $\mathfrak{r}$ in $\mathfrak{g}$ by replacing $n$ and $n-1$ by $\infty$ in the definition of finite line. A cycle in $\mathfrak{g}$ is a finite line with an additional pair of edges connecting $v_n$ with $v_0$. We define a tree as a connected graph without cycles.

A maximal path in $\mathfrak{g}$ is a doubly infinite line, i.e. the union of two rays intersecting only in one vertex. Let $\mathfrak{r}_1$ and $\mathfrak{r}_2$ be two rays whose vertex sets are respectively denoted by $V_1=
\left\lbrace v_i : i\in \mathbb{Z}_{\geq 0}\right\rbrace $ and $V_2=\left\lbrace v'_i: i\in \mathbb{Z}_{\geq 0}\right\rbrace$. We say that $\mathfrak{r}_1$ and $\mathfrak{r}_2$ are equivalent if there exists $t, i_0 \in \mathbb{Z}_{\geq 0}$ such that $v_{i}= v_{i+t}'$, for all $i \geq i_0$. In this case we write $\mathfrak{r}_1 \sim \mathfrak{r}_2$. We define the visual limit, also called the end set, $\partial_{\infty}(\mathfrak{g})$ of $\mathfrak{g}$ as the set of equivalence classes of rays $\rf$ in $\mathfrak{g}$. We denote the class of $\rf$ by $\partial_{\infty}(\mathfrak{r})$. By a cuspidal ray in a graph $\mathfrak{g}$, we mean a ray such that every non-initial vertex has valency two in $\mathfrak{g}$. A cusp in $\mathfrak{g}$ is an equivalence class of cuspidal rays in $\mathfrak{g}$. We denote the cusp set of $\mathfrak{g}$ by $\partial^{\infty}(\mathfrak{g})$. We say that a graph is combinatorially finite if it is the union of a finite graph and a finite number of cuspidal rays. In particular, when a graph is combinatorially finite its visual limit is also finite.


\subsection{The Bruhat-Tits tree}\label{subsection BTT}

Let $k$ be the function field of a smooth, projective, geometrically integral curve $\mathcal{C}$ defined over a field $\mathbb{F}$. Let $K$ be the completion of $k$ with respect to a discrete valuation $\nu: k^{*} \to \mathbb{Z}$, and let $\mathcal{O}$ be its ring of integers. Recall that a tree is a connected graph without cycles. 

An example of tree is the Bruhat-Tits building $\mathfrak{t}=\mathfrak{t}(K)$ associated to the reductive group $\text{SL}_2$ and the field $K$. In order to introduce this tree, we have to fix some definitions concerning lattices. Let $\pi \in \mathcal{O}$ be a fixed uniformizing parameter of $K$. A lattice in a $K$-vector space $V$ is a finitely generated $\mathcal{O}$-submodule of $V$, which generates $V$ as a vector space. Assume that $V$ is a two-dimensional $K$-vector space. Then, every lattice on $V$ is free of rank $2$. The group $K^{*}$ acts on the set of lattices by homothetic transformations. The vertex set of $\mathfrak{t}(K)$ can be defined as the set of homothetic classes of lattices in $V$. We adopt this convention. Let $\Lambda$ and $\Lambda'$ be two lattices in $V$. By the Invariant Factor Theorem of Algebraic Number Theory, there is an $\mathcal{O}$-basis $\lbrace e_1, e_2 \rbrace$ of $\Lambda$ and integers $a,b$ such that $\lbrace \pi^a e_1, \pi^b e_2 \rbrace$ is an $\mathcal{O}$-basis for $\Lambda'$. The set $\lbrace a,b \rbrace$ does not depend on the choice of basis for $\Lambda, \Lambda'$. Moreover, if we replace $\Lambda$ by $x\Lambda$, and $\Lambda'$ by $y \Lambda'$, where $x,y \in K^{*}$, then $\lbrace a,b \rbrace$ changes into $\lbrace a+c , b+c\rbrace$, where $c=\nu(y/x)$. So, the integer $|a-b|$ is called the distance between the classes $[\Lambda]$ and $[\Lambda']$. We define one pair of mutually reverse edges in $\mathfrak{t}(K)$ for each pair of lattice classes at distance one. This defines a graph, which can be proved to be a tree (cf.~\cite[Chapter II, \S 1, Theorem 1]{S}). The group $\text{GL}_2(K)$ acts on $\mathfrak{t}$ by $g \cdot [\Lambda]=[g(\Lambda)]$, for any $\mathcal{O}$-lattice $\Lambda \subset K^2$ and any $g \in \text{GL}_2(K)$. This induces an action of $\text{PGL}(V)=\text{PGL}_2(K)$ on $\mathfrak{t}$.

An order in $\mathbb{M}_2(K)$ is a lattice with a ring structure induced by the multiplication of $\mathbb{M}_2(K)$. We say that an order is maximal when it fails to be contained in any other order.
One can reinterpret the Bruhat-Tits tree for $\text{SL}_2$ in several ways. One of these arises from the following remark. There exists a bijective map from the vertex set of $\mathfrak{t}(K)$ to the set of maximal orders in $\mathbb{M}_2(K)$. Indeed, this function is defined by $[\Lambda] \mapsto \mathrm{End}_{\mathcal{O}}(\Lambda)$, which is valid, since the endomorphism rings of $\Lambda$ and $x\Lambda$ coincide for any $x \in K^{*}$. Moreover, under this identification, two maximal orders $\mathfrak{D}$ and $\mathfrak{D}'$ are neighbors if the pair $ \left\lbrace \mathfrak{D}, \mathfrak{D}'\right\rbrace$ is $\mathrm{GL}_2(K)$-conjugate to the pair $ \left\lbrace \sbmattrix {\oink}{\oink}{\oink}{\oink}, \sbmattrix {\oink}{\pi^{-1} \oink}{\pi \oink}{\oink} \right\rbrace.$

Another reinterpretation of the Bruhat-Tits tree for $\text{SL}_2$ comes from the topological structure of $K$, which is very useful in order to have a concrete representation of its visual limit. 
We denote by $B_a^{|r|}$ the closed ball in $K$ whose center is $a$ and radius is $|\pi^r|$. 
Then, we can define the function $\Sigma$ between the set of closed balls and the set of maximal orders in $\mathbb{M}_2(K)$ by $ B_a^{|r|} \mapsto  \mathrm{End}_{\mathcal{O}}( \Lambda_B ),$ where $\Lambda_B = \left\langle  
\binom{a}{1}, \binom{\pi^r}{0}
\right\rangle $. 
It follows from \cite[\S 4]{AAC} that $\Sigma$ is bijective. 
Thus, this induces a correspondence between the vertex set of $\mathfrak{t}=\mathfrak{t}(K)$ and the set of closed balls in $K$. 
Moreover, if we say that two balls are neighbors whenever one is a proper maximal sub-ball of the other, then $\Sigma$ induces an isomorphism of graphs. 
In other words, under the previous definition, we have that two balls $B$ and $B'$ are neighbors precisely when $\Sigma(B)$ and $\Sigma(B')$ are neighbors.
So, by using this reinterpretation of the Bruhat-Tits tree in terms of balls, it is straightforward that any ray $\mathfrak{r}$ in $\mathfrak{t}$ satisfies either $V(\mathfrak{r})=\left\lbrace B_{a}^{|r+n|} : n \in \mathbb{Z}_{\geq 0}  \right\rbrace$, for certain $a \in K$ and $r \in \mathbb{Z}$, or $V(\mathfrak{r})=\left\lbrace B_{0}^{|r-n|} : n \in \mathbb{Z}_{\geq 0}  \right\rbrace$, for certain $r \in \mathbb{Z}$. In the first case, the visual limit of $\mathfrak{r}$ can be identified with $a \in K$, and, in the second, we identify it with the point at infinity $\infty$.
This brief remark shows that the visual limit of the Bruhat-Tits tree $\mathfrak{t}=\mathfrak{t}(K)$ is in natural correspondence with the $K$-points of the projective line $\mathbb{P}^1$.
In all that follows, the equivalence classes of rays in $\partial_{\infty}(\mathfrak{t})$ are called ends of $\mathfrak{t}$. 
This set of ends is acted 
on naturally by the group $\mathrm{GL}_2(K)$, via Moebius transformations with coefficients in $K$. In fact, this action is compatible with the previously defined action of $\mathrm{GL}_2(K)$ on lattices, or the subsequent action on balls induced by the former (cf.~\cite[\S 4]{AAC}).
It follows from the density of $k$ in $K$, that for any finite line $\mathfrak{p}$ of $\mathfrak{t}$, there is a ray containing $\mathfrak{p}$ whose end corresponds to a rational element $s \in \mathbb{P}^1(k) \subset \mathbb{P}^1(K)$.


 \section{On combinatorially finite quotients of the Bruhat-Tits tree}\label{Section comb finite}

We keep the notation from last section. Here we give a detailed description of the quotient graphs of the Bruhat-Tits tree $\mathfrak{t}$ by certain subgroups of $\mathrm{GL}_2(k)$. In order to do this, we introduce the following notion.

 \begin{defi}\label{def good quotient}
Let $H$ be a subgroup of $\mathrm{GL}_2(k)$. We say that $H$ \textit{closes enough umbrellas} if there exists a finite family of rays $ \mathfrak{R}_{H}= \lbrace \mathfrak{r}_i \rbrace_{i=1}^{\gamma} \subset \mathfrak{t}$, each with a vertex set $\lbrace v_{n}(i) \rbrace_{n>0}^{\infty}$, where $v_n(i)$ and $v_{n+1}(i)$ are neighbors, satisfying each of the following statements:
\begin{itemize}
    \item[(a)] The set of ends of all rays in $\mathfrak{R}_{H}$ is a representative system of $H \backslash \mathbb{P}^1(k)$.
    \item[(b)] $H \backslash \mathfrak{t}$ is obtained by attaching all the images $\overline{\mathfrak{r}_i} \subseteq H \backslash \mathfrak{t}$ to a certain finite graph $Y_{H}$.
    \item[(c)] No $\mathfrak{r}_i$ contains a pair of vertices in the same $H$-orbit, and $\overline{\mathfrak{r}_i} \cap \overline{\mathfrak{r}_j} = \emptyset$, for each $i \neq j$.
    \item[(d)] For each index $i$ and each $n>0$, we have $ \mathrm{Stab}_H(v_{n}(i))\subseteq \mathrm{Stab}_H(v_{n+1}(i))$.
    \item[(e)] $\mathrm{Stab}_H(v_{n}(i))$ acts transitively on the set of  neighboring vertices  in $\mathfrak{t}$ of $v_{n}(i)$, other than $v_{n+1}(i)$.
\end{itemize}
In particular, if $H$ closes enough umbrellas, then $H \backslash \mathfrak{t}$ is combinatorially finite. Moreover, for any ray $\mathfrak{r} \subset \mathfrak{t}$ whose visual limits belongs to $\mathbb{P}^1(k)$, there exists a subray $\mathfrak{r}' \subseteq \mathfrak{t}$ satisfying conditions (d) and (e). Note that the notion of ``closing umbrellas'' corresponds to these two statements, while (a), (b) and (c) convey the idea of ``closing \textit{enough} umbrellas'', so as to have a good quotient graph.
\end{defi}

We say that a subgroup $H$ of $\mathrm{GL}_2(k)$ is net when each element in $H$ fails to admit a root of unit different than one as an eigenvalue. It follows from \cite[Chapter II, \S 2.1- \S 2.3]{Se2} that every net subgroup of $\mathrm{GL}_2(k)$ closes enough umbrellas. We say that two groups are commensurable if they have a common finite index subgroup. The notion of ``closing enough umbrellas'' behaves well when we pass to commensurable groups as is shown in the following results. This is probably known to experts but, as far as we are aware, a precise reference does not exist.

\begin{theorem}\label{Teo comb finito}
Let $H $ be a discrete subgroup of $\mathrm{GL}_2(k)$. Let $H' \subseteq \mathrm{GL}_2(k)$ be a group that is commensurable with $H$. If $H$ closes enough umbrellas, then $H'$ also closes enough umbrellas.
\end{theorem}

In order to prove this theorem, we analyze separately in the following two proposition the cases of subgroups of $H$ and of groups containing $H$.

\begin{prop}\label{prop comb finito up}
Let $H$ be a subgroup of $\mathrm{GL}_2(k)$. Assume that $H$ closes enough umbrellas. Let $H'\subset \mathrm{GL}_2(k)$ be a group containing $H$ as a finite index normal subgroup. Then, $H'$ also closes enough umbrellas.
\end{prop}

In order to prove the proposition, we need the following lemma.

\begin{lemma}\label{lemma comb finite}
Let $\mathfrak{g}$ be a combinatorially finite graph, and let $G$ be a finite group acting on this graph. Then, each cuspidal ray $\mathfrak{r}$ in $\mathfrak{g}$ has a finite number of vertices in the same $G$-orbit. In particular, $\mathfrak{r}$ has a subray whose image in $G \backslash \mathfrak{g}$ is a cuspidal ray, and hence $G \backslash \mathfrak{g}$ is a combinatorially finite graph.

\end{lemma}

\begin{proof}
By definition we have that there exists a set of rays $\mathfrak{R}=\lbrace \tilde{\mathfrak{r}}_i \rbrace_{i=1}^{\gamma}$ all contained in $\mathfrak{g}$, such that $\mathfrak{g}$ is obtained by attaching all $\tilde{\mathfrak{r}}_i$ to a certain finite graph $Y$. Let $\tilde{\mathfrak{r}}$ be a ray in $\mathfrak{R}$. Since $G$ acts simplicially on $\mathfrak{g}$, for each $g \in G$, the graph $g \cdot \tilde{\mathfrak{r}}$ is also a ray in $\mathfrak{g}$. Then, since $\mathfrak{g}$ is combinatorially finite, $g \cdot \tilde{\mathfrak{r}}$ has the same visual limit as some ray in $\mathfrak{R}$. First assume that $\partial_{\infty}(\tilde{\mathfrak{r}})=\partial_{\infty}(g \cdot \tilde{\mathfrak{r}})$. Then, $\mathfrak{r}^{\circ}:= \tilde{\mathfrak{r}} \cap (g \cdot \tilde{\mathfrak{r}})$ is a ray. Since each non-initial vertex of $\tilde{\mathfrak{r}}$ and $g \cdot \tilde{\mathfrak{r}}$ has valency two, we get $\mathfrak{r}^{\circ}=  \tilde{\mathfrak{r}}$ or $\mathfrak{r}^{\circ}=g \cdot \tilde{\mathfrak{r}}$. In other words, $\tilde{\mathfrak{r}} \subseteq g \cdot \tilde{\mathfrak{r}}$ or $\tilde{\mathfrak{r}} \supseteq g \cdot \tilde{\mathfrak{r}}$. Assume that $\tilde{\mathfrak{r}} \subseteq g \cdot \tilde{\mathfrak{r}}$, then 
$$\tilde{\mathfrak{r}}\subseteq g \cdot \tilde{\mathfrak{r}} \subseteq \cdots \subseteq g^k \cdot \tilde{\mathfrak{r}} \subseteq g^{k+1} \cdot \tilde{\mathfrak{r}}, \text{ for all } k \in \mathbb{Z}_{\geq 0}.$$
Since $G$ is finite, we get that $\tilde{\mathfrak{r}}=g \cdot \tilde{\mathfrak{r}}$. By an analogous argument we also prove that $\tilde{\mathfrak{r}}=g \cdot \tilde{\mathfrak{r}}$, when $\tilde{\mathfrak{r}} \supseteq g \cdot \tilde{\mathfrak{r}}$. We conclude that $g$ fixes every vertex in this case.

Now, assume that the visual limit of $g \cdot \tilde{\mathfrak{r}}$ is not $\partial_{\infty}(\tilde{\mathfrak{r}})$. Then, $\tilde{\mathfrak{r}} \cap (g \cdot \tilde{\mathfrak{r}})$ is a finite graph. So, for each index $i$, we define the ray $\tilde{\mathfrak{r}}_i'$ as the unique unbounded connected component of
$$\tilde{\mathfrak{r}}_i \smallsetminus \left( \bigcup_{\substack{ h \in G \\ \partial_{\infty}(\tilde{\mathfrak{r}}_i) \neq \partial_{\infty}(h \cdot \tilde{\mathfrak{r}}_i) } } \tilde{\mathfrak{r}}_i \cap (h \cdot \tilde{\mathfrak{r}}_i) \right) .  $$
By definition, and by the final statement in last paragraph, the ray $\tilde{\mathfrak{r}}_i'$ does not have two vertices in the same $G$-orbit. Since $\tilde{\mathfrak{r}}_i$ and $\tilde{\mathfrak{r}}_i'$ differ by a finite graph, the first assertion follows. In order to prove the last assertion, we say that $\tilde{\mathfrak{r}}_i'$ and $\tilde{\mathfrak{r}}_j'$ are $G$-equivalent if $\partial_{\infty}(\tilde{\mathfrak{r}}'_i) = \partial_{\infty}(g \cdot \tilde{\mathfrak{r}}_j')$, for some $g=g(i,j) \in G$. So, we define $\mathfrak{r}''_i \subseteq G \backslash \mathfrak{g}$ as the intersection of the images by $\pi:\mathfrak{g} \to G \backslash \mathfrak{g}$ of all rays $\tilde{\mathfrak{r}}_j'$ in the $G$-equivalence class of $\tilde{\mathfrak{r}}_i'$. We claim that $\mathfrak{r}''_i$ is a cuspidal ray in $G \backslash \mathfrak{g}$. Indeed, any element $g \in G$ sending $\partial_{\infty}(\tilde{\mathfrak{r}}'_i)$ to $\partial_{\infty}(\tilde{\mathfrak{r}}'_j)$ gives an injective simplicial correspondence between the vertices in either ray, whose image contains a pre-image in $\mathfrak{g}$ of $\mathfrak{r}_i''$. This correspondence is independent on the choice of $g$, since a different choice $g'$ defines an element $g'g^{-1}$ fixing every vertex in $\tilde{\mathfrak{r}}'_i$. This proves the claim. Finally, let us define $Y''$ as the union of $\pi(Y)$ with all $\pi(\tilde{\mathfrak{r}}_j) \smallsetminus \mathfrak{r}''_i$, for all pairs $(i,j)$ whose corresponding rays $\tilde{\mathfrak{r}}_i'$ and $\tilde{\mathfrak{r}}_j'$ are $G$-equivalent. Thus, $G \backslash \mathfrak{g}$ is obtained by attaching all $\mathfrak{r}''_i$ to the finite graph $Y''$.
\end{proof}

\begin{proof}[Proof of Proposition \ref{prop comb finito up}]
By hypothesis there exists a family of rays $\mathfrak{R}_{H}=\lbrace \mathfrak{r}_j \rbrace_{j=1}^{\gamma}$ satisfying (a), (b), (c), (d) and (e) in Definition \ref{def good quotient}. For all index $j$, we denote by $\xi_j $ the visual limit of $\mathfrak{r}_j$, and by $\lbrace v_n(\xi_j) \rbrace_{n=1}^{\infty}$ the vertex set of $\mathfrak{r}_j$. So, we have $\mathbb{P}^1(k)= H \cdot \lbrace \xi_j \rbrace_{j=1}^{\gamma} $.

Let $\lbrace \omega_i \rbrace_{i=1}^{\delta}$ be a set of representatives of $H' \backslash \mathbb{P}^1(k)$. Then, each $\omega_i$ can be written as $\omega_i=h \cdot \xi_j$ for some suitable index $j=j(i)$ and some suitable element $h=h(i) \in H$. Thus, we define $\widehat{\mathfrak{r}}'_i$ as the intersection of $h \cdot \mathfrak{r}_j$ with the unique ray in $\mathfrak{t}$ joining $B_0^{|0|}$ with $\omega_i$. Let us write $\mathrm{V}(\widehat{\mathfrak{r}}'_i)=\lbrace v_n(\omega_i) \rbrace_{n=1}^{\infty}$, where $v_{n}(\omega_i)$ and $v_{n+1}(\omega_i)$ are neighbors. By definition, for each vertex $v_n(\omega_i)$, there exists $m=m(n)\in \mathbb{Z}_{>0}$ such that $v_n(\omega_i)=h \cdot v_m(\xi_j)$. Thus, we have $\mathrm{Stab}_H(v_{n}(\omega_i))=h \mathrm{Stab}_H(v_{m}(\xi_j)) h^{-1}$, where $H \subseteq H'$. Hence, condition (e) follows.

Let $G$ be the finite group $H'/H$. Note that $H' \backslash \mathbb{P}^1(k)= G \backslash (H \backslash \mathbb{P}^1(k))$. Moreover, note that the quotient graph $H' \backslash \mathfrak{t}$ is the quotient of the combinatorially finite graph $H \backslash \mathfrak{t}$ by the finite group $G$. Then, it follows from Lemma \ref{lemma comb finite} that $H' \backslash \mathfrak{t}$ is combinatorially finite, and that for each ray $\overline{\mathfrak{r}_j}$ in $H \backslash \mathfrak{t}$ there exists a subray $\tilde{\mathfrak{r}_j}^{\circ}$ not containing two vertices in the same $G$-orbit.

Let $\mathfrak{r}_j^{\circ} \subseteq \mathfrak{r}_j\subset \mathfrak{t}$ be a lift of $\tilde{\mathfrak{r}_j}^{\circ}$. So, for each index $i$, we define $\mathfrak{r}_i'$ as the intersection of $\widehat{\mathfrak{r}}_i'$ with $h(i) \cdot \mathfrak{r}_j^{\circ}$. We write $\mathrm{V}(\mathfrak{r}_i')=\lbrace v_n(\omega_i) \rbrace_{n=N_i}^{\infty}$, where $N_i>0$. Then, for each $n \geq N_i+1$, the vertices $v_{n-1}(\omega_i)$ and $v_{n+1}(\omega_i)$ are not in the same $H'$-orbit. So, since, by condition (e), all other neighbors are in the same $\mathrm{Stab}_{H'}(v_n(\omega_i))$-orbit as $v_{n-1}(\omega_i)$, we see that $\mathrm{Stab}_{H'}(v_n(\omega_i))$ stabilizes $v_{n+1}(\omega_i)$, i.e. condition (d) holds.

In order to check condition (c) on $\mathfrak{R}_{H'}:= \lbrace \mathfrak{r}_i' \rbrace_{i=1}^{\delta}$, we just have to prove the projections $\overline{\mathfrak{r}_i'}$ and $\overline{\mathfrak{r}_l'}$ to $H' \backslash \mathfrak{t}$ do not intersect when $i \neq l$ in $\lbrace 1, \cdots, \delta \rbrace$. Indeed, it follows from Lemma \ref{lemma comb finite}, and the construction of the rays $\mathfrak{r}_i'$, that $\overline{\mathfrak{r}_i'} \cap \overline{\mathfrak{r}_l'} \neq \emptyset $ if and only if $\overline{\mathfrak{r}_i'}=\overline{\mathfrak{r}_l'}$, and also if and only if their visual limits coincide. By definition, the last assertion does not hold if $i \neq l$. Finally, condition (b) is an immediate consequence of Lemma \ref{lemma comb finite}, and condition (a) is immediate by construction.
\end{proof}

\begin{prop}\label{prop comb finito down}
Let $H$ be a discrete subgroup of $\mathrm{GL}_2(k)$. Assume that $H$ closes enough umbrellas. Then, any finite index subgroup $H_0$ of $H$ closes enough umbrellas.
\end{prop}

\begin{proof} Since any finite index subgroup of $H$ contains a normal subgroup, by Proposition \ref{prop comb finito up} we may assume that $H_0$ is normal in $H$. By hypothesis there exists a family of rays $\mathfrak{R}_{H}=\lbrace \mathfrak{r}_j \rbrace_{j=1}^{\gamma}$ satisfying (a), (b), (c), (d) and (e) in Definition \ref{def good quotient}. For all index $j$, we denote by $\xi_j $ the visual limit of $\mathfrak{r}_j$, and by $\lbrace v_n(\xi_j) \rbrace_{n=1}^{\infty}$ the vertex set of $\mathfrak{r}_j$. So, we have $\mathbb{P}^1(k)= H \cdot \lbrace \xi_j \rbrace_{j=1}^{\gamma} $.

Let $\lbrace \mu_i \rbrace_{i=1}^{\beta}$ be a set of representatives of $H_0 \backslash \mathbb{P}^1(k)$. Then, each $\mu_i$ can be written as $\mu_i=h \cdot \xi_j$ for some suitable index $j=j(i)$ and some suitable element $h=h(i) \in H$. Thus, we define $\widehat{\mathfrak{r}}_i= h \cdot \mathfrak{r}_j$, i.e. $\mathrm{V}(\widehat{\mathfrak{r}}_i)=\lbrace v_n(\mu_i) \rbrace_{n=1}^{\infty}$, where $v_n(\mu_i)=h \cdot v_n(\xi_j)$. So, we have 
$$\mathrm{Stab}_{H_0}(v_{n}(\mu_i))= H_0 \cap h \mathrm{Stab}_H(v_{n}(\xi_j)) h^{-1}.$$
In particular, condition (d) for $H_0$ follows immediately.

Now, we check condition (c) for $H_0$. Indeed, assume that there exist $v_0 \in  \mathrm{V}(\widehat{\mathfrak{r}}_k)$, $w_0 \in \mathrm{V}(\widehat{\mathfrak{r}}_l)$ and $h_0 \in H_0$ such that $h_0 \cdot v_0 = w_0$. Write $v_0= h(k) \cdot v$ and $w_0= h(l) \cdot w $, with $v \in \mathrm{V}(\mathfrak{r}_{j(k)})$ and $w \in  \mathrm{V}(\mathfrak{r}_{j(l)})$. Then $h \cdot v=w$ with $h= h(l)^{-1} h_0 h(k) \in H$, which contradicts condition (c) for $H$. So, condition (c) for $H_0$ follows.

Let $G$ be the finite group $H/H_0$. Let $\pi: \mathrm{Stab}_H(v_n(\mu_i)) \to G$ be the map defined by composing the natural inclusion $\mathrm{Stab}_H(v_n(\mu_i)) \to H$ with the projection $H \to G$. Since, for each $n \in \mathbb{Z}_{\geq 1}$ we have $\mathrm{ker}(\pi)=\mathrm{Stab}_{H_0}(v_n(\mu_i))$, we obtain from condition (d) for $H$ the chain of contentions
$$ \mathrm{Stab}_H(v_{1}(\mu_i))/\mathrm{Stab}_{H_0}(v_{1}(\mu_i)) \subseteq \cdots \subseteq \mathrm{Stab}_H(v_{n}(\mu_i))/\mathrm{Stab}_{H_0}(v_{n}(\mu_i)) \subseteq \cdots$$
Then, since $G$ is a finite set, there exists $t_0=t_0(i) \in \mathbb{Z}_{\geq 1}$ such that, for each $n \geq t_0$
\begin{equation}\label{eq stab in closing umbrellas}
 \mathrm{Stab}_H(v_{n}(\mu_i))/\mathrm{Stab}_{H_0}(v_{n}(\mu_i)) = \mathrm{Stab}_H(v_{n+1}(\mu_i))/\mathrm{Stab}_{H_0}(v_{n+1}(\mu_i)). 
\end{equation}
Recall that, since $K$ is locally compact, we have that $\mathrm{Stab}_{\mathrm{GL}_2(K)}(v_{n}(\mu_i))$ is compact. Then, for each discrete subgroup $D$, for instance $H$ or $H_0$, we get that $\mathrm{Stab}_{D}(v_{n}(\mu_i))$ is finite. Then, Equation \eqref{eq stab in closing umbrellas} implies that, for each $n > t_0$ 
$$
 |\mathrm{Stab}_{H_0}(v_{n}(\mu_i))/\mathrm{Stab}_{H_0}(v_{n-1}(\mu_i))|=  |\mathrm{Stab}_H(v_{n}(\mu_i))/\mathrm{Stab}_{H}(v_{n-1}(\mu_i)) | .
$$
In particular, the injective map 
$$\psi: \mathrm{Stab}_{H_0}(v_{n}(\mu_i))/\mathrm{Stab}_{H_0}(v_{n-1}(\mu_i)) \to \mathrm{Stab}_{H}(v_{n}(\mu_i))/\mathrm{Stab}_{H}(v_{n-1}(\mu_i)),$$
induced by the inclusion $\iota: \mathrm{Stab}_{H_0}(v_{n}(\mu_i)) \to \mathrm{Stab}_{H}(v_{n}(\mu_i))$, is a bijection. It follows from condition (e) for $H$ and the orbit-stabilizer relation that the set $\mathrm{Stab}_{H}(v_{n}(\mu_i))/\mathrm{Stab}_{H}(v_{n-1}(\mu_i))$ parametrizes all the neighboring vertices in $\mathfrak{t}$ of $v_n(\mu_i)$ other than $v_{n+1}(\mu_i)$. So, since $\psi$ is a bijection, we deduce that the set $ \mathrm{Stab}_{H_0}(v_{n}(\mu_i))/\mathrm{Stab}_{H_0}(v_{n-1}(\mu_i))$ also parametrizes the aforementioned set of vertices. In other words, up to replacing $\widehat{\mathfrak{r}}_i$ by the ray $\mathfrak{r}_i'$ defined by the vertex set $\lbrace v_n(\mu_i)\rbrace_{i=t_0+1}^{\infty}$, condition (e) follows for $H_0$.

Now, note that the graph $H \backslash \mathfrak{t}$ is the quotient of the graph $H_0 \backslash \mathfrak{t}$ by the finite group $G$. In particular, the pre-image of the finite graph $Y_{H}$ by the projection $H_0 \backslash \mathfrak{t} \to H \backslash \mathfrak{t} $ is a finite graph. So, since $\widehat{\mathfrak{r}}_i \smallsetminus \mathfrak{r}_i'$ is also a finite graph, we conclude that condition (b) holds for $H_0$. Condition (a) for $H_0$ follows from definition. Thus, we conclude the proof.
\end{proof}

\begin{proof}[Proof of Theorem \ref{Teo comb finito}]
Let $H_0$ be a common finite index subgroup containing in $H$ and $H'$. By replacing $H_0$ by a smaller subgroup if needed, we can assume that $H_0$ is normal in $H'$. Then, it follows from Proposition \ref{prop comb finito down} that $H_0$ closes enough umbrellas. By applying Proposition \ref{prop comb finito up} to $H_0$ and $H'$, we conclude that $H'$ also closes enough umbrellas.
\end{proof}

As in \S \ref{Section on the principal problem} and \S \ref{subsection BTT},
let $k$ be the function field of a smooth, projective, geometrically integral curve $\mathcal{C}$ defined over a field $\mathbb{F}$.
Let $Q$ be a closed point in $\mathcal{C}$, and set $U'=\mathcal{C} \smallsetminus \lbrace Q\rbrace$. Denote by $R'$ the ring of regular functions on $U'$. Let $\nu_Q$ be the discrete valuation map defined from the closed point $Q$. Let us denote by $k_Q$ the completion of $k$ with respect to $\nu_Q$. In the remaining of this section, we give a detailed description of certain quotient of the Bruhat-Tits tree $\mathfrak{t}=\mathfrak{t}(k_Q)$ defined from $\mathrm{SL}_2$ and $k_Q$. In order to do this, let us introduce the following definition:

\begin{defi}\label{definition of max and eichler orders}
A $\mathcal{C}$-order of maximal rank $\mathfrak{R}$ is a locally free sheaf of $\mathcal{O}_{\mathcal{C}}$-algebras whose generic fiber is $\mathbb{M}_2(k)$. We say that a $\mathcal{C}$-order $\mathfrak{D}$ is maximal when it is maximal with respect to inclusion. An Eichler $\mathcal{C}$-order $\mathfrak{E}$ is the sheaf-theoretical intersection of two maximal $\mathcal{C}$-orders.
\end{defi}

\begin{ex}
Let us denote by $\mathfrak{D}_0$ the sheaf $\mathfrak{D}_0=\mathbb{M}_2(\mathcal{O}_{\mathcal{C}})$. Then $\mathfrak{D}_0$ is a maximal $\mathcal{C}$-order. Moreover $\mathfrak{D}_0(U')^{*}=\mathrm{GL}_2(R')$. 
\end{ex}

\begin{corollary}\label{Cor comb fin}
Let $H\subset \mathrm{GL}_2(k)$ be a group commensurable with $\mathrm{GL}_2(R')$. Then $H$ closes enough umbrellas. In particular, for any Eichler $\mathcal{C}$-order $\mathfrak{E}$, we have that $\tilde{H}=\mathfrak{E}(U)^{*}$ and $\tilde{\Gamma}= \mathrm{Stab}_{\mathrm{GL}_2(k)} (\mathfrak{E}(U))$ close enough umbrellas.
\end{corollary}

\begin{proof}
First, it follows from \cite[Chapter II, \S 2.1- \S 2.3]{S} that $\mathrm{GL}_2(R')$ closes enough umbrellas. Then, it follows from Theorem \ref{Teo comb finito} that any group $H \subset \mathrm{GL}_2(k)$ commensurable with $\mathrm{GL}_2(R')$ closes enough umbrellas.

Now, we claim that $\tilde{H}$ and $\tilde{\Gamma}$ are commensurable with $\mathrm{GL}_2(R')$. Indeed, let $\mathfrak{D}$ be a maximal $\mathcal{C}$-order containing $\mathfrak{E}$. Let us fix $\tilde{\Gamma}_0= \text{Stab}_{\text{GL}_2(k)}(\mathfrak{D}(U'))$.
Note that $\tilde{\Gamma}_0$ and $\tilde{\Gamma}$ are commensurable, since they contain the respective finite index subgroups $\tilde{H}_0= \mathfrak{D}(U')^{*}$ and $\tilde{H}$, where $\tilde{H}$ is a finite index subgroup of $\tilde{H}_0$ (cf.~\cite[Theorem 1.2]{A2}). Moreover, note that $\tilde{H}_0$ belongs to the same commensurability class as $\mathrm{GL}_2(R')$, since $\mathfrak{D} \cap \mathfrak{D}_0$ is a finite index Eichler $\mathcal{C}$-order simultaneously contained in $\mathfrak{D}$ and $\mathfrak{D}_0$.
\end{proof}

\begin{rem}
Theorem \ref{Teo comb finito} and Corollary \ref{Cor comb fin} can be easily extended to subgroups of $\mathrm{PGL}_2(k)$. 
\end{rem}

\section{Spinor class fields}\label{Section Spinor}

In this section we introduce the basic definitions and results about completions, spinor genera and spinor class fields of orders. See \cite{abelianos} for details.

We denote by $|\mathcal{C}|$ the set of closed points in the smooth projective geometrically integral curve $\mathcal{C}$, and we fix $P_{\infty} \in |\mathcal{C}|$. Let $U_0$ be the affine open set $ \mathcal{C} \smallsetminus \lbrace P_{\infty} \rbrace$. For every point $P \in |\mathcal{C}|$, we denote by $k_P$ the completion at $P$ of the function field $k=k(\mathcal{C})$, and by $\mathcal{O}_{P}$ the ring of integers of the former. For any open set $U \subseteq \mathcal{C}$, we define the ad\`ele ring $\ad_{U}$ of $U$ as the subring of $\prod_{P \in |U|} k_P$ consisting of the tuples $a=(a_{P})_{P}$ where $a_P$ lies in the ring $\mathcal{O}_{P}$ for all but finitely many closed points $P$. We also define the id\`ele group $\mathbb{I}_U$ as the group of invertible ad\`eles $\ad_U^{*}$. We write $\ad=\ad_{\mathcal{C}}$ and $\mathbb{I}=\mathbb{I}_{\mathcal{C}}$.

A $\mathcal{C}$-lattice or $\mathcal{C}$-bundle in a finite dimensional $k$-vector space $V$ is a locally free subsheaf of the constant sheaf $V$. For any sheaf of groups $\Lambda$ on $\mathcal{C}$ we denote by $\Lambda(U)$ its group of $U$-sections. In particular, this convention applies to $\mathcal{C}$-lattices. By definition, the completion at $P$ of $\Lambda$, denoted $\Lambda_P$, is the topological closure of $\Lambda(U)$ in $V_P:=V \otimes_{k} k_{P}$, where $U$ is an affine open subset containing $P$. Thus defined, $\Lambda_P$ does not depend on the choice of $U$. Note that, for every affine open subset $U\subseteq \mathcal{C}$,
the $\oink_\mathcal{C}(U)$-module $\Lambda(U)$ is an $\oink_\mathcal{C}(U)$-lattice. The same property holds for orders. As in the affine context, every $\mathcal{C}$-lattice is determined by its set of local completions $\lbrace \Lambda_P: P\in |\mathcal{C}| \rbrace$, as follows:
\begin{enumerate}
\item[(a)] For any two lattices $\Lambda$ and $\Lambda'$ in $V$, we have $\Lambda_P=\Lambda'_P$ for almost all $P$,
\item[(b)] if $\Lambda_P=\Lambda'_P$ for all $P$, then $\Lambda=\Lambda'$, and 
\item[(c)] every family $\{\Lambda''(P)\}_P$ of local lattices satisfying  $\Lambda''(P)=\Lambda_P$ for almost all $P$ is the family of completions of a global lattice $\Lambda''$ in $V$.
\end{enumerate}
In particular, the preceding properties apply for order. 
Given a finite dimensional vector space $W$ over $k$, we define its adelization $W_{\ad}$ as $W_{\ad}=W \otimes_{k} \ad$. This is a $k$-vector space isomorphic to $\ad^{\mathrm{dim}_k W}$. In particular, this definition applies to $W=\text{End}_k(V)$. We also define the adelization of a lattice $\Lambda$ by $\Lambda_{\mathbb{A}} = \prod_{P \in |\mathcal{C}|} \Lambda_P$, which is an open and compact subgroup of $V_{\mathbb{A}}$. Given an arbitrary $\mathcal{C}$-lattice $\Lambda$ and an adelic matrix  
$$a\in \mathrm{End}_\ad(V_\ad)= \big(\text{End}_k(V)\big)_{\mathbb{A}},$$
we define the adelic image $L=a\Lambda$ of $\Lambda$ as the unique $\mathcal{C}$-lattice satisfying $L_\ad=a\Lambda_\ad$. To each $\mathcal{C}$-lattice $\Lambda$ in $k^2$, we associate the $\mathcal{C}$-order $\Da_{\Lambda}=\mathrm{End}_{\oink_X}(\Lambda)$ in the matrix algebra $\matrici_2(k)$, which is defined on every open set $U\subseteq \mathcal{C}$ by
$$\Da_\Lambda(U)=\left\{a\in \matrici_2(k)\Big|a\Lambda(U)\subseteq\Lambda(U)\right\}.$$
This is a maximal $\mathcal{C}$-order in $\mathbb{M}_2(k)$ (cf.~Definition \ref{definition of max and eichler orders}). Moreover, every maximal $\mathcal{C}$-order in the two-by-two matrix algebra equals $\Da_\Lambda$, for some $\mathcal{C}$-lattice $\Lambda$ in $k^2$. In particular, if we fix a maximal $\mathcal{C}$-order $\Da$, then any other maximal $\mathcal{C}$-order in $\mathbb{M}_2(k)$ is equal to $\Da'=a\Da a^{-1}$, for some $a\in\mathrm{GL}_2(\ad)$. In general, if we fix an $\mathcal{C}$-order $\Da$ of maximal rank, then we can define the genus $\mathrm{gen}(\Da)$ of $\Da$ as the set of all $\mathcal{C}$-orders $a\Da a ^{-1}$, for $a\in \matrici_2(\ad)^*$. So, the previous statement is equivalent to the fact that the set of maximal $\mathcal{C}$-orders is a genus, which we denote by $\mathbb{O}_0$.

Let $\mathfrak{D}$ be a $\mathcal{C}$-order of maximal rank, i.e. of rank $4$. Let $U$ be either an affine open set of $\mathcal{C}$ or the full set $\mathcal{C}$. We define the $U$-spinor class field of $\mathfrak{D}$ as the field corresponding, via class field theory, to the subgroup $k^*H(\Da, U)\subseteq \mathbb{I}=\ad^*$, where
\begin{equation}\label{Eq class fields}
H(\Da, U)=\left \lbrace \mathrm{det}(a)|a\in\matrici_2(\ad)^*,\  a\Da(V) a^{-1}=\Da(V), \, \forall \, \, V \stackrel{\circ}{\subseteq} U \right \rbrace.
\end{equation}
The symbol $\stackrel{\circ}{\subseteq}$ above denotes an open subset. This field depends only on the genus $\mathbb{O}=\mathrm{gen}(\Da)$ of $\Da$, and we denote it by $\Sigma(\mathbb{O},U)$. When $U=\mathcal{C}$ we simplify the notation by using $\Sigma=\Sigma(\mathbb{O})$. Let $\mathbb{I} \to  \mathrm{Gal}(\Sigma/k)$, $t\mapsto [t,\Sigma/k]$ be the Artin map on the id\`ele group (cf.~\cite[Chapter VI, \S 5, p. 387]{Neukirch}). There exists a well-defined distance map $\rho:\mathbb{O}\times\mathbb{O}\rightarrow\mathrm{Gal}\big(\Sigma/k\big)$, given by $\rho(\Da,\Da')=[\det(a),\Sigma/k]$, where $a\in\mathrm{GL}_2(\ad)$ is any adelic element satisfying $\Da'=a\Da a^{-1}$. The distance map has a multiplicative property, in the sense that, for any tuple $(\Da,\Da',\Da'')\in\mathbb{O}^3$, it satisfies $\rho(\Da,\Da'')=\rho(\Da,\Da')\rho(\Da',\Da'')$. The kernel of $\rho$ consists of the pairs $(\Da,\Da')$ such that $\Da(U)$ and $\Da'(U)$ are $\mathrm{GL}_2(\mathcal{O}_{\mathcal{C}}(U))$-conjugate for every affine open subset $U\subseteq \mathcal{C}$. In the case of maximal orders, the map defined above is $\rho_0:\mathbb{O}_0^2\rightarrow\mathrm{Gal}\big(\Sigma(\mathbb{O}_0)/k\big)$, and it can be characterized as follows: The image of $\rho_0$ for a pair $(\Da,\Da')$ of maximal orders is given by the formula $\rho_0(\Da,\Da')=[[D(\Da,\Da'),\Sigma(\mathbb{O}_0)/k]]$, where $D\mapsto [[D,\Sigma(\mathbb{O}_0)/k]]$ is the Artin map on divisors and the divisor $D(\Da,\Da')$ is defined in Equation \eqref{eq divisor def from two max or}.

\section{Eichler orders and grids}\label{Section grids}

In this section notation is as in \S \ref{Section Spinor}. 

\subsection{The grid defined from and Eichler order}

In the local algebra $\matrici_2(k_P)$, any two $\mathcal{O}_{P}$-maximal orders are simultaneously $\mathrm{GL}_2(k_P)$-conjugate to the orders $\Da_P=\sbmattrix {\oink_P}{\oink_P}{\oink_P}{\oink_P}$ and $\Da'_P=\sbmattrix {\oink_P}{\pi_P^d\oink_P}{\pi_P^{-d}\oink_P}{\oink_P}$ for some $d \in \mathbb{Z}_{\geq 0}$, where $\pi_P$ is a local uniformizing parameter in $k_P$. So, we define the local distance $d_P$ between maximal orders in $\matrici_2(k_P)$ by taking $d_\mathfrak{p}(\Da_P,\Da'_P)=d$, where $d$ is as above. As we introduce in Definition \ref{definition of max and eichler orders}, an Eichler $\mathcal{C}$-order, or simply an Eichler order, is the intersection of two maximal $\mathcal{C}$-orders. This is a local definition in the sense that $\Da_P\cap\Da'_P=(\Da\cap\Da')_P$ for every pair of orders. Moreover, locally, for any Eichler order $\mathfrak{E}_{P}$ there exists a unique pair of maximal orders whose intersection is $\mathfrak{E}_{P}$. So, we define the level of a local Eichler order as the distance between the preceding maximal orders. Globally, there exists a well-defined distance map on the set of maximal $\mathcal{C}$-orders, whose image on a pair $(\Da, \Da')$ is the effective divisor
\begin{equation}\label{eq divisor def from two max or}
 D=D(\Da,\Da')=\sum_{P\in|\mathcal{C}|}d_P(\Da_P,\Da'_P)P.   
\end{equation}
In particular, there exists a global level defined, on an Eichler $\mathcal{C}$-order $\Ea = \Da \cap \Da'$, as the distance $D(\Da,\Da')$. A useful property of the level function is that two local Eichler orders are $\mathrm{GL}_2(k_P)$-conjugate precisely when their local levels equal. This property can be interpreted in terms of genera by saying that two Eichler $\mathcal{C}$-orders belong to the same genus exactly when they have the same global level. So, for any effective divisor $D$, there exists a genus of Eichler $\mathcal{C}$-orders of level $D$, which is denoted by $\mathbb{O}_D$.

It follows, from the characterization of the Bruhat-Tits tree in terms of maximal orders (cf.~\S \ref{Section BTT}), that there exists a bijective map between the set of local Eichler orders $ \Ea$ of level $\kappa$ and the set of finite lines $\mathfrak{p}$ of length $\kappa$ in the Bruhat-Tits tree. Formally, a local Eichler order $\mathfrak{E}$ corresponds to the finite line $\mathfrak{p}=\mathfrak{s}(\mathfrak{E})$ whose vertices are the maximal orders containing $\mathfrak{E}$. Let $\Ea$ be an Eichler $\mathcal{C}$-order of level $D=\sum_P n_P P$. Let us denote by $S(\mathfrak{E})$ the product of finite lines $S(\mathfrak{E})=\prod_P\mathfrak{s}(\mathfrak{E}_P)$, where $P$ belongs to the set of closed points such that $n_P>0$. This is called the grid of $\mathfrak{E}$. It follows from Property (c) in \S \ref{Section Spinor} that the set of maximal $\mathcal{C}$-orders containing $\Ea$ corresponds to the vertex set in $S(\Ea)$. Moreover, it is easy to see that this correspondence is compatible with the action of $\mathrm{PGL}_2(k)$ on Eichler $\mathcal{C}$-orders by conjugation. 
In order to compare different Eichler $\mathcal{C}$-orders, we fix an effective divisor
$D=\sum_P n_P P$, and a finite set of places $T \supseteq \mathrm{Supp}(D).$ Denote by $\mathrm{Eich}(D, T)$ the set of Eichler $\mathcal{C}$-orders of level $D$ satisfying $\mathfrak{E}_Q = \mathbb{M}_2(\mathcal{O}_Q)$ for $Q \notin T$. 
Then, given an Eichler $\mathcal{C}$-order in $\mathrm{Eich}(D, T)$, its grid can be seen naturally as a subcomplex of the finite product of Bruhat-Tits trees $\prod_{P \in T} \mathfrak{t}(k_P)$. 
Any grid of the form $S(\mathfrak{E})$, for $\mathfrak{E} \in \mathrm{Eich}(D, T)$, is called a concrete $D$-grid. 
Note that the group $G_T=\mathrm{GL}_2(\mathcal{O}_{\mathcal{C}}(\mathcal{C} \smallsetminus T))$ acts on the set of concrete $D$-grids by conjugation. Indeed, we can define this action as the extension of the conjugacy action of $G_T$ on the set of maximal $\mathcal{C}$-orders to $D$-grids, which is valid since $G_T$ acts simplicially on each local tree. The orbits of concrete $D$-grids by this action are called abstract $D$-grids. Any representative of an abstract grid is called a concrete representative. Note that all these definitions depend on the set $T$. This is why it is important to consider the following result. 

\begin{prop}\cite[Proposition 3.1]{A5}\label{eichler=grillas}
Let $D$ be an effective divisor. Then, there exists a finite set of places $ T $ containing $\mathrm{Supp}(D)$ such that every $\mathrm{PGL}_2(k)$-conjugacy class of Eichler $\mathcal{C}$-orders contains a representative in $\mathrm{Eich}(D, T)$.
\end{prop}

Let $Q$ be a closed point of $\mathcal{C}$, and write $D=D'+n_Q Q$, where $n_Q>0$ and the support of $D'$ fails to contain $Q$. Then, each concrete $D$-grid is a paralellotope having two concrete $D'$-grids as opposite faces. These opposite faces are called the $Q$-faces of the $D$-grid. We say that two concrete $D'$-grids $S$ and $S'$ are $Q$-neighbors if there exists a concrete $D$-grid $\tilde{S}$, with $D=D'+Q$ and $Q \notin \mathrm{Supp}(D')$, such that $S$ and $S'$ are the $Q$-faces of $\tilde{S}$. Let $\mathfrak{D}$ be a maximal $\mathcal{C}$-order corresponding to a vertex $v$ in $S$. Then, there exists one and only one $Q$-neighbor $v'$ among the vertices of $S'$. We call it the $Q$-neighbor of $v$ in $\tilde{S}$.


\subsection{Classifiying graphs}
As an intermediate step to prove Theorem \ref{teo cusp}, we characterize a quotient graph of $\mathfrak{t}$ other than $ \mathfrak{t}_D=\mathrm{H}_D \backslash \mathfrak{t}$. In order to introduce this quotient structure, fix $D$ an effective divisor, and let $\mathbb{O}_D$ be the genus containing all Eichler $\mathcal{C}$-orders of level $D$. Let $Q \in |\mathcal{C}|$ be a closed point not contained in $\mathrm{Supp}(D)$. Let $V_0$ be the affine open set $\mathcal{C} \smallsetminus \lbrace Q \rbrace$. Then, any order in $\mathbb{O}_D$ is maximal at $Q$, i.e. its completion at $Q$ is maximal. For any $\mathfrak{E} \in \mathbb{O}_D$, we define the C-graph $C_Q(\mathfrak{E})=\Gamma\backslash\mathfrak{t}$, where $\mathfrak{t}=\mathfrak{t}(k_Q)$, and $\Gamma$ is the stabilizer of $\mathfrak{E}(V_0)$ in $\mathrm{PGL}_2(k)$. Note that, it follows from Corollary \ref{Cor comb fin} that $C_Q(\Da)$ is combinatorially finite. Two Eichler $\mathcal{C}$-orders $\mathfrak{E}$ and $\mathfrak{E}'$ such that $\rho(\mathfrak{E} ,\mathfrak{E}')$ belongs to the group generated by $[[Q, \Sigma(\mathbb{O}_D)/k]]$ define isomorphic quotient graphs. Indeed, it follows from \cite[\S 2]{abelianos} that, if $\rho(\mathfrak{E} ,\mathfrak{E}') \in \left \langle [[Q, \Sigma(\mathbb{O}_D)/k]] \right \rangle$, then $\mathfrak{E}(U_0)$ and $\mathfrak{E}'(U_0)$ are $\mathrm{GL}_2(k)$-conjugate. In this case we write $\mathfrak{E}\sim \mathfrak{E}'$. We denote by $\mathfrak{Sp}(\mathbb{O}_D, Q)$ the quotient set of $\mathbb{O}_D$ by the previous equivalence relation. The classifying graph $C_{Q}(\mathbb{O}_D)$ is the disjoint union of the finitely many C-graphs corresponding to all elements in $\mathfrak{Sp}(\mathbb{O}_D, Q)$. In particular, it is combinatorially finite.


All definitions and conventions introduced in \S \ref{Section BTT} apply to $C_{Q}(\mathbb{O}_D)$ by adapting them to the context of disjoint union of graphs. In particular, by the cusp set of $C_{Q}(\mathbb{O}_D)$ we mean the disjoint union of the cusp sets of all connected components of $C_{Q}(\mathbb{O}_D)$. In the following section we study the combinatorial structure of the classifying graphs of Eichler orders. With this in mind, we make frequent use of the next result:

\begin{prop}\cite[Proposition 3.2]{A5}\label{func corresp D-grillas}
Let $D$ be an effective divisor supported away from the place $Q$. The vertices of the classifying graph $C_Q(\mathbb{O}_D)$ are in bijection with the abstract $D$-grids, while its pairs of mutually reverse edges are in bijection with the abstract $(D+Q)$-grids. The endpoints of an edge are the vertices of $C_Q(\mathbb{O}_D)$ corresponding to the $Q$-faces of the grid corresponding to that edge.  
\end{prop}

\subsection{Spinor class fields of Eichler orders} We finish this section by recalling some results about the spinor class field associated to the genus of Eichler $\mathcal{C}$-orders of level $D$. Let $D$ be a divisor, which we write as $D=\sum_{i=1}^r n_i P_{i}$, where $P_i \neq P_{\infty}$, and let $U_0$ be the affine set $ \mathcal{C} \smallsetminus \lbrace P_{\infty} \rbrace$ as above. It follows from \cite[Theorem 1.2]{A13} that the spinor class field $\Sigma_D=\Sigma(\mathbb{O}_D)$ (resp. $\Sigma(\mathbb{O}_D,U_0)$), for Eichler $\mathcal{C}$-orders of level $D$, is the maximal subfield of $\Sigma_0=\Sigma(\mathbb{O}_0)$ (resp. $\Sigma(\mathbb{O}_0,U_0)$) splitting at every place $P_i$ for which $n_i$ is odd.

\begin{prop}\label{gal grupo de clase}
Let $J= \lbrace i: n_i \text{ is odd}\rbrace$. The Galois group $\mathrm{Gal}(\Sigma_D/k)$ is isomorphic to the abelian group $\mathrm{Pic}(\mathcal{C})/(2 \mathrm{Pic}(\mathcal{C})+ \langle \overline{P_{j}}: j \in J \rangle)$. Using the same notation, $\mathrm{Gal}(\Sigma(\mathbb{O}_D,U_0)/k)$ is isomorphic to $\mathrm{Pic}(\mathcal{C})/(2 \mathrm{Pic}(\mathcal{C})+ \langle  \overline{P_{\infty}} \rangle+ \langle \overline{P_{j}}: j \in J  \rangle)$.
\end{prop}

\begin{proof}
Let $L/F$ be a finite abelian extension (i.e. Galois with abelian Galois group) of global fields. It follows from \cite[Chapter VI, \S 6, Theorem 6.1 and Corollary 6.6]{Neukirch} that there exists an isomorphism from $\mathrm{Gal}(L/F)$ to $\mathbb{I}_F/F^{*} H(L)$, where $\mathbb{I}_F$ is the id\`ele group of $F$, and $H(L):=\lbrace N_{L/F}(a): a \in \mathbb{I}_L \rbrace$ is the kernel of the Artin map, which satisfies the following properties:
\begin{itemize}
    \item[(i)] $Q$ is unramified in $L/F$ if and only if $\mathcal{O}_Q^{*} \subseteq F^{*} H(L)$.
    \item[(ii)] $Q$ splits completely in $L/F$ if and only if $F_Q^{*} \subseteq F^{*} H(L)$.
\end{itemize}
Apply this when $F$ is the global function field $k= \mathbb{F}(\mathcal{C})$ and $L$ is $\Sigma_D$.
Recall that $H(\Sigma_0)$ equals $H(\mathbb{M}_2(\mathcal{O}_{\mathcal{C}}), \mathcal{C})$ as in Equation \eqref{Eq class fields}. Then, since the localization of $\mathbb{M}_2(\mathcal{O}_{\mathcal{C}})$ at $Q$ is $\mathbb{M}_2(\mathcal{O}_{Q})$, it is easy to see that $k_Q^{*2}\mathcal{O}_Q^{*} \subseteq H(\Sigma_0)$, for all closed points $Q \in \mathcal{C}$. In particular, since $\Sigma_D \subseteq \Sigma_0$, we obtain $ k_Q^{*2}\mathcal{O}_Q^{*} \subseteq H(\Sigma_0) \subseteq H(\Sigma_D)$. So, if we write $\mathbb{I}_{k, \infty}:=\prod_{Q \in \mathcal{C}} \mathcal{O}_Q^{*}$, then $\mathbb{I}_{k}^2 \mathbb{I}_{k, \infty}$ is contained in $k^{*}H(\Sigma_D)$. And, since $\mathbb{I}_k/ k^{*} \mathbb{I}_{k, \infty} \cong \mathrm{Pic}(\mathcal{C})$, the Galois group of $\Sigma_D/k$ is a quotient of $\mathrm{Pic}(\mathcal{C})/2\mathrm{Pic}(\mathcal{C})$.

Let us write $e(Q)$ for the id\`ele whose coordinate at $Q$ is $\pi_Q$ and any other coordinate equals one. Since $\Sigma_D/k$ splits at $P_j$, with $j \in J$, we deduce from (ii) that all $e(P_j)$, with $j \in J$, belong to $k^{*}H(\Sigma_D)$. In particular, we obtain the inclusion
$$k^{*} \mathbb{I}_k^{2} \mathbb{I}_{k, \infty} \langle e(P_j): j \in J \rangle \subseteq k^{*} H(\Sigma_D).$$
Furthermore, the maximality condition on $\Sigma_D$ implies equality. We conclude that $$\mathrm{Gal}(\Sigma_D/k) \cong \frac{\mathbb{I}_k}{k^{*} \mathbb{I}_k^{2} \mathbb{I}_{k, \infty} \langle e(P_j): j \in J \rangle} \cong \frac{\mathrm{Pic}(\mathcal{C})}{2 \mathrm{Pic}(\mathcal{C})+ \langle \overline{P_{j}}: j \in J \rangle}.$$

Moreover, we can analogously prove that $\mathfrak{G} =\mathrm{Gal}(\Sigma(\mathbb{O}_D,U_0)/k)$ is isomorphic to $\mathrm{Pic}(\mathcal{C})/(2 \mathrm{Pic}(\mathcal{C})+ \langle \overline{P_{\infty}}, \overline{P_{j}}: j \in J \rangle)$, by noting that to compute $\mathfrak{G}$ we no longer need a condition at the place $P_{\infty}$, so the corresponding local stabilizer must be replaced by the full local group $\mathrm{GL}_2(k_{P_{\infty}})$.  
\end{proof}

Moreover, it follows from \cite[Proposition 6.1]{A2} that the corresponding distance function $\rho_D$ on the genus of Eichler $\mathcal{C}$-orders of level $D$ is related to $\rho_0$ through restriction, i.e.
\begin{equation}\label{eq rho}
\rho_D(\Ea_{\Lambda,\Lambda'},\Ea_{L,L'})=\rho_0(\Da_\Lambda,\Da_L)\Big|_{\Sigma(D)},
\end{equation}
for any four $\mathcal{C}$-lattices $\Lambda, \Lambda', L$ and $L'$ (cf.~\S \ref{Section Spinor}).

\section{On quotient graphs of Eichler groups}\label{Section Quotient}

The objective of this section is to prove Theorem \ref{teo cusp}. To do so, we extensively use the following remark. As we said in \S \ref{Section BTT}, every subgroup of $\text{GL}_2(k_{P_{\infty}})$ acts on $\mathfrak{t}$ via its image in $\text{PGL}_2(k_{P_{\infty}})$. In particular, the topological space $\mathfrak{t}_D$ equals the quotient of $\mathfrak{t}=\mathfrak{t}(k_{P_{\infty}})$ by the projective image $\text{PH}_D$ of 
$$\small
\text{H}_D= \left\lbrace \left( \begin{array}{cc}
a &   b\\
c & d \end{array} \right) \in \text{GL}_2(R) : c \equiv 0 \, (\text{mod } I_D)\right\rbrace ,
\normalsize $$
where $I_D$ is the $R$-ideal defined as $I_D=\mathfrak{L}^{-D}(U_0)=\mathfrak{L}^{-D}(\mathcal{C} \smallsetminus \lbrace P_{\infty }\rbrace)$. \\

We start this section by presenting a proof of Theorem \ref{teo cusp} assuming the following result, which is implied by Proposition \ref{lema4} below.

\begin{prop}\label{prop aux}
The number of cusps of any connected component of $C_{P_{\infty}}(\mathbb{O}_D)$ is the same, and it equals
\begin{equation}
c(D)= \alpha(D) [2\mathrm{Pic}(\mathcal{C})+\left\langle \overline{P_{a_1}}, \cdots ,\overline{P_{a_u}}, \overline{P_{\infty}}  \right\rangle : \langle \overline{P_{\infty}} \rangle],
\end{equation}
where 
$$ \alpha(D)= 1 + \frac{1}{q-1} \prod_{i=1}^{r} \left( q^{\deg(P_i) \lfloor \frac{n_i}{2}\rfloor}-1\right),$$
and $P_{a_1}, \cdots, P_{a_u}$ are the closed points in $\mathcal{C}$ whose coefficients in $D=\sum_{i=1}^r n_i P_i$ are odd.
\end{prop}

\subsection{A proof of Theorem \ref{teo cusp}}\label{subsection proof of Theorem} As we just said, we can replace $\text{H}_D$ by its image $\text{PH}_D$ in $\text{PGL}_2(k)$ to compute the cusp number of $\mathfrak{t}_D$.
First, we prove inequality \eqref{numero de patas}. Set $\Gamma= \mathrm{Stab}_{\text{PGL}_2(k)}(\mathfrak{E}_D(U_0))$. On one hand, it follows from Proposition \ref{prop aux} that the cusp number of $\Gamma \backslash \mathfrak{t}$ is equal to $c(D)$. On the other hand, it follows from \cite[Theorem 1.2]{A2} that $$[\Gamma:\text{PH}_D]=\frac{2^r|g(2)|}{[\Sigma(\mathbb{O}_0,U_0):\Sigma(\mathbb{O}_D,U_0)]},$$
where $g(2)$ is the maximal exponent-2 subgroup of $\mathrm{Pic}(R)$.
So, we obtain from Proposition \ref{gal grupo de clase},
\begin{equation}\label{eq gamma/phd}
[\Gamma:\text{PH}_D]=\frac{2^r|g(2)|}{[2 \text{Pic}(\mathcal{C})+ \langle \overline{P_{a_1}}, \cdots, \overline{P_{a_u}}, \overline{P_{\infty}} \rangle: 2 \text{Pic}(\mathcal{C})+ \langle \overline{P_{\infty}} \rangle]}.
\end{equation}
Recall now that, by Corollary \ref{Cor comb fin}, the set of cups of $\mathfrak{t}_D$ (resp. $\Gamma \backslash \mathfrak{t}$) is parametrized by $\mathbb{P}^1(k) / \text{PH}_D$ (resp. $\mathbb{P}^1(k) / \Gamma$). Then, the cusp number of $\mathfrak{t}_D$ cannot exceed $c(\text{H}_D)=c(D)[\Gamma:\text{PH}_D]$ and inequality \eqref{numero de patas} follows. Now, we assume that each $n_i$ is odd and $g(2)$ is trivial.
Then, we have to prove that the cusp number of $\mathfrak{t}_D$ is exactly $c(\text{H}_D)$. This is a consequence of the following lemma.

\begin{lemma}\label{lemma cubrimiento de cuspides}
Assume that each $n_i$ is odd and that $g(2)$ is trivial. Then, there are exactly $[\Gamma: \text{PH}_D]$ cusps in $\mathfrak{t}_D$ with the same image in $\Gamma  \backslash \mathfrak{t}$.
\end{lemma}

\begin{proof}
Let $\Theta=\Theta(\eta)$ be the set of cusps of $\mathfrak{t}_D$ whose image in $\Gamma  \backslash \mathfrak{t}$ is the cusp $\eta$.
Then, $\mathrm{Card}(\Theta)$ is strictly less than $[\Gamma: \text{PH}_D]$ precisely when there exists an element $\overline{g} \in \Gamma/\text{PH}_D$ stabilizing an element of $\Theta$. Since the set of cups of $\mathfrak{t}_D$ (resp. $\Gamma \backslash \mathfrak{t}$) is parametrized by $\mathbb{P}^1(k) / \text{PH}_D$ (resp. $\mathbb{P}^1(k) / \Gamma$), if we prove that, for any $s \in \mathbb{P}^1(k)$, we have
\begin{equation}\label{id stab H}
\mathrm{Stab}_{\Gamma}(s) \subset \text{PH}_D,
\end{equation}
then, the result follows. Let $g \in \Gamma$ and assume that $g$ stabilizes some class of rays corresponding to $s \in \mathbb{P}^1(k)$. By definition, $g \mathfrak{E}_D(U_0) g^{-1}= \mathfrak{E}_D(U_0)$. So, as we saw in \S \ref{Section grids}, $g$ acts on the concrete $D$-grid $S_D$ associated to $\mathfrak{E}_D$ as an automorphism. In particular, we have

\begin{figure}
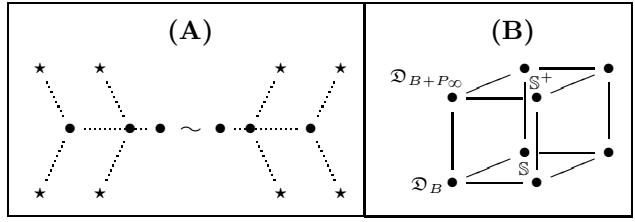

\[ 
\fbox{ \xygraph{
!{<0cm,0cm>;<.8cm,0cm>:<0cm,.8cm>::} 
!{(2.5,2.6) }*+{\textbf{(A)}}="na"
!{(4.5,1) }*+{\bullet}="b11" !{(3.,1) }*+{\bullet}="c11"
!{(0.5,1) }*+{\bullet}="a11" !{(2.,1) }*+{\bullet}="d11"
!{(1.5,1) }*+{\bullet}="z11" !{(3.5,1) }*+{\bullet}="y11"
!{(0,-0.08) }*+{\star}="e1"   
!{(0,2.0) }*+{\star}="e3" 
!{(1,-0.08) }*+{\star}="d1"   
!{(1,2.0) }*+{\star}="d3" 
!{(5,-0.08) }*+{\star}="e2" 
!{(5,2.0) }*+{\star}="e4" 
!{(4,-0.08) }*+{\star}="d2" 
!{(4,2.0) }*+{\star}="d4" 
"a11"-@{.}"e3" "a11"-@{.}"e1" "c11"-@{~}"d11" "c11"-@{.}"b11" "a11"-@{.}"d11" "b11"-@{.}"e2" "b11"-@{.}"e4" "z11"-@{.}"d3" "z11"-@{.}"d1" "y11"-@{.}"d2" "y11"-@{.}"d4"
 } }
\fbox{ \xygraph{
!{<0cm,0cm>;<0.8cm,0cm>:<0cm,0.8cm>::} 
!{(0.8,2.6) }*+{\textbf{(B)}}="na"
 !{(-0.2,0.1) }*+{\bullet}="e2" !{(1.2,0.1) }*+{\bullet}="e3"
 !{(-0.6,-0.0) }*+{{}^{\mathfrak{D}_B}}="e2n" !{(-0.6,1.8) }*+{{}^{\mathfrak{D}_{B+P_{\infty}}}}="e3n" !{(1.3,1.72) }*+{{}^{\mathbb{S}^{+}}}="gn"
!{(-0.2,1.5) }*+{\bullet}="e4" !{(1.2,1.5) }*+{\bullet}="e5"
 !{(1,0.6) }*+{\bullet}="e6" !{(2.4,0.6) }*+{\bullet}="e7"
 !{(1.0,0.25) }*+{{}^{\mathbb{S}}}="gn"
!{(1,2.0) }*+{\bullet}="e8" !{(2.4,2.0) }*+{\bullet}="e9"
"e2"-@{-}"e4" "e2"-@{-}"e3""e4"-@{-}"e5""e3"-@{-}"e5" "e6"-@{-}"e8" "e6"-@{-}"e7""e8"-@{-}"e9""e7"-@{-}"e9" 
 "e2"-@{-}"e6" "e3"-@{-}"e7""e4"-@{-}"e8""e5"-@{-}"e9"   }}
\] 
\caption{Figure \textbf{(A)} shows the Bruhat-Tits tree $\mathfrak{t}(k_{P_i})$, where $\mathfrak{p}_i$ corresponds to the finite central line and where the middle edge of $\mathfrak{p}_i$ is represented by a cursive edge. Figure \textbf{(B)} shows a concrete $(D+P_{\infty})$-grid, or equivalently two $P_{\infty}$-neighboring $D$-grids.}\label{fig 1}
\end{figure}

\begin{itemize}
\item[(1)] $g \in \mathrm{Stab}((\mathfrak{D}_0)_{Q})$ for every $Q \neq P_1, \cdots, P_r, P_{\infty}$, and
\item[(2)] $g (\mathfrak{D}_0)_{P_i} g^{-1}= (\mathfrak{D}_{\epsilon_i D})_{P_i}$ with $\epsilon_i \in \lbrace 0,1 \rbrace$ for any $P_i$ in the support of $D$, i.e., $g$ can either pointwise fix the line in $\mathfrak{t}(k_{P_i})$ joining $(\mathfrak{D}_0)_{P_i}$ with $(\mathfrak{D}_{D})_{P_i}$, or flip it.
\end{itemize}
If some $\epsilon_i =1$, then $g$ acts on $\mathfrak{t}_i=\mathfrak{t}(k_{P_i})$ without fixing any point of the finite path $\mathfrak{p}_{i}= \mathfrak{s}(\mathfrak{E}(U_0)_{P_i})$ whose length is odd. Let $\mathfrak{t}_i^{d}$ be the topological space obtained from $\mathfrak{t}_i$ by removing the central edge of $\mathfrak{p}_i$. Then, the action of $g$ on $\mathfrak{t}_i$ exchanges the two connected components of $\mathfrak{t}_i^{d}$. See Figure \ref{fig 1}\textbf{(A)}. We conclude that $g$ fixes no visual limit in $\mathfrak{t}_i$, whence it fixes no element in $\mathbb{P}^1(k)$, which contradicts the hypothesis on $g$. On the other hand, if every $\epsilon_i=0$, we have 
$$g \in \Gamma_0=\mathrm{Stab}_{\text{PGL}_2(k)}(\mathfrak{D}_0(U_0)) \cap \mathrm{Stab}_{\text{PGL}_2(k)}(\mathfrak{D}_D(U_0)) = \mathrm{Stab}_{\text{PGL}_2(k)}(\mathfrak{E}_D(U_0)).$$ 
We claim that $\Gamma_0/\text{PH}_D \hookrightarrow g(2)$, and since $g(2)$ is trivial, we get $g \in \mathrm{PH}_D$, which concludes the proof.

To prove the claim, we follow \cite[Theorem 1.2]{A2}. For $E \in \lbrace 0,D \rbrace$ we denote by $\Lambda_E$ the $\mathcal{O}_{\mathcal{C}}(U_0)$-lattice satisfying $\mathfrak{D}_E(U_0)=\mathrm{End}_{\mathcal{O}_{\mathcal{C}}(U_0)}( \Lambda_E )$. Let $h \in \Gamma_0$ be an arbitrary element, and fix $h_0 \in \mathrm{GL}_2(k)$ a lift of $h$. Then, by definition, we get $h_0 \in \mathrm{Stab}_{\mathrm{GL}_2(k)}(\mathfrak{D}_E(U_0))$, for $E \in \lbrace 0,D \rbrace$. Then, there exists $b_E \in \mathbb{I}_{U_0}$ such that $h_0 \Lambda_E= b_E \Lambda_E$ (recall \S \ref{Section Spinor}). By taking determinant in the preceding equality, we deduce that $b_E^2 \mathcal{O}_{\mathcal{C}}(U_0)=\det(h_0)\mathcal{O}_{\mathcal{C}}(U_0)$, whence we deduce $b:=b_0=b_D$, since $ \mathcal{O}_{\mathcal{C}}(U_0)$ is a Dedekind domain. Recall that $\mathrm{Pic}(A)$ is isomorphic to the ideal class group $\mathbb{I}_{U_0}/(k^{*} \prod_{Q \in |U_0|} \mathcal{O}_Q^{*})$. Moreover, note that the class $[b]$ of $b=b(h_0)$ in $\mathrm{Pic}(A)$ only depends on the class $h=[h_0] \in \mathrm{PGL}_2(k)$. Indeed, if we change the representative $h_0$ of $h$ by $\lambda h_0$, then we obtain $[b(\lambda h_0)]=[b(h_0) \cdot \mathrm{div}(\lambda)]=[b(h_0)]$. In all that follows we denote by $[b]$ the class of any $b=b(h_0)$, which only depends on $h$. Let us define $\Xi: \Gamma_0 \to \mathrm{Pic}(A)$ as the function satisfying $\Xi(h)=[b]\in \mathrm{Pic}(A)$. On one hand, note that $2 \Xi(h)=[b^2]=[\mathrm{div}(\det(h_0))]=0$. In particular, we have $\mathrm{Im}(\Xi)\subseteq g(2)$. On the other hand, if $\Xi(h)=0$, then $b=\mathrm{div}(\lambda)$, for some $\lambda \in k^{*}$. This implies that $h_0 \lambda^{-1} \in \mathrm{Aut}(\Lambda_E)=\mathfrak{D}_E(U_0)^{*}$, for all $E \in \lbrace 0,D \rbrace$. Thus, we get $h_0 \lambda^{-1} \in \mathfrak{E}_D(U_0)^{*}$, whence $h \in \mathrm{PH}_D$. Hence, we conclude $\Gamma_0/\mathrm{PH}_D$ injects into $g(2)$. 
\end{proof}

This concludes the proof of Theorem \ref{teo cusp}. In the remainder of this section, we prove Proposition \ref{lema4}, which is a stronger version of Proposition \ref{prop aux}. In \S \ref{subsection dec. grids}, we study vertices in the classification graph $C_{P_{\infty}}(\mathbb{O}_D)$. To do so, we present the concept of \textit{semi-decomposition datum} of $D$-grids. Then, in \S \ref{subsection combinatorial structure}, we analyze the topological structure of the classification graph. Specifically, we define and characterize a ramified covering $C_{P_{\infty}}(\mathbb{O}_D) \to C_{P_{\infty}}(\mathbb{O}_0) $ via the characterization of vertices in $C_{P_{\infty}}(\mathbb{O}_D)$ obtained in  \S \ref{subsection dec. grids}. Finally, we use what is known about the classifying graph $C_{P_{\infty}}(\mathbb{O}_0)$, which is summarized in the following result.

\begin{theorem}\cite[Theorem 1.2]{A1}\label{Teo Arenas cusp in C0}
The classifying graph $C_{P_{\infty}}(\mathbb{O}_0)$ is combinatorially finite, and it has exactly $\text{Pic}(R) \cong \text{Pic}(\mathcal{C}) / \langle P \rangle$ cuspidal rays. The vertices corresponding to conjugacy classes of the form $[\mathfrak{D}_{D}]$, for some divisor $D$ on $\mathcal{C}$, are located in the cuspidal rays of $C_{P_{\infty}}(\mathbb{O}_0)$. Conversely, almost every vertex in a cuspidal ray $\mathfrak{r}_{\sigma}$ of $C_{P_{\infty}}(\mathbb{O}_0)$, with $\sigma \in  \mathrm{Pic}(A)$, corresponds to a $\mathrm{GL}_2(k)$-conjugacy class of the form $[\mathfrak{D}_{B+nP_{\infty}}]$, where $B=B(\sigma)$ depends only on $\sigma$.
\end{theorem}

In Theorem \ref{Teo Arenas cusp in C0} and hereafter, by ``almost every'' we mean all but finitely many.



\subsection{On the decomposition of grids}\label{subsection dec. grids}
Here the main goal is to establish and prove a ``decomposition criterion'' for grids, and subsequently for Eichler orders. 

We fix the following notation for the rest of this section. Let $D$ be an effective divisor, and write $D=\sum_{i=1}^{r} n_i P_i$ where the points $P_1, \cdots, P_r, P_{\infty}$ are all different. Denote by $d_D$ the degree of $D$. Using Proposition \ref{eichler=grillas} we fix a finite set of places $T$ such that every every $\mathrm{GL}_2(k)$-conjugacy class of Eichler $\mathcal{C}$-orders contains a representative in $\mathrm{Eich}(D, T)$. 

For any pair of divisors $(B,B')$ such that $B+B'$ is effective, consider the Eichler $\mathcal{C}$-order
$$
\Ea[B,B']:=\sbmattrix {\oink_\mathcal{C}}{\mathfrak{L}^{-B'}}{\mathfrak{L}^{-B}}{\oink_\mathcal{C}},
$$
whose level is $B+B'$. Any $\mathrm{GL}_2(k)$-conjugate of such an order is called split. We denote an abstract grid by $\mathbb{S}$, and we often choose a representative of this class by writing $S \in \mathbb{S}$, or any verbal analog. 

\begin{defi}
For any basis $\beta \subset k^2$ we denote by $A(\beta)$ the matrix whose columns are the vectors in $\beta$. A maximal $\mathcal{C}$-order $\mathfrak{D}$ is called $\beta$-split if $A(\beta) \mathfrak{D} A(\beta)^{-1} = \mathfrak{D}_E$, for some divisor $E$. We say that a $D$-grid $S$ is $\beta$-split if every vertex $S$ is $\beta$-split as an order. This is equivalent to
\begin{equation}\label{semi-desc pi}
A(\beta) S A(\beta)^{-1} = S \left( \Ea \left[ B,B+D \right] \right),
\end{equation}
for some divisor $B$. A corner of a $D$-grid $S$ is a vertex of $S$ having a unique $P$-neighbor, for each $P$ in the support of $D$. Let $D' \leq D$ be an effective divisor. A $D'$-corner of $S$ is a $D'$-grid $S' \subset S$ containing a corner of $S$.
\end{defi}

\begin{defi}\label{defi semi-desc}
A semi-decomposition datum of $S$ is a $3$-tuple $(\beta, B, D')$, where:
\begin{itemize}
    \item[(a)] $\beta$ is a basis of $k^2$,
    \item[(b)] $B$ and $D'$ are two divisors on $\mathcal{C}$ satisfying $2D \geq 2D' \geq D$, 
    \item[(c)] there exists a corner of $S$ of the form $v_0=A(\beta)^{-1} \mathfrak{D}_{B}A(\beta)$,
    \item[(d)] there is a $\beta$-split $D'$-corner $\Pi_{\mathrm{SD}} \subseteq S$ whose set of vertices is
    $$ \mathrm{V}(\Pi_{\mathrm{SD}}) = \left\lbrace  A(\beta)^{-1}  \mathfrak{D}_{E} A(\beta) :  B \leq E\leq  B+D' \right \rbrace, $$
    \item[(e)] no vertex outside $\Pi_{\mathrm{SD}}$ is $\beta$-split.
\end{itemize}
If $D'=D$, then $(\beta,B, D')$ is called a total decomposition datum. The basis $\beta$ is called the semi-decomposition basis of $S$. The subgrid $\Pi_{\mathrm{SD}}$ is called the decomposed subgrid of $S$ associated to the datum. The degree of a semi-decomposition datum $(\beta, B , D')$ is the degree of $B$. 
\end{defi}

\begin{ex}\label{ex semi desc for mult free} Assume that $D$ is multiplicity free.
Then, condition (b) in Definition \ref{defi semi-desc} implies that any semi-decomposition datum of a concrete $D$-grid is a total decomposition datum.
\end{ex}

Note that the pair of divisors $(B,D')$ in the previous definition depends only on the $\mathrm{GL}_2(k)$-conjugacy class of $D$-grids. Indeed, if $(\beta, B, D')$ is a semi-decomposition datum of $S$, and $S=G S'G^{-1}$ with $G \in \text{GL}_2(k)$, then $(\beta', B, D')$ is a semi-decomposition datum of $S'$, where $\beta'=G(\beta)$. Furthermore, we have $A(\beta')=GA(\beta)$. This allows us to extend the definition of semi-decomposition data to abstract $D$-grids. However, in order to (partially) extend the notion of degree, we need the following result.

\begin{lemma}\label{grado bien definido}
Let $S$ be a concrete $D$-grid. Let $(\beta, B, D')$ and $(\beta^{\circ}, B^{\circ}, D^{ \circ '})$ be two semi-decomposition data of $S$ with positive degree. Then $B$ and $B^{\circ}$ are linearly equivalent.
\end{lemma}

\begin{proof} Set $A=A(\beta)$, $A^{\circ}=A(\beta^{\circ})$, and let $G$ be the base change matrix from $\beta$ to $\beta^{\circ}$. We start by showing that we can restrict our proof to the context where $D$ is multiplicity free. Indeed, $A^{\circ} S(A^{\circ})^{-1}$ is the image of the concrete $D$-grid $ASA^{-1}$ by the conjugation map induced by $G$. Let us write $D= \sum_{i=1}^{r} n_i P_i$, and let $\lbrace P_{b_1}, \cdots, P_{b_s}\rbrace$ be the set of such points with an odd coefficient $n_i$. Set $D_c=P_{b_1}+ \cdots+ P_{b_s}$. Let $\Delta$ be the subcomplex of $S$ whose vertex set is
$$ \text{V}(\Delta) = \left\lbrace  A^{-1} \mathfrak{D}_{E} A : B_0 \leq E \leq B_0+D_c \right\rbrace,$$
where $B_0=B+\sum_{i=1}^r \lfloor \frac{n_i}{2}\rfloor P_i$. Note that the intersection of the maximal $\mathcal{C}$-orders corresponding to the vertex set of $\Delta$ is an Eichler $\mathcal{C}$-order of level $D_c$. Equivalently, we get that $\Delta$ is a concrete $D_c$-grid. Since $\text{PGL}_2(k)$ acts simplicially on each grid, we get $G A \Delta A^{-1} G^{-1} = A^{\circ} \Delta (A^{\circ})^{-1}$. Moreover, we have that 
$$ \text{V}(\Delta)= \left\lbrace  (A^{\circ})^{-1} \mathfrak{D}_{E} A^{\circ} : B_0^{\circ} \leq E \leq B_0^{\circ}+D_c \right\rbrace,$$
where $B_0^{\circ}= B^{\circ}+\sum_{i=1}^r \lfloor \frac{n_i}{2} \rfloor P_i$. Thus, we conclude that the $D_c$-grid $\Delta$ has two induced positive degree total-decomposition data $(\beta, B_0, D_c)$ and $(\beta^{\circ}, B_0^{\circ}, D_c)$. Note that if $B_0$ and $B_0^{\circ}$ are linearly equivalent, then $B$ and $B^{\circ}$ are also. Therefore, by replacing $S$ by $\Delta$, we can assume that $D$ is multiplicity free.

Now, assume that $D$ is multiplicity free. In this case, every vertex in $A SA^{-1} $ and $A^{\circ} S(A^{\circ})^{-1}$ correspond to a split maximal $\mathcal{C}$-order (cf.~Example \ref{ex semi desc for mult free}). Set $\mathfrak{D}_{B''}=G\mathfrak{D}_{B} G^{-1}$. Then one of the following holds:
\begin{enumerate}
    \item $B''$ is principal,
    \item $\mathfrak{L}^{-B''}(\mathcal{C})=\lbrace 0 \rbrace$, or
    \item $\mathfrak{L}^{B''}(\mathcal{C})=\lbrace 0 \rbrace$.
\end{enumerate}
If $B''$ is principal, then $\mathfrak{D}_{B''}(\mathcal{C})= \mathbb{M}_2(\mathbb{F})$. On the other hand, since $\deg(B)>0$, we obtain $\mathfrak{L}^{-B}(\mathcal{C})=\lbrace 0 \rbrace$. Thus, we conclude $\mathfrak{D}_{B}(\mathcal{C})= \sbmattrix {\mathbb{F}^{*}}{\mathfrak{L}^{B}(\mathcal{C})}{0}{\mathbb{F}^{*}}$, which is not a simple algebra. So, the first case is impossible.

Now, assume the second case, i.e. $\mathfrak{L}^{-B}(\mathcal{C})=\mathfrak{L}^{-B''}(\mathcal{C})=\lbrace 0 \rbrace$. Then, it follows from \cite[\S 4, Proposition 4.1]{A1} that one of the following conditions holds:
\begin{itemize}
\item[(a)] $G= \sbmattrix {x}{y}{0}{z}$ and $B-B''= \text{div}(x^{-1}z)$, or
\item[(b)] $G=\sbmattrix {0}{x}{z}{0}$ and $B+B''= \text{div}(x^{-1}z)$.
\end{itemize}
In case (a), for any divisor $0 \leq E \leq D$ we have that 
\begin{equation}
 G\mathfrak{D}_{B+E} G^{-1} \subseteq \sbmattrix {\mathcal{O}_{\mathcal{C}}-yz^{-1}\mathfrak{L}^{-B-E}}{\mathfrak{J}}{xz^{-1}\mathfrak{L}^{-B-E}}{\mathcal{O}_{\mathcal{C}}+yz^{-1}\mathfrak{L}^{-B-E}},
\end{equation}
for some invertible sheaf $\mathfrak{J}$, where the other coefficients are optimal. Since  $$xz^{-1}\mathfrak{L}^{-B-E}=\mathfrak{L}^{-B-\text{div}(xz^{-1})-E}=\mathfrak{L}^{-B''-E},$$
and $G\mathfrak{D}_{B+E} G^{-1}$ is a split maximal $\mathcal{C}$-order, we deduce that $G\mathfrak{D}_{B+E} G^{-1}=\mathfrak{D}_{B''+E}$, for any divisor $0 \leq E \leq D$. This implies that $B''=B^{\circ}$, and then $B$ and $B^{\circ}$ are linearly equivalent.

In case (b), for any divisor $0 \leq E \leq D$ we have that 
\begin{equation}
 G\mathfrak{D}_{B+E} G^{-1}=\sbmattrix {\mathcal{O}_{\mathcal{C}}}{xz^{-1}\mathfrak{L}^{-B-E}}{zx^{-1}\mathfrak{L}^{B+E}}{\mathcal{O}_{\mathcal{C}}}.
\end{equation}
Then, it follows directly from the previous equation that $G\mathfrak{D}_{B+E} G^{-1}= \mathfrak{D}_{B''-E}$, for any divisor $0 \leq E \leq D$. Therefore $B''=B^{\circ}+D$, whence $B$ is linearly equivalent to $-B^{\circ}-D$. But, the last condition contradicts the hypotheses of positive degree on $B$ and $B^{\circ}$. We conclude that only case (a) can hold. 

Finally, if $\mathfrak{L}^{B''}=\lbrace 0 \rbrace$ we can replace $B''$ by $-B''$ in the preceding argument.
\end{proof}

\begin{defi}
Let $\mathbb{S}$ be an abstract $D$-grid.
A semi-decomposition datum of $\mathbb{S}$ is a pair $(B, D')$, where $(\beta, B, D')$ is a semi-decomposition datum of some concrete representative $S \in \mathbb{S}$. When $D'=D$, we say that $(B, D')$ is a total decomposition datum of $\mathbb{S}$. When $\deg(B)>0$, the degree of this datum is by definition $\deg(B)$, which is well-defined by Lemma \ref{grado bien definido}. 
\end{defi}

Let $\beta= \lbrace e_1, e_2 \rbrace$ be a basis of $k^2$. We say that a two-dimensional vector bundle $\mathfrak{L}$ on $\mathcal{C}$ is $\beta$-split if $\mathfrak{L}= \mathfrak{L}^B e_1 \oplus \mathfrak{L}^C e_2 $, where $B$ and $C$ are divisors on $\mathcal{C}$. Then, a maximal $\mathcal{C}$-order $\mathfrak{D}_{\Lambda}$ splits in the base $\beta$ if and only if at least one (and therefore every) vector bundle in the class $[\Lambda]$ is $\beta$-split. Moreover, either condition is equivalent to $\sbmattrix 1000, \sbmattrix 0001\in \mathfrak{D}_{\Lambda}(\mathcal{C})$. More generally, an Eichler $\mathcal{C}$-order is split precisely when it contains a non-trivial idempotent as a global section, or equivalently, when the corresponding grid has a total-decomposition datum. In fact, we have a more precise result that follows immediately from the current paragraph and \cite[Theorem 1.2]{A5}:

\begin{prop}\label{prop des}
Let $\mathfrak{E}$ be an Eichler $\mathcal{C}$-order of level $D$. Then the following statements are equivalent:
\begin{itemize}
\item[(1)] $\mathfrak{E}$ is split,
\item[(2)] $S=S(\mathfrak{E})$ has a total-decomposition datum,
\item[(3)] The ring of global sections $\mathfrak{E}(\mathcal{C}) \subseteq \mathbb{M}_2(k)$ contains a non trivial idempotent matrix,
\item[(4)] There exists a nontrivial idempotent matrix of $\mathbb{M}_2(k)$ contained in the ring of global sections $\mathfrak{D}(\mathcal{C})$ for every maximal $\mathcal{C}$-order $\mathfrak{D}$ corresponding to a vertex of $S$.
\end{itemize}
Assume moreover that $D$ is multiplicity free. Then, for almost all conjugacy classes in $\mathbb{O}_D$, the orders in this class are split.
\end{prop}

\begin{ex}\label{ex0} 
Assume that $D=0$. Let $Q$ be a closed point on $\mathcal{C}$, and set $R'=\mathcal{O}_{\mathcal{C}}(\mathcal{C} \smallsetminus \lbrace Q \rbrace)$. Then, each concrete $D$-grid consists in precisely one vertex, which represents a maximal $\mathcal{C}$-order. According to Theorem \ref{Teo Arenas cusp in C0}, almost all abstract $0$-grids admit a total decomposition datum, and these $0$-grids correspond to vertices located in a finite union of rays in $C_{Q}(\mathbb{O}_0)$, which are parametrized by $\text{Pic}(R') \cong \text{Pic}(\mathcal{C}) / \langle \overline{Q} \rangle$. Moreover, almost every vertex in a cuspidal ray $\mathfrak{r}_{\sigma}$ of $C_{P_{\infty}}(\mathbb{O}_0)$, with $\sigma \in  \mathrm{Pic}(A)$, corresponds to a $\mathrm{GL}_2(k)$-conjugacy class of the form $[\mathfrak{D}_{B+nP_{\infty}}]$, where $B=B(\sigma)$ depends only on $\sigma$. In particular, given a divisor $D_0$, almost all classes of maximal $\mathcal{C}$-orders have a representative of the form $\mathfrak{D}_E$ with $\deg(E) > \deg(D_0)$.
\end{ex}

The rest of this sub-section is exclusively devoted to proving the following proposition, which generalizes the previous example.

\begin{prop}\label{max semi-desc} Let $D$ be an effective divisor and $d_D=\deg(D)$. Then, for almost every abstract $D$-grid $\mathbb{S}$ there exists a semi-decomposition datum $(B,D')$ of degree $d$ for $\mathbb{S}$ with $d > d_D$. Moreover, $D'$ is unique and the class of $B$ in $\mathrm{Pic}(\mathcal{C})$ is unique. In particular, $d$ is unique.
\end{prop}

In order to prove Proposition \ref{max semi-desc} we extensively work with some subgrids defined by a certain stratification, which we formalize in Definition \ref{defi strata} and Lemma \ref{Lemma buenos caminos}.

\begin{defi}\label{defi strata}
Let $S$ be a concrete $D$-grid.

We define the $P_i$-axis $S(P_i) \subseteq S$ as the finite line in $S$ whose vertex set is $\lbrace v_j \rbrace_{j=0}^{n_i}$, where $v_0$ is as in Definition \ref{defi semi-desc}, and where $v_j$ is $P_i$-neighbor of $v_{j+1}$, whenever $0 \leq j \leq n_{i-1}$. 

Write $D=D_0+nQ$, with $D_0$ supported away from $Q$ and $n>0$. Then, the vertex set of $S$ can be naturally written as the disjoint union of the vertex sets of $n+1$ different $D_0$-grids, denoted by $S_0, \cdots, S_n$. We do this in a way such that $S_0$ is a $Q$-face of $S$, and $S_i$ is a $Q$-neighbor of $S_{i+1}$, for each $i \in \lbrace 0, \cdots, n-1 \rbrace$. These $D_0$-grids are called the $Q$-strata of $S$. See Figure \ref{fig 01}\textbf{(B)}. The numbering of the strata can be inverted if necessary. 
\end{defi}

The sequence of strata $\lbrace S_0, \cdots, S_n \rbrace$ defines a finite line in $\mathfrak{t}(k_Q)$ of length $n$. The following lemma characterizes, for almost every grid, the image of this line in the classifying graph.

\begin{lemma}\label{Lemma buenos caminos}
Let us write $D=D_0+nQ$ as above. Then, for almost every abstract $D$-grid $\mathbb{S}$, there exists $S \in \mathbb{S}$ such that the corresponding line $\mathfrak{c}(S)$ in $\mathfrak{t}(k_Q)$ is defined by vertices $ z_0, \cdots, z_s, \zeta_{s+1}, \cdots \zeta_{n} \in \mathfrak{t}(k_{Q})$ satisfying:
\begin{itemize}
\item[(i)] vertices in each pair $(z_i$, $z_{i+1})$, $(\zeta_i$, $\zeta_{i+1})$, or $(z_{s}$, $\zeta_{s+1})$ are neighbors,
\item[(ii)] $s> \lfloor \frac{n+1}{2}\rfloor $, 
\item[(iii)] $z_0, \cdots, z_{s}$ are pairwise non-$\Gamma$-equivalent vertices,
\item[(iv)] $z_0, \cdots, z_{s}$ are on the maximal path joining $\infty$ with some $\epsilon \in k$, and
\item[(v)] $\zeta_{s+i}$ and $z_{s-i}$ are $\Gamma$-equivalent, for any $i\in \lbrace 1, \cdots, s\rbrace$.
\end{itemize}
In particular, the image $\mathfrak{c}(\mathbb{S})$ of $\mathfrak{c}(S)$ in $C_Q(\mathbb{O}_{D_0})$ is a line of length $s$ contained in a cuspidal ray.
\end{lemma}

\begin{proof} Let $\mathbb{S}$ be an abstract $D$-grid, and let $S \in \mathbb{S}$ be a concrete representative. Let $\mathfrak{c}=\mathfrak{c}(S)$ be the finite line in $\mathfrak{t}(k_{Q})$ corresponding to $S$. The image $\mathfrak{c}(\mathbb{S})$ of $\mathfrak{c}$ in $C_{Q}(\mathbb{O}_{D_0})$ is a line of length $s \leq n$, which only depends on $\mathbb{S}$ by Proposition \ref{func corresp D-grillas}.

Now, as we noted in \S \ref{Section grids} (see the paragraphs before Proposition \ref{func corresp D-grillas}) the graph $C_{Q}(\mathbb{O}_{D_0})$ is combinatorially finite. Thus, there exist finitely many lines of length at most $n$ that are not contained in a cuspidal ray. Hence, we can assume that $\mathfrak{c}(\mathbb{S})$ is contained in a cuspidal ray $\mathfrak{r}$ of $C_{Q}(\mathbb{O}_{D_0})$. Let $\mathfrak{X}=  \Gamma \backslash\mathfrak{t}(k_{Q})$ be the connected component in $C_{Q}(\mathbb{O}_{D_0})$ containing $\mathfrak{r}$. Let $T$ be a maximal subtree of $\mathfrak{X}$, let $j: T \rightarrow \mathfrak{t}(k_{Q})$ be a lift of $T$, and let $\pi: \mathfrak{t}(k_{Q}) \rightarrow \mathfrak{X}$ be the canonical projection. By Corollary \ref{Cor comb fin} we may assume that the visual limit $\epsilon=\epsilon(\mathfrak{r})$ of $j(\mathfrak{r})$ belongs to $\mathbb{P}^1(k)$, and, up to changing the lift, we can assume $\epsilon\neq \infty$. Moreover, there exists a subray $\mathfrak{r}_0 \subset j(\mathfrak{r})$ such that:
\begin{itemize}
    \item $\text{V}(\mathfrak{r}_0)= \lbrace w_i\rbrace_{i=0}^{\infty}$, where $w_i$ and $w_{i+1}$ are adjacent,
    \item $\text{Stab}_{\Gamma}(w_i) \subset \text{Stab}_{\Gamma}(w_{i+1})$, and
    \item $\text{Stab}_{\Gamma}(w_i)$ acts transitively on the set of neighboring vertices of $w_i$ other than $w_{i+1}$. 
\end{itemize}
Note that the last statement implies that all neighbors of $w_i$, besides $w_{i+1}$, are in the same $\Gamma$-orbit. By induction on $L \geq 1$, we can show that, for any $i \geq 0$ and for any vertex $v$ at distance $\leq L$ of $w_{i+L}$ such that the line connecting $v$ and $w_{i+L}$ does not contain $w_{i+L+1}$, $v$ is in the same $\mathrm{Stab}_{\Gamma}(w_{i+L})$-orbit than $w_{i} \in V(\mathfrak{r}_0)$. 

Define $\mathfrak{r}_0'$ as the subray of $\mathfrak{r}_0$ with vertices are $\lbrace w_i\rbrace_{i=n}^{\infty}$. In particular, this last statement applies to $L < n$ and every vertex in $\mathfrak{r}_0'$. Since $\mathfrak{r} \smallsetminus \pi(\mathfrak{r}_0')$ is a finite line, arguing as above, we may assume that $\pi(\mathfrak{c})$ is contained in $\pi(\mathfrak{r}_0')$. This implies that $\mathfrak{c}$ is $\Gamma$-equivalent to a finite line that intersects $\mathfrak{r}_0'$ in a finite line of length $s \leq n$ (recall that $\mathfrak{t}(k_Q)$ is a tree), so we may assume that this is the case for $\mathfrak{c}$. Let us write $\mathrm{V}(\mathfrak{c})= \lbrace \varpi_i \rbrace_{i=0}^n$, where $\varpi_i$ and $\varpi_{i+1}$ are adjacent, and $\mathrm{V}(\mathfrak{c} \cap \mathfrak{r}_0') = \lbrace \varpi_i \rbrace_{i=t}^{s+t+1}$, where $\varpi_{s+t+1}$ is the vertex that is the closest to $\epsilon$. In particular, there exists $k >n$ such that, for each $i \in \lbrace t, \cdots, s+t+1 \rbrace$, the vertex $\varpi_i$ equals $w_{i+k}$. Since $t<n$ and $w_{k}$ and $\varpi_0$ are both at distance $t$ from $\varpi_t=w_{t+k}$, we see that $w_{k}$ is $\mathrm{Stab}_{\Gamma}(\varpi_{t})$-equivalent to $\varpi_0$. Then, up to replacing $\mathfrak{c}$ by a $\Gamma$-equivalent line, we may assume that $t=0$.

We claim that we can assume that $s+1 \geq \lfloor \frac{n+1}{2}\rfloor$. Indeed, if $s+1 < \lfloor \frac{n+1}{2}\rfloor$ we argue as follows. Since $n-s-1<n$ and $w_{2s+2-n+k}$ and $\varpi_n$ are both at distance $n-s-1$ from $\varpi_{s+1}=w_{s+1+k}$, we note that $w_{2s+2-n+k}$ is $\mathrm{Stab}_{\Gamma}(\varpi_{s+1})$-equivalent to $\varpi_n$. Equivalently, there exists $\gamma \in \Gamma$ such that $\gamma \cdot \lbrace \varpi_i \rbrace_{i=s+1}^{n} \subset \mathfrak{r}_0'$, and we may replace $\mathfrak{c}$ by $\gamma \cdot \mathfrak{c}$.

Write $V(\mathfrak{c})= \lbrace z_0, \cdots, z_s, \zeta_{s+1}, \cdots \zeta_{n}\rbrace$, which satisfies conditions (i), (ii) in Lemma \ref{Lemma buenos caminos}. Condition (iii) follows from the fact that the image of $\pi(\mathfrak{c})=\mathfrak{c}(\mathbb{S})$ belongs to a cuspidal ray. Condition (iv) is immediate since we can always extend a ray to a maximal path reaching infinity. Finally, condition (v) follows by the same argument used above.
\end{proof}
\begin{figure}
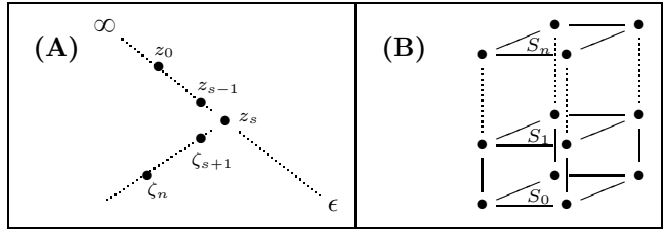

\[ 
\fbox{ \xygraph{ 
!{<0cm,0cm>;<.8cm,0cm>:<0cm,.8cm>::} 
!{(-0.2,2.6) }*+{\textbf{(A)}}="na"
!{(0.6,2.995) }*+{\infty}="e3" !{(2.6,1.4) }*+{\bullet}="f3"  !{(3,1.4) }*+{{}^{z_s}}="fi3" !{(1.5,2.3) }*+{\bullet}="f4" !{(1.6,2.5) }*+{{}^{z_0}}="f4" 
!{(2.2,1.7) }*+{\bullet}="f5"  !{(2.5,1.9) }*+{{}^{z_{s-1}}}="fi5"
 !{(4.4,0) }*+{\epsilon}="e5" 
 !{(1.5,0.2)}*+{{}^{\zeta_{n}}}="fi2" !{(1.3,0.5)}*+{\bullet}="f2"  !{(2.4,0.7)}*+{{}^{\zeta_{s+1}}}="fi4" !{(2.2,1.1)}*+{\bullet}="f4"
 !{(0.5,0) }*+{}="e4" 
 "e3"-@{.}"f3"  "f3"-@{.}"e5" "f3"-@{.}"e4"
 } }
\fbox{ \xygraph{
!{<0cm,0cm>;<0.8cm,0cm>:<0cm,0.8cm>::} 
!{(-1.5,2.6) }*+{\textbf{(B)}}="na"
 !{(-0.2,0.0) }*+{\bullet}="e2" !{(1.2,0.0) }*+{\bullet}="e3" !{(0.75,0.09) }*+{{}^{S_0}}="s1" !{(0.75,1.09) }*+{{}^{S_1}}="s2" !{(0.75,2.59) }*+{{}^{S_{n}}}="s2"
!{(-0.2,1.0) }*+{\bullet}="e4" !{(1.2,1.0) }*+{\bullet}="e5"
 !{(1,0.5) }*+{\bullet}="e6" !{(2.4,0.5) }*+{\bullet}="e7"
!{(-0.2,2.5) }*+{\bullet}="e10" !{(1.2,2.5) }*+{\bullet}="e11"
 !{(1,3.0) }*+{\bullet}="e12" !{(2.4,3.0) }*+{\bullet}="e13"
!{(1,1.5) }*+{\bullet}="e8" !{(2.4,1.5) }*+{\bullet}="e9"
 "e2"-@{-}"e4" "e2"-@{-}"e3""e4"-@{-}"e5""e3"-@{-}"e5" "e6"-@{-}"e8" "e6"-@{-}"e7""e8"-@{-}"e9""e7"-@{-}"e9" 
 "e2"-@{-}"e6" "e3"-@{-}"e7""e4"-@{-}"e8""e5"-@{-}"e9"  
"e10"-@{-}"e11" "e10"-@{-}"e12""e13"-@{-}"e11""e13"-@{-}"e12"
  "e4"-@{.}"e10" "e11"-@{.}"e5" "e12"-@{.}"e8" "e13"-@{.}"e9" 
 }}
\] 
\caption{Figure \textbf{(A)} shows the finite line $\mathfrak{c}(S)$ as in Lemma \ref{Lemma buenos caminos}, while Figure \textbf{(B)} shows the strata of a concrete $D$-grid. In the latter, vertical neighbors are $P_1$-neighbors.}\label{fig 01}
\end{figure}

We prove the existence of semi-decomposition data by induction on $r$. In order to be able to use the inductive hypothesis we need the following result.

\begin{lemma}\label{lemma finitas D-grillas} Let $D$ be an effective divisor, and write $D=D_0+nQ$, where $D_0$ is supported away from $Q$ and $n>0$. Let $\mathbb{S}_0$ be an abstract $D_0$-grid. Then the set of abstract $D$-grids $\mathbb{S}$ such that there exist $S_0 \in \mathbb{S}_0$ and $S \in \mathbb{S}$ with $S_0 \subset S$, is finite.
\end{lemma}

\begin{proof}
Let $q'$ be the cardinality of the residue field of $k_Q$. First we claim that any concrete $D_0$-grid is contained in finitely many concrete $D$-grids. In order to prove the claim, let us fix an Eichler $\mathcal{C}$-order $\mathfrak{E}_0$ of level $D_0$. Let $\mathfrak{E}$ be an Eichler $\mathcal{C}$-order of level $D$ contained in $\mathfrak{E}_0$. Since, for each $P\neq Q$, the $P$-coefficient of $D$ and $D_0$ is the same, we have $(\mathfrak{E}_0)_P=(\mathfrak{E})_P$, for all $P \neq Q$. Moreover, we have that $(\mathfrak{E}_0)_Q$ is a maximal order, while $(\mathfrak{E})_Q$ is a local Eichler order of level $n$. Denote by $v_0$ the vertex in $\mathfrak{t}(k_Q)$ corresponding to $(\mathfrak{E}_0)_Q$. Then $(\mathfrak{E})_Q$ corresponds to the intersection of all maximal orders in a finite line of length $n$ in $\mathfrak{t}(k_Q)$ containing $v_0$. Moreover, it follows from the local-global principles (b) and (c) in \S \ref{Section Spinor} that there exists a bijective map between the set of Eichler $\mathcal{C}$-orders $\mathfrak{E}$ of level $D$ contained in $\mathfrak{E}_0$, and the set of finite lines of length $n$ in $\mathfrak{t}(k_Q)$ containing $v_0$. In particular, there are finitely many Eichler $\mathcal{C}$-orders $\mathfrak{E}$ of level $D$ contained in $\mathfrak{E}_0$. Since there exists a bijective map between the set of concrete $D$-grids containing $S_0=S(\mathfrak{E}_0)$, and the set of Eichler $\mathcal{C}$-orders $\mathfrak{E}$ of level $D$ contained in $\mathfrak{E}_0$ (cf.~\S \ref{Section grids}), the claim follows.

Let us define $\mathbb{W}$ as the set of abstract $D$-grids $\mathbb{S}$ such that there exist $S_0 \in \mathbb{S}_0$ and $S \in \mathbb{S}$ with $S_0 \subseteq S$. Fix a concrete representative $S_0^{\circ} \in \mathbb{S}_0$. We define $W$ as the set of concrete $D$-grids containing $S_0^{\circ}$. Then, it follows from the previous claim that $W$ is finite. Now, we claim that the map $\phi: W \to \mathbb{W}$, defined by $\phi(S) = [S]$, is surjective. Indeed, let $\mathbb{S} \in \mathbb{W}$. Then, by definition, there exist $S_0' \in \mathbb{S}_0$ and $S' \in \mathbb{S}$ with $S_0' \subseteq S'$. Since $S_0^{\circ}$ and $S_0'$ belong to $\mathbb{S}_0$, there exists $\gamma \in \mathrm{GL}_2(k)$ such that $S_0^{\circ}=\gamma \cdot S_0'$. Thus, the $D$-grid $S=\gamma \cdot S'$ contains $S_0^{\circ}$. This implies that $S \in W$, and, by definition, we have $\phi(S)=[\gamma \cdot S']=[S']=\mathbb{S}$. So, the claim follows. In particular, since $\phi$ is surjective, we conclude that $\mathbb{W}$ is finite, which completes the proof.
\end{proof}

We are now ready to prove Proposition \ref{max semi-desc}.

\begin{proof}[Proof of Proposition \ref{max semi-desc}]
Recall that $D=\sum_{i=1}^r n_i P_i$. We prove the existence of semi-decomposition data by induction on $r \in \mathbb{N}$. Note that, if $r=0$, then $D$ is trivial, and, in such case, the result follows from Example \ref{ex0}. Now we prove the result for $r>0$. We may assume that $P_1$ has minimal degree in the support of $D$. Set $U'= \mathcal{C } \smallsetminus \lbrace P_1\rbrace$. Then, we can write $D=n_1P_1+ D_0$, where $D_0$ is supported in $U'$. Let $\mathbb{S}$ be an abstract $D$-grid, and fix a concrete representative $S \in \mathbb{S}$. Let $S_0, \cdots, S_{n_1}$ be the $D_0$-grids in the $P_1$-strata of $S$. See Definition \ref{defi strata} and Figure \ref{fig 01}\textbf{(B)}. These define a line $\mathfrak{c}(S)$ in $\mathfrak{t}(k_{P_1})$, and we may assume that it satisfies statements (i), (ii), (iii), (iv) and (v) in Lemma \ref{Lemma buenos caminos}, and that its image $\mathfrak{c}(\mathbb{S})$ is contained in a cuspidal ray of $C_{P_1}(\mathbb{O}_{D_0})$. Moreover, we can enumerate the strata of $S$ in a way such that $S_i$ corresponds to $z_i \in \mathrm{V}\big(\mathfrak{t}(k_{P_1})\big)$, for each $i \in \lbrace 0, \cdots, s \rbrace$. Now, by the inductive hypothesis and Lemma \ref{lemma finitas D-grillas}, we may assume that there exists a semi-decomposition datum $(\beta, B, D_0')$ of degree $\deg(B)>d_{D}>d_{D_0}$ for $S_0$. We claim that $(\beta, B, sP_1+D_0')$ is a semi-decomposition datum for $S$. Let $A=A(\beta)$ be the base change matrix, as defined in Definition \ref{defi semi-desc}. Then, by definition of a semi-decomposition datum, $\text{V}(S_0)$ contains the vertex set $ \mathrm{V}(\Pi(S_0))=\left\lbrace  A \mathfrak{D}_{E} A^{-1}: B \leq E \leq B+D_0'  \right\rbrace$. We already know that the vertex $\pi_V(z_i)$ corresponding to the class $\mathbb{S}_i$ of $S_i$ has valency two in $C_{P_{1}}(\mathbb{O}_{D_0})$. Denote by $\mathbb{S}_{-1} \neq \mathbb{S}_1$ the other $P_1$-neighbor of $\mathbb{S}_0$. Then, we can describe a decomposed subgrid of some concrete representatives of $\mathbb{S}_{1}$ and $\mathbb{S}_{-1}$. Indeed, we know that the $D'_0$-grid $\Pi(S_0)$ is $P_1$-neighbor to the $D_0'$-grids $\nabla_{-1}, \nabla_{1}$, whose vertex sets are respectively
$$ \text{V}(\nabla_{-1})= \lbrace A^{-1} \mathfrak{D}_{E} A : B-P_1 \leq E \leq  B-P_1+D_0' \rbrace ,$$
and
$$ \text{V}(\nabla_{1})= \lbrace A^{-1} \mathfrak{D}_{E} A : B+P_1 \leq E \leq B+P_1+D_0' \rbrace .$$

Then, we can complete $\nabla_{-1}, \nabla_1$ in order to obtain two $D_0$-grids, denoted respectively by $S_{-1}$ and $S_{1}$, which are $P_1$-neighbors to $S_0$. We claim that $S_{1}$ and $S_{-1}$ belong to different abstract $D_0$-grids. Indeed, if $S_{1}$ and $S_{-1}$ define the same abstract $D_0$-grid, then each concrete $D_0$-grid in the class has two positive degree semi-decomposition data of the form $(\beta_1, B-P_1, D_0')$ and $(\beta_2, B+P_1, D_0')$.
Note that $\deg(B)> \deg(P_1)$ by hypothesis, whence $\deg(B+P_1)>0$ and $\deg(B-P_1)>0$. Thus, by Lemma \ref{grado bien definido} we deduce that $B+P_1$ is linearly equivalent to $B-P_1$, which is impossible. So, the claim follows.

Now, we claim that $S_1 \in \mathbb{S}_1$ and $S_{-1} \in \mathbb{S}_{-1}$. In order to prove this, let $\mathfrak{E}_0'$ and $\mathfrak{E}_{-1}'$ be the Eichler $\mathcal{C}$-orders defined as the intersection of all the maximal orders corresponding to vertices of the respective grids $S_0$ and $S_{-1}$.
Since $S_0$ and $S_{-1}$ are $D_0$-grids, the level of $\mathfrak{E}_0'$ and $\mathfrak{E}_{-1}'$ is $D_0$. In particular, these Eichler $\mathcal{C}$-orders are maximal at $P_1$.
By definition $\Pi(S_0) \subseteq S_0$ and $\nabla_{-1}\subseteq S_{-1}$. This implies that $$ \mathfrak{E}_0' \subseteq \bigcap_{B \leq E \leq B+D_0'} A^{-1}\mathfrak{D}_E A, \, \, \text{  and  }  \, \, \mathfrak{E}_{-1}' \subseteq \bigcap_{B-P_1 \leq E \leq B-P_1+D_0'} A^{-1} \mathfrak{D}_E A .$$
Thus, if $m$ is the coefficient of $P_1$ in $B$, then 
$(\mathfrak{E}_0')_{P_1} \subseteq A^{-1}(\mathfrak{D}_{mP_1})_{P_1} A$ and $(\mathfrak{E}_1')_{P_1} \subseteq A^{-1}(\mathfrak{D}_{(m-1)P_1})_{P_1} A$. Moreover, since $(\mathfrak{E}_0')_{P_1}$ and $(\mathfrak{E}_{-1}')_{P_1}$ are maximal, we get that $(\mathfrak{E}_0')_{P_1} = A^{-1}(\mathfrak{D}_{mP_1})_{P_1} A$ and $(\mathfrak{E}_{-1}')_{P_1}=  A^{-1} (\mathfrak{D}_{(m-1) P_1})_{P_1} A$. 
Since the vertex $v_{-1}$ in $\mathfrak{t}(k_{P_1})$ corresponding to $S_{-1}$ is the projection of some (any) maximal $\mathcal{C}$-order in $\mathrm{V}(S_{-1})$ by localizing at $P_1$, it coincides with $A^{-1} (\mathfrak{D}_{(m-1) P_1})_{P_1} A$. This implies that $v_{-1}$ is the unique vertex in the maximal path joining the ends $\epsilon$ and $\infty$ that is further from $\epsilon$ than $z_0$. Thus, we deduce that $S_{-1} \in \mathbb{S}_{-1}$, whence we also obtain $S_{1} \in \mathbb{S}_{1}$. We conclude from this that $(\beta, B+P_1, D_0')$ is a semi-decomposition datum of degree $d_{D_0}$ of $S_1$. So, by an inductive argument, we can show that, for each $i \in \lbrace 0, \cdots, n_1 \rbrace$, we have that $(\beta, B+ i P_1, D_0')$ is a semi-decomposition datum of degree $d_{D_0}$ for $S_i$. Therefore, $(\beta, B, sP_1+D_0')$ is a semi-decomposition datum for $S$.

We are only left to prove the condition $\deg(B)> d_D$. Let $\mathbb{S}$ be an abstract $D$-grid as above, and set $S \in \mathbb{S}$. Let us denote by $v_{\mathbb{S}} \in \mathrm{V}(C_{P_{\infty}}(\mathbb{O}_D))$ the vertex corresponding to $\mathbb{S}$. Then we know that almost every vertex $v_{\mathbb{S}}$ belongs to a cuspidal ray $\mathfrak{r} \subset C_{P_{\infty}}(\mathbb{O}_D)$. In particular, almost every vertex $v_{\mathbb{S}}$ has exactly two $P_{\infty}$-neighbors. Let $v_{\mathbb{S}}^{+}$ be the neighbor of $v_{\mathbb{S}}$ that is closest to the end of $\mathfrak{r}$. Then, the argument here above shows that, for almost every abstract $D$-grid $\mathbb{S}$, with a semi-decomposition datum $(B,D')$, the abstract grid corresponding to $v_{\mathbb{S}}^{+}$ has the semi-decomposition datum $(B+P_{\infty}, D')$. See Figure \ref{fig X}. Recalling that there are finitely many cuspidal rays in $C_{P_{\infty}}(\mathbb{O}_D)$, we see that, up to increasing $n \in \mathbb{Z}$ at the divisor $B+nP_{\infty}$, for almost every abstract grid, and any concrete representative, there is a semi-decomposition datum of degree $>d_D$.

\begin{figure}
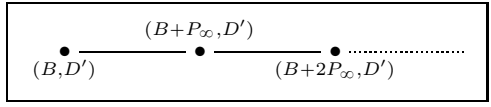

\[ 
\fbox{ \xygraph{
!{<0cm,0cm>;<0.9cm,0cm>:<0cm,0.9cm>::} 
!{(-1,0.0) }*+{{}^{(B, D')}}="i0" !{(1.0,0.6) }*+{{}^{ (B+P_{\infty}, D') }}="i1" !{(3,0.0)}*+{{}^{(B+2P_{\infty}, D')    }}="i2"
 !{(-1,0.3) }*+{\bullet}="e0"  !{(1.0,0.3) }*+{\bullet}="e1"!{(3,0.3) }*+{\bullet}="e2"
 !{(5,0.3) }*+{}="e3"   
"e2"-@{.}"e3" "e0"-@{-}"e1" "e1"-@{-}"e2"  }}
\]
\caption{A cuspidal ray $\mathfrak{r} \subseteq C_{P_{\infty}}(\mathbb{O}_D)$. Each vertex in $\mathrm{V}(\mathfrak{r})$ represents an abstract $D$-grid with a semi-decomposition datum of the form $(B+nP_{\infty}, D')$.}\label{fig X}
\end{figure}

We prove now the uniqueness of semi-decomposition data. Let $\mathbb{S}$ be an abstract $D$-grid as above, and set $S \in \mathbb{S}$ be a concrete representative. In particular, $S$ has a semi-decomposition datum $(\beta, B, D')$ of degree $>d_D$. Write $D'=\sum_{i=1}^r s_i P_i$, where $s_i \leq \lfloor \frac{n_i+1}{2}\rfloor$.
We can characterize the isomorphism class of non-$\beta$-split vertices in $\mathrm{V}(S)$, i.e. the vertices in $\mathrm{V}(S) \smallsetminus \mathrm{V}(\Pi_{\mathrm{SD}})$. Indeed, condition (v) in Lemma \ref{Lemma buenos caminos} implies that any non-$\beta$-split maximal $\mathcal{C}$-order in $\mathrm{V}(S)$ is isomorphic to a split maximal $\mathcal{C}$-order in $\Pi_{\mathrm{SD}}$. Then, the vertex set $\lbrace v_j \rbrace_{j=0}^{n_i}$ of the $P_i$-axis of $S$ satisfies that $v_j \cong \mathfrak{D}_{B+j P_i}$, if $j\leq s_i$, and $v_{j} \cong \mathfrak{D}_{B+(2s_i-j)P_i}$, if $j \geq s_i$. On the other hand, by hypothesis we have $\deg(B)> d_D \geq 0$. Then, Lemma \ref{grado bien definido} implies that $\mathfrak{D}_{B+(2s_i-j)P_i}$ is not isomorphic to $\mathfrak{D}_{B+jP_i}$, for any $j \geq s_i$.
This implies that the integers $\lbrace s_i \rbrace_{i=1}^{r}$ are unique, and then $D'$ is unique. The uniqueness of the class of $B$ follows from Lemma \ref{grado bien definido}.
\end{proof}

An immediate corollary of the end of the preceding proof is the following statement, which we state here for further reference.

\begin{corollary}\label{coro semi-desc of neighbor}
Let $\mathbb{S}$ be an abstract $D$-grid corresponding to a vertex in a cuspidal ray in $C_{P_{\infty}}(\mathbb{O}_D)$. Let $(B,D')$ be a semi-decomposition datum of $\mathbb{S}$ of degree $>d_D$. Let $\mathbb{S}^{+}$ be the abstract $D$-grid corresponding to the unique neighbor in $C_{P_{\infty}}(\mathbb{O}_D)$ that is closer to the end of the cuspidal ray. Then $(B+P_{\infty},D')$ is a semi-decomposition datum of $\mathbb{S}^{+}$.
\end{corollary}

\begin{rem} Any semi-decomposed grid with a sufficiently negative degree datum is totally decomposed. Indeed, let $S$ be a $D$-grid, and let $(\beta, B, D')$ be a semi-decomposition datum for $S$. Replacing $S$ by another representative in the same class if needed, we can assume that $\beta$ is the canonical basis, or equivalently that $A(\beta)=\mathrm{Id}$. Let us write $D=\sum_{i=1}^{r} n_iP_i$ and $D'=\sum_{i=1}^{r} s_iP_i$. For each $i \in \lbrace 1, \cdots, r \rbrace$, we denote by $S(P_i) \subset \mathfrak{t}(k_{P_i})$ the $P_i$-axis of $S$. Note that $S(P_i)$ is a length-$n_i$ line. Moreover, if we write $\mathrm{V}(S(P_i))=\lbrace v_j\rbrace_{j=0}^{n_i}$, then, for any $0 \leq j \leq s_i$, the vertex $v_j$ corresponds to the local maximal order $(\mathfrak{D}_{B+jP_i})_{P_i}$. Thus, all vertices in $\lbrace v_j \rbrace_{j=0}^{s_i}$ are located on the maximal path $\mathfrak{p}(0,\infty)$ joining the the visual limits $0$ and $\infty$ in  $\partial_{\infty}(\mathfrak{t}(k_{P_i}))= \mathbb{P}^1(k_{P_i})$. Let $\mathfrak{E}$ be the Eichler $\mathcal{C}$-order corresponding to $\mathbb{S}$, i.e., assume that $\mathbb{S}=[S(\mathfrak{E})]$. Set $U_i= \mathcal{C} \smallsetminus \lbrace P_i \rbrace$ and $\Gamma= \mathrm{Stab}_{\mathrm{PGL}_2(k)}(\mathfrak{E}(U_i))$. Arguing as in the proof of Lemma \ref{Lemma buenos caminos} we can prove that, for each $s_i+1 \leq j \leq n_i$, $v_j$ is $\Gamma$-equivalent to a vertex in $\mathfrak{p}(0,\infty)$ not in $\lbrace v_l \rbrace_{l=0}^{j-1}$. Thus, we get that $s_i=n_i$, and hence any semi-decomposition datum $(\beta, B, D')$ of $S$ is a total decomposition datum. 

Finally, note that, if $\mathfrak{I}=\sbmattrix {0}{1}{1}{0}$, then $S$ is in the same class as $\mathfrak{I} A(\beta) S A(\beta)^{-1}\mathfrak{I}$, whose total-decomposition datum $(\beta_0, D-B,D)$ has a positive degree.

\end{rem}

\subsection{On the combinatorial structure of the classifying graph}\label{subsection combinatorial structure} 

Here the main goal is to compute the cusp number of the classifying graph $C_{P_{\infty}}(\mathbb{O}_D)$, which corresponds to Proposition \ref{prop aux}.
Actually, we prove a more precise result stated in Proposition \ref{lema4} below.
To do this, we use the existence and uniqueness of semi-decomposition data of almost every abstract $D$-grid proved in the previous section. More specifically, using Proposition \ref{max semi-desc} we define a natural simplicial map $\tilde{\mathfrak{d}}: C_{P_{\infty}}(\mathbb{O}_D) \smallsetminus Y \to C_{P_{\infty}}(\mathbb{O}_0)$, where $Y$ is a finite subgraph, which is a regular cover on the cuspidal rays of $C_{P_{\infty}}(\mathbb{O}_0)$. Then, we compute the number of pre-images of a given cuspidal ray and we apply Theorem \ref{Teo Arenas cusp in C0} in order to obtain the desired result.

\begin{defi}\label{ideal tip}
Let $S$ be a concrete $D$-grid with a semi-decompostion datum $(\beta, B, D')$ of positive degree. We denote by $\delta(S)$ the $\text{PGL}_2(k)$-conjugacy class of the maximal $\mathcal{C}$-order $\mathfrak{D}_{B}$.
Let $\mathbb{S}$ be an abstract $D$-grid. Assume that $\mathbb{S}$ has a semi-decomposition datum $(B, D')$ with positive degree. We define the principal corner of $\mathbb{S}$ as $\mathfrak{d}(\mathbb{S}):=\delta(S)$, where $S \in \mathbb{S}$.
\end{defi}

It follows from Lemma \ref{grado bien definido} that, if $S,S^{\circ} \in \mathbb{S}$ have respective semi-decomposition data $(\beta, B, D')$ and $(\beta^{\circ}, B^{\circ}, {D'}^{\circ})$ with positive degree, then $\mathfrak{D}_{B}$ is $\mathrm{GL}_2(k)$-conjugate to $\mathfrak{D}_{B^{\circ}}$. Hence the previous definition is valid.

\begin{defi}
Let $\mathbb{S}$ be an abstract $D$-grid with a semi-decomposition datum $(B, D')$. Let us write $D=\sum_{i=1}^r n_i P_i$ and $D'=\sum_{i=1}^r s_i P_i$. We define the semi-decomposition vector associated to the previous datum as $l=l(\mathbb{S}):=(s_1, \cdots, s_r)$. Note that $(B, D')$ is a total decomposition datum exactly when $l=(n_1, \cdots, n_r)$.

\end{defi}

Now, it follows from Proposition \ref{max semi-desc} that there exists a finite graph $Y \subset C_{P_{\infty}}(\mathbb{O}_D)$ such that, for each vertex $v \in \mathrm{V}(C_{P_{\infty}}(\mathbb{O}_D) \smallsetminus Y)$, the corresponding abstract $D$-grid $\mathbb{S}=\mathbb{S}(v)$ has a representative with a semi-decomposition datum of degree $>\deg(D)$. Then, $\mathfrak{d}$ induces a well-defined function $\widetilde{\mathfrak{d}}: \mathrm{V}(C_{P_{\infty}}(\mathbb{O}_D) \smallsetminus Y) \rightarrow \mathrm{V}(C_{P_{\infty}}(\mathbb{O}_0))$. Let $\mathbb{S}$ be an abstract $D$-grid with a semi-decomposition datum $(B,D')$ of degree $>\deg(D)$. Let $\mathbb{S}^{+}$ be the abstract $D$-grid corresponding to the unique neighbor $v_{\mathbb{S}}^{+}$ of $v_{\mathbb{S}}$ in $C_{P_{\infty}}(\mathbb{O}_D)$ that is closer to the end of the cuspidal ray containing $v_{\mathbb{S}}$. By Corollary \ref{coro semi-desc of neighbor}, $(B+P_{\infty},D')$ is a semi-decomposition datum of $\mathbb{S}^{+}$. In this case, we get $\mathfrak{d}(\mathbb{S}^{+}) = [\mathfrak{D}_{B+P_{\infty}}]$. See Figure \ref{fig 1}\textbf{(B)}. In particular, this implies that the function $\widetilde{\mathfrak{d}}$ sends neighboring vertices into neighboring vertices. So, we can extend $\widetilde{\mathfrak{d}}$ to a simplicial map from $C_{P_{\infty}}(\mathbb{O}_D) \smallsetminus Y$ to $C_{P_{\infty}}(\mathbb{O}_0)$, which we also denote by $\widetilde{\mathfrak{d}}$. The following proposition describes the fibers of $\widetilde{\mathfrak{d}}$.

\begin{prop}\label{lema 3}
Let $\mathfrak{D}=\mathfrak{D}_B$ be a split maximal $\mathcal{C}$-order satisfying $\deg(B)> \deg(D)$. Assume that $[\mathfrak{D}] \neq \mathfrak{d}(\mathbb{S})$ for $\mathbb{S}$ corresponding to a vertex in $Y$. Then, $\widetilde{\mathfrak{d}}^{-1}([\mathfrak{D}])$ contains
\begin{itemize}
\item[(1)] Exactly one totally decomposed $D$-grid, and 
\item[(2)] Exactly $\frac{1}{q-1} \prod_{s_i \neq n_i} (q^{\deg(P_i)}-1) q^{\deg(P_i)(n_i-s_i-1)}$ semi-decomposed $D$-grids whose semi-decomposed vector is $(s_1, \cdots, s_r) \neq (n_1, \cdots, n_r)$.
\end{itemize}
\end{prop}

In order to prove this proposition we have to use the following lemma.

\begin{lemma}\label{lemma equal tip}
Let $\mathbb{S}$ and $\mathbb{S}^{\circ}$ two $D$-grids with semi-decomposition data of degree $> \deg(D)$ such that $\mathfrak{d}(\mathbb{S})=\mathfrak{d}(\mathbb{S}^{\circ})$ and $l(\mathbb{S})=l(\mathbb{S}^{\circ})$. Then there exists concrete $D$-grids $S \in \mathbb{S}$ and $S^{\circ} \in \mathbb{S}^{\circ}$ such that their respective decomposed subgrids $\Pi_{\mathrm{SD}}$ and $\Pi_{\mathrm{SD}}^{\circ}$ are equal.
\end{lemma}

\begin{proof}
Let $\mathbb{S}$ and $\mathbb{S}^{\circ}$ be two abstract $D$-grids such that $\mathfrak{d}(\mathbb{S})=\mathfrak{d}(\mathbb{S}^{\circ})$. Set $S_0\in \mathbb{S}$ and $S_0^{\circ} \in \mathbb{S}^{\circ}$, and let $(\beta, B, D')$ and $(\beta^{\circ}, B^{\circ}, D')$ be their respective semi-decomposition data. Let $\Pi_{\mathrm{SD}} \subset S_0$ and $\Pi_{\mathrm{SD}}^{\circ} \subset S^{\circ}_0$ be the respective decomposed subgrids, and write $A=A(\beta)$ and $A^{\circ}=A(\beta^{\circ})$. By hypothesis $\mathfrak{d}(\mathbb{S})=\mathfrak{d}(\mathbb{S}^{\circ})$, whence we have that $\delta(S_0)= \mathfrak{D}_{B}$ is $\mathrm{GL}_2(k)$-conjugate to $\delta(S^{\circ}_0)= \mathfrak{D}_{B^{\circ}}$. Moreover, by hypothesis, $\deg(B), \deg(B^{\circ})> \deg(D) \geq 0$. So, by \cite[\S 4, Proposition 4.1]{A1}, there exists $f \in k^{*}$ such that $B-B^{\circ}= \text{div}(f)$. Set $G= \sbmattrix {f}{0}{0}{1} \in \text{GL}_2(k)$ so that $G \mathfrak{D}_{B^{\circ}} G^{-1}=\mathfrak{D}_{B}$. We claim that $S:=(GA) S_0 (GA)^{-1} \in \mathbb{S}$ and $S^{\circ}:=(A^{\circ}) S_0^{\circ} (A^{\circ})^{-1} \in \mathbb{S}^{\circ}$ have the same decomposed subgrids. Indeed, for any divisor $E$, satisfying $0 \leq E \leq D'$, we have
$$ G \mathfrak{D}_{B^{\circ}+E} G^{-1} = \sbmattrix {\mathcal{O}_{\mathcal{C}}}{f^{-1}\mathfrak{L}^{-B^{\circ}-E}}{f \mathfrak{L}^{B^{\circ}+E}}{\mathcal{O}_{\mathcal{C}}}= \mathfrak{D}_{B+E}.$$
In particular, we obtain that $(GA) \Pi_{\mathrm{SD}} (GA)^{-1}= (A^{\circ}) \Pi_{\mathrm{SD}}^{\circ} (A^{\circ})^{-1}$, i.e. the decomposed subgrid of $S$ and $S^{\circ}$ are equal.
\end{proof}

\begin{proof}[Proof of Proposition \ref{lema 3}]
Let $\mathbb{S}$ and $\mathbb{S}^{\circ}$ be two abstract $D$-grids such that $\mathfrak{d}(\mathbb{S})=\mathfrak{d}(\mathbb{S}^{\circ})=[\mathfrak{D}_B]$, where $\deg(B) > \deg(B)$. It follows from Lemma \ref{lemma equal tip} that there exist $S \in \mathbb{S}$ and $S^{\circ} \in\mathbb{S}^{\circ}$ with the same decomposed subgrid $\Pi_{\mathrm{SD}}$. So, if all $s_i=n_i$, then $S=S'$, whence $\mathbb{S}=\mathbb{S}^{\circ}$. On the other hand, we can always consider the totally decomposed $D$-grid whose vertices are $\mathfrak{D}_{B+E}$, where $0 \leq E \leq D$. Thus (1) follows. More generally, $\mathbb{S}=\mathbb{S}^{\circ}$ if and only if there exists $g \in \text{GL}_2(k)$ such that
\begin{itemize}
\item[(1D)] $g \Pi_{\mathrm{SD}} g^{-1}= \Pi_{\mathrm{SD}}$, and
\item[(2D)] $g (S \smallsetminus \Pi_{\mathrm{SD}})g^{-1}= S^{\circ} \smallsetminus \Pi_{\mathrm{SD}}$,
\end{itemize} 
where we recall that no vertex in $S \smallsetminus \Pi_{\mathrm{SD}}$ and $S^{\circ} \smallsetminus \Pi_{\mathrm{SD}}$ is split in the canonical basis. As $\Pi_{\mathrm{SD}}$ is totally decomposed, by \cite[\S 4, Proposition 4.1]{A1}, we have $g=\sbmattrix {x}{y}{0}{z}$, where $\text{div}(x^{-1}z)=0$, i.e. $x^{-1}z \in \mathbb{F}^{*}$. Let $m_i$ be the multiplicity of $P_i$ in $B$. Then, Condition (1D) is equivalent to the following statement: \textit{For every $i \in \lbrace 1, \cdots, r \rbrace$, the action of $g \in \mathrm{GL}_2(k)$ on the tree $\mathfrak{t}_i=\mathfrak{t}(k_{P_i})$ point-wisely stabilizes the finite path $\mathfrak{c}_i$ whose vertex set is
\begin{equation}
\left\lbrace 
 \sbmattrix {\mathcal{O}_{P_{i}}}{ \pi_{i}^{-m_i-t}\mathcal{O}_{P_{i}}}{\pi_{i}^{m_i+t}\mathcal{O}_{P_{i}}}{\mathcal{O}_{P_{i}}}  : t \in \lbrace 0, \cdots, s_i \rbrace \right\rbrace,
\end{equation}
where $\pi_i \in \mathcal{O}_{P_i}$ is a local uniformizing parameter}. Moreover, this condition is equivalent to $\nu_{P_i}(y) \geq -m_i$. On the other hand, when $s_i \neq n_i$, the localizations at $P_i$ of orders in $S$ that are different from the localizations of vertices in $ \Pi_{\mathrm{SD}}$, correspond to the vertices of a line $\widehat{\mathfrak{p}}_i$ of length $(n_i-s_i-1)$ in $\mathfrak{t}(k_{P_i})$, which does not intersect the maximal path joining $0$ and $\infty$. See Figure \ref{fig 01}\textbf{(A)}. Hence, vertices in $\widehat{\mathfrak{p}}_i$ are in correspondence with the local rings of endomorphisms of the lattices
\begin{equation}
\small
\binom{a}{1}\mathcal{O}_{P_i}+ \binom{\pi_i^{-m_i-s_i-j}}{0} \mathcal{O}_{P_i},
 \normalsize
\end{equation}
where $a = a_1 \pi_{i}^{-m_i-s_i+1}+ \cdots+ a_j \pi_{i}^{-m_i-s_i+j}$, and where, for any $k>1$, we have $a_k \in \mathbb{F}(P_i)= \mathcal{O}_{P_i}/ \pi_i\mathcal{O}_{P_i}$, while $a_1 \in  \mathbb{F}(P_i)^{*}$. The same characterization holds for the localizations at $P_i$ of orders in $S^{\circ}$, by replacing $a_i$ by $a_i^{\circ} \in \mathbb{F}(P_i)$. Since $\nu_{P_i}(y) \geq -m_i$, Condition (2D) is equivalent to $a_j= a_j^{\circ}(x^{-1}z)$, for each $j \in \lbrace 1, \cdots, n_i-s_i \rbrace$. Since two $D$-grids are equal if and only if all their $P_i$-projections coincide, it follows that $\mathrm{Card}(\tilde{\mathfrak{d}}^{-1}[\mathfrak{D}_B])$ equals the number of $\mathbb{F}^{*}$-homothety classes in 
$$\prod_{s_i \neq n_i} \left( \mathbb{F}(P_i)^{*} \times \mathbb{F}(P_i)^{n_i-s_i-1}\right),$$
which is $\frac{1}{q-1} \prod_{s_i \neq n_i} (q^{\deg(P_i)}-1) q^{\deg(P_i)(n_i-s_i-1)}$.
\end{proof}

Let us denote by $\alpha(D)$ the positive integer
$$\alpha(D)=1+\frac{1}{q-1} \prod_{i=1}^r \left(  q^{\deg(P_i)\lfloor \frac{n_i}{2}\rfloor }-1\right).$$

\begin{lemma}\label{lem preimagen de eta}
Let $B$ be a divisor, whose degree is greater than $\deg(D)$, and assume that $[\mathfrak{D}_B] \neq \mathfrak{d}(\mathbb{S})$ for $\mathbb{S}$ corresponding to a vertex in $Y$. Then, $\tilde{\mathfrak{d}}^{-1}([\mathfrak{D}_B])$ has $\alpha(D)$ elements. In particular, there are $\alpha(D)$ different cuspidal rays in $C_{P_{\infty}}(\mathbb{O}_D)$ whose initial vertex corresponds to an abstract $D$-grid $\mathbb{S}$ satisfying that $\mathfrak{d}(\mathbb{S})= [\mathfrak{D}_{B}]$.
\end{lemma}

\begin{proof}
It follows directly from Proposition \ref{lema 3} that $\tilde{\mathfrak{d}}^{-1}([\mathfrak{D}_{B}])$ contains one and only one split abstract $D$-grid, and, for each possible semi-decomposition vector $l=(s_1, \cdots, s_r) \neq (n_1, \cdots, n_r)$, there are precisely $\frac{1}{q-1} \prod_{i=1}^{r} (q^{\deg(P_i)}-1) q^{\deg(P_i)(n_i-s_i-1)}$ non-split abstract $D$-grids whose semi-decomposition vector is $l$. By definition of semi-decomposition data we know that $l \in \Upsilon= \lbrace (s_1, \cdots, s_r): n_i \geq s_i \geq \lfloor \frac{n_i+1}{2}\rfloor \rbrace$. Then, $\tilde{\mathfrak{d}}^{-1}([\mathfrak{D}_{B}])$ contains
$$ 1+ \sum_{\Upsilon'} \frac{1}{q-1} \prod_{s_i \neq n_i} (q^{\deg(P_i)}-1) q^{\deg(P_i)(n_i-s_i-1)}= 1+\frac{1}{q-1} \prod_{} \left(  q^{\deg(P_i)\lfloor \frac{n_i}{2}\rfloor }-1\right)$$
different abstract $D$-grids, where $\Upsilon':= \Upsilon \smallsetminus \lbrace (n_1, \cdots, n_r) \rbrace$. Hence, the first claim is proved. Now, recall that, by Corollary \ref{coro semi-desc of neighbor}, if $\mathbb{S}$ and $\mathbb{S}^{\circ}$ are two $P_{\infty}$-neighboring $D$-grids satisfying $\mathfrak{d}(\mathbb{S})=[\mathfrak{D}_{B}]$ and $\mathfrak{d}(\mathbb{S}^{\circ})=[\mathfrak{D}_{B^{\circ}}]$, then $(B,D')$ and $( B+P_{\infty},D')$ are respective semi-decomposition data of $\mathbb{S}$ and $\mathbb{S}^{\circ}$. So, it follows that $\mathbb{S}$ and $\mathbb{S}^{\circ}$ have the same semi-decomposition vectors. Thus, the second claim follows.
\end{proof}

Recall that, by definition, the number of cusps of a disjoint finite union of graphs is the sum of the number of cusps in each of its connected components.

\begin{prop}\label{prop cuspides en C}
The number of cusps of $C_{P_{\infty}}(\mathbb{O}_D)$ equals $\mathrm{Card}(\mathrm{Pic}(R))\alpha(D)$. 
\end{prop}

\begin{proof} It follows from Theorem \ref{Teo Arenas cusp in C0} that the vertices of $C_{P_{\infty}}(\mathbb{O}_0)$ corresponding to the $\text{GL}_2(k)$-classes of split maximal $\mathcal{C}$-orders $\mathfrak{D}_B$ are located in a finite disjoint union of infinite lines or half lines in $C_{P_{\infty}}(\mathbb{O}_0)$. If we assume that a infinite line is the union of two half lines, then the number of such half lines is equal to the number of cusps of $C_{P_{\infty}}(\mathbb{O}_0)$, which coincide with $\text{Pic}(R)$.

Recall that we can remove a finite set of vertices in the classifying graphs in order to simplify some arguments. Thus, we can consider only vertices associated to abstract $D$-grids with a semi-decomposition datum of degree $>\deg(D)$ (cf.~Proposition \ref{max semi-desc}). 
Recall also that $\tilde{\mathfrak{d}}$ is a simplicial map from the disjoint union of a set of cuspidal rays in $C_{P_{\infty}}(\mathbb{O}_D)$, representing all classes of cuspidal rays, to an analog set in $C_{P_{\infty}}(\mathbb{O}_0)$. In particular, the image of a cuspidal ray of the former set under $\tilde{\mathfrak{d}}$ is also a cuspidal ray. So, $\tilde{\mathfrak{d}}$ can be seen as a function between such cusps. In particular, to compute the cusp number in $C_{P_{\infty}}(\mathbb{O}_D)$ it suffices to compute the number of pre-images of each cusp in $C_{P_{\infty}}(\mathbb{O}_0)$. Moreover, this can be reduced to computing the number of pre-images of any vertex in a cuspidal ray of $C_{P_{\infty}}(\mathbb{O}_0)$. Hence, the proposition follows from Lemma \ref{lem preimagen de eta}.
\end{proof}


We say that a cusp $\eta \in \partial^{\infty}(C_{P_{\infty}}(\mathbb{O}_D))$ is split if it is represented by a cuspidal ray formed only by vertices corresponding to totally decomposed $D$-grids. In any other case we say that the cusp is non-split. Note that the arguments given in the proof above imply that $\widetilde{\mathfrak{d}}$ induces a natural function $\widetilde{\mathfrak{d}}^{\infty}:\partial^{\infty}( C_{P_{\infty}}(\mathbb{O}_D)) \rightarrow \partial^{\infty}( C_{P_{\infty}}(\mathbb{O}_0))$ between the respective cusp sets of $C_{P_{\infty}}(\mathbb{O}_D)$ and $C_{P_{\infty}}(\mathbb{O}_0)$. Then, Lemma \ref{lema 3}.(1) and Lemma \ref{lem preimagen de eta} imply that for each $\eta \in \partial^{\infty}(C_{P_{\infty}}(\mathbb{O}_0))$ there exists a unique split cusp in $(\widetilde{\mathfrak{d}}^{\infty})^{-1}(\eta)$, and $\alpha(D)-1$ non split cusps.

\begin{rem}\label{rem 3}
We can characterize the unique split cusp in $(\widetilde{\mathfrak{d}}^{\infty})^{-1}(\eta)$. Indeed, let $\eta'_B \in \partial^{\infty}( C_{P_{\infty}}(\mathbb{O}_D))$ be the class of the cuspidal ray $\mathfrak{r}_{B,D} \subseteq C_{P_{\infty}}(\mathbb{O}_D)$, whose vertices correspond to the $\text{PGL}_2(k)$-classes of decomposable Eichler $\mathcal{C}$-orders $\Ea_{B,n}$ (equiv. the totally-decomposable abstract $D$-grids $\mathbb{S}=[S(\Ea_{B,n})]$) defined by
\begin{equation}
\Ea_{B,n}=\sbmattrix {\oink_\mathcal{C}}{\mathfrak{L}^{B+nP_{\infty}}}{\mathfrak{L}^{-B-nP_{\infty}+D}}{\oink_\mathcal{C}}, \quad n \geq 0.
\end{equation}
Then, the image by $\widetilde{\mathfrak{d}}^{\infty}$ of $\eta'_B$ is the class $\eta_B \in \partial^{\infty}( C_{P_{\infty}}(\mathbb{O}_0))$ of the cuspidal ray $\mathfrak{r}_B \subseteq C_{P_{\infty}}(\mathbb{O}_0)$, whose vertices correspond to $ \lbrace [\mathfrak{D}_{B+nP_{\infty}}] : n \geq 0 \rbrace$. This is the unique split cusp in $(\widetilde{\mathfrak{d}}^{\infty})^{-1}(\eta)$. Moreover, it follows from Theorem \ref{Teo Arenas cusp in C0} that, if we fix a representative set $\Delta_R \subset \mathrm{Div}(\mathcal{C})$ of $\mathrm{Pic}(R)$, then $\lbrace \eta_B : B \in \Delta_R \rbrace$ is a representative set of $\partial^{\infty}(C_{P_{\infty}}(\mathbb{O}_0))$.
\end{rem}

We are finally ready to prove Proposition \ref{prop aux} as an immediate consequence of the following result, which concludes the proof of Theorem \ref{teo cusp}.

\begin{prop}\label{lema4}
The number of cusps of any connected component of $C_{P_{\infty}}(\mathbb{O}_D)$ is the same, and it equals
\begin{equation}
c(D)= \alpha(D) [2\mathrm{Pic}(\mathcal{C})+\left\langle \overline{P_{a_1}}, \cdots ,\overline{P_{a_u}}, \overline{P_{\infty}}  \right\rangle : \langle \overline{P_{\infty}} \rangle],
\end{equation}
where $P_{a_1}, \cdots, P_{a_u}$ are the closed points in $\mathrm{Supp}(D) \subseteq \mathcal{C}$ whose coefficients are odd. Moreover, there are $c(D)/\alpha(D)$ split cusps in any connected component of $C_{P_{\infty}}(\mathbb{O}_D)$.
\end{prop}

\begin{proof}
We will use the bijection between the set of abstract $D$-grids and the set of Eichler $\mathcal{C}$-orders of level $D$ (cf.~Proposition \ref{eichler=grillas}). In particular, we will assign to any Eichler $\mathcal{C}$-order of level $D$ a vertex in $C_{P_{\infty}}(\mathbb{O}_D)$. Let $\mathfrak{E}$ and $\mathfrak{E'}$ be two split Eichler $\mathcal{C}$-orders associated to two vertices $v,v'$ in the same connected component of $C_{P_{\infty}}(\mathbb{O}_D)$. We denote by $\mathbb{S}$ and $\mathbb{S}'$ be the respective abstract $D$-grids corresponding to $\mathfrak{E}$ and $\mathfrak{E'}$. Assume that $\mathfrak{d}(\mathbb{S})= [\mathfrak{D}_B]$, $\mathfrak{d}(\mathbb{S}')= [\mathfrak{D}_{B'}]$, and $\deg(B), \deg(B')> \deg(D)$. Let $\Sigma_D=\Sigma(\mathbb{O}_D)$ be the spinor class field of Eichler $\mathcal{C}$-orders of level $D$. By definition of $C_{P_{\infty}}(\mathbb{O}_D)$, the vertices $v$ and $v'$ are in the same connected component if and only if $\rho(\mathfrak{E},\mathfrak{E'}) \in \lbrace \text{id}_{\Sigma_D}, [[P_{\infty},\Sigma_D/k]]\rbrace$. It follows from the Equation \eqref{eq rho} that $\rho(\mathfrak{E},\mathfrak{E'})= \rho(\mathfrak{D}_B, \mathfrak{D}_{B'})=[[B-B', \Sigma_D/k]]$. Hence, by Proposition \ref{gal grupo de clase}, $v$ and $v'$ are in the same connected component precisely when $\overline{B-B'} \in 2\text{Pic}(\mathcal{C})+\left\langle \overline{P_{a_1}}, \cdots ,\overline{P_{a_u}}, \overline{P_{\infty}} \right\rangle$.

Moreover, since the cuspidal ray $\mathfrak{r} \subseteq C_{P_{\infty}}(\mathbb{O}_{D})$ containing $v$ consists of vertices corresponding to $[\mathfrak{D}_{B+nP_{\infty}}]$ (cf.~Corollary \ref{coro semi-desc of neighbor}), we see that $\overline{B+\langle P_{\infty} \rangle}$ determines the image in $C_{P_{\infty}}(\mathbb{O}_{0})$ of $\mathfrak{r}$. This tells us that there are $[2\mathrm{Pic}(\mathcal{C})+\left\langle \overline{P_{a_1}}, \cdots ,\overline{P_{a_u}}, \overline{P_{\infty}}  \right\rangle : \langle \overline{P_{\infty}} \rangle]$ possible images via $\widetilde{\mathfrak{d}}^{\infty}$ for a cusp in a given connected component of $C_{P_{\infty}}(\mathbb{O}_D)$. Since by Lemma \ref{lem preimagen de eta} there are $\alpha(D)$ non equivalent cuspidal rays in $C_{P_{\infty}}(\mathbb{O}_D)$ covering each cuspidal ray in $C_{P_{\infty}}(\mathbb{O}_{0})$, we get that there are $c(D)$ different cusps in any connected component of $C_{P_{\infty}}(\mathbb{O}_D)$. This proves the first statement.

For the last statement, it follows from Lemma \ref{lema 3}.(1) that there exists only one split cuspidal ray in each fiber of $\widetilde{\mathfrak{d}}^{\infty}$. Thus, it suffices to count the number of cusps in the image of $\widetilde{\mathfrak{d}}^{\infty}$ in any given connected component. As we proved above, this number is precisely $c(D)/\alpha(D)$.
\end{proof}


\section{Stabilizers and amalgams}\label{Section amalgams}

In this section we analyze the structure of $\text{H}_D$ as an amalgam. More specifically, the main goal of this section is to prove Theorems \ref{teo grup} and \ref{teo grup ab}. For this reason, in all that follows we assume that $g(2)$ is trivial and that each $n_i$ is an odd positive integer. We also assume that $\mathrm{Supp}(D) \neq \emptyset$, since any other case can be reduced to Serre's result. In order to prove the aforementioned results we extensively use Bass-Serre theory (cf.~\cite[Chapter I, \S 5]{S}).

Let $\mathbf{C}$ be a set indexing all the cusps in $\mathfrak{t}_D$. It follows from Theorem \ref{teo cusp} that $\mathfrak{t}_D$ is the union of a finite graph $Y$ with a finite number of cuspidal rays, namely $\mathfrak{r}(\sigma)$, for $\sigma \in \mathbf{C}$. Moreover, the same results implies the following identity:
$$\mathrm{Card}(\mathbf{C})=c(\mathrm{H}_D)=2^r |g(2)| \left| \frac{ 2\mathrm{Pic}(\mathcal{C})+ \langle \overline{P_{\infty}} \rangle}{ \langle \overline{P_{\infty}} \rangle} \right|
\normalsize
\left(  1+ \frac{1}{q-1} \prod_{i=1}^r \left( q^{\text{deg}(P_i)\lfloor  \frac{n_i}{2} \rfloor}-1\right) \right).$$
Now, we choose a maximal tree $T$ of $\mathfrak{t}_D$ and a lift $j: T \rightarrow \mathfrak{t}=\mathfrak{t}(K)$. Note that each cuspidal ray $\mathfrak{r}(\sigma)$ of $\mathfrak{t}_D$ is contained in $T$, whence $j(\mathfrak{r}(\sigma))$ is a ray of $\mathfrak{t}$. More explicitly, we can fix a tree $T$ by taking the union of the cuspidal rays $\mathfrak{r}(\sigma)$ with a maximal tree in the finite graph $Y \subseteq \mathfrak{t}_D$.

\subsection{Review of Bass-Serre Theory}\label{subsection recalls on BS Theory} Let us recall some definitions from \S \ref{Section BTT}. Let $(s, t, r)$ and $(\tilde{s}, \tilde{t}, \tilde{r})$ be the triplets indicating source, target and reverse maps for the graphs $\mathfrak{t}$ or $\mathfrak{t}_D$ respectively. An orientation on $\mathfrak{t}_D$ is a subset $O$ of $\mathrm{E}(\mathfrak{t}_D)$ such that $\mathrm{E}(\mathfrak{t}_D)$ is the disjoint union of $O$ and $\tilde{r}(O)$. In order to simplify some of the subsequent definitions, let us fix an orientation $O$ for $\mathfrak{t}_D$, and set $o(y)=0$, if $y \in O$, while $o(y)=1$, if $y \not\in O$, i.e., if $\tilde{r}(y) \in O$. 

We extend $j$ to a function $j: \mathrm{E}(\mathfrak{t}_D) \to \mathrm{E}(\mathfrak{t})$ satisfying the relation
\begin{equation}\label{eq j ext}
j(\tilde{r}(y))=r(j(y)),
\end{equation}
as follows: For each $y \in O \smallsetminus \mathrm{E}(T)$, we choose $j(y)$ so that $s(j(y)) \in \mathrm{V}(j(T))$. For the remaining edges we define $j(y)$ by the relation \eqref{eq j ext}. Note that we have $s(j(y))=j(\tilde{s}(y))$, for all $y \in O$. In general, however, the corresponding relation for the target does not hold. Next, for each $y \in O \smallsetminus \mathrm{E}(T)$ we choose $g_{y} \in \mathrm{H}_D$ satisfying $t(j(y))=g_{y}\cdot j(\tilde{t}(y))$. This is always possible since $t(j(y))$ and $j(\tilde{t}(y))$ have the same image $\tilde{t}(y)$ in the quotient set $\mathrm{V}(\mathfrak{t}_D)$. Now, we extend the map $y \mapsto g_{y}$ to all edges in $\mathfrak{t}_D$ by setting
$g_{y}=\mathrm{id}$, for all $y \in \mathrm{E}(T)$, and for all remaining edges
$g_{r(y)}=g_{y}^{-1}$. Note that the latter relation holds for each pair of reverse edges. Therefore, for each edge $y$ in the quotient graph, we get $s(j(y))= g_{y}^{-o(y)} j(\tilde{s}(y))$ and $ t(j(y))=g_{y}^{1-o(y)}j(\tilde{t}(y))$. 

For each vertex $\overline{v}\in \mathrm{V}(\mathfrak{t}_D)$, we define $\mathrm{Stab}_{\mathrm{H}_D}({\overline{v}})$ as the stabilizer in $\mathrm{H}_D$ of the lift $j(\overline{v})$. An analogous convention applies to an edge $y$. Thus, for each pair $(\overline{v},y)$ where $\overline{v}=\tilde{t}(y)$, we have a morphism $f_y: \mathrm{Stab}_{\mathrm{H}_D}(y)\to \mathrm{Stab}_{\mathrm{H}_D}(\overline{v})$ defined by $g \mapsto g_{y}^{o(y)-1} g g_{y}^{1-o(y)}$. This function is well defined since
$$g_{y}^{o(y)-1} \mathrm{Stab}_{\mathrm{H}_D}\big(j(y)\big) g_{y}^{1-o(y)} \subseteq \mathrm{Stab}_{\mathrm{H}_D}\Big(j\big(\tilde{t}(y)\big)\Big).$$
Thus, the data presented above allow us to define the graph of groups $(\mathfrak{h}_D, \mathfrak{t}_D)=(\mathfrak{h}_D, T, \mathfrak{t}_D)$ associated to the action of $\mathrm{H}_D$ on $\mathfrak{t}$ (cf.~\cite[Chapter~I, \S 4.4]{S}). 

Now, we can define the fundamental group associated to this graph of groups. Indeed, let $F(\mathfrak{h}_D, \mathfrak{t}_D)$ be the group generated by all $\mathrm{Stab}_{\mathrm{H}_D}(\overline{v})$, where $\overline{v} \in \mathrm{V}(\mathfrak{t}_D)$, and elements $a_y$, for each $y\in \mathrm{E}(\mathfrak{t}_D)$, subject to the relations
$$ a_{\tilde{r}(y)}=a_y^{-1}, \, \,  \mathrm{ and } \, \, a_y f_y(b) a_y^{-1}= f_{\tilde{r}(y)}(b), \, \, \forall y \in \mathrm{E}(\mathfrak{t}_D), \, \,  \forall b \in \mathrm{Stab}_{\mathrm{H}_D}(y). $$
The fundamental group $\pi_1(\mathfrak{h}_D)=\pi_1(\mathfrak{h}_D,\mathfrak{t}_D)$ is, by definition, the quotient of $F(\mathfrak{h}_D,\mathfrak{t}_D)$ by the normal subgroup generated by the elements $a_y$ for $y\in \mathrm{E}(T)$. In other words, if we denote by $h_y$ the image of $a_y$ in $\pi_1(\mathfrak{h}_D, \mathfrak{t}_D)$, then the group $\pi_1(\mathfrak{h}_D, \mathfrak{t}_D)$ is generated by all $\mathrm{Stab}_{\mathrm{H}_D}(\overline{v})$, where $\overline{v} \in \mathrm{V}(\mathfrak{t}_D)$, and the elements $h_y$, for $y \in \mathrm{E}(\mathfrak{t}_D)$, subject to the relations
\[\begin{array}{cl}
    h_{\tilde{r}(y)}=h_y^{-1}, \quad
    h_y  f_y(b) h_y^{-1}= f_{\tilde{r}(y)}(b), \quad \mathrm{and} \quad
    h_z=\mathrm{id},
\end{array}\]
for all $(z,y) \in \mathrm{E}(T) \times \mathrm{E}(\mathfrak{t}_D)$ and for all $b \in \mathrm{Stab}_{\mathrm{H}_D}(y)$. It can be proven that the group $\pi_1$ is independent, up to isomorphism, 
of the choice of the graph of groups $\mathfrak{h}_D$, and in particular 
of the tree $T\subset \mathfrak{t}_D$.

As mentioned in \S \ref{Section Intro}, Bass-Serre Theory implies that all subgroups of $\mathrm{GL}_2(K)$ can be described from their actions on $\mathfrak{t}(K)$ (cf.~\cite[Chapter~I, \S 5.4]{S}). More specifically, they are isomorphic to their corresponding fundamental groups, as defined above (cf.~\cite[Chapter~I, \S 5, Theorem 13]{S}). In our case, $\mathrm{H}_D$ isomorphic to the fundamental group $\pi_1(\mathfrak{h}_D)=\pi_1(\mathfrak{h}_D, \mathfrak{t}_D)$.

Let $\mathfrak{r}(\sigma)$ be a cuspidal ray in $\mathfrak{t}_D$. We denote by $\mathcal{P}_{\sigma}$ the fundamental group $\pi_1(\mathfrak{h}_D|_{\mathfrak{r}(\sigma)})$ of the restriction of $\mathfrak{h}_D$ to $\mathfrak{r}(\sigma)$. Analogously, we define $H=\pi_1(\mathfrak{h}_D|_{Y})$, for $Y$ as in Theorem \ref{teo cusp}. For each $\sigma$ as above, let $\mathcal{B}_{\sigma}$ be the vertex stabilizer in $\mathrm{H}_D$ of the unique vertex in $Y \cap \mathfrak{r}(\sigma)$. We have canonical injections $\mathcal{B}_{\sigma} \to \mathcal{P}_{\sigma}$ and $\mathcal{B}_{\sigma} \to H$. Now, as Serre points out in \cite[Chapter~II, \S 2.5, Theorem 10]{S}, if we have a graph of groups $\mathfrak{h}$, which is obtained by ``gluing'' two graphs of groups $\mathfrak{h}_1$ and $\mathfrak{h}_2$ by a tree of groups $\mathfrak{h}_{12}$, then there exist two injections $\iota_1:\mathfrak{h}_{12} \to \mathfrak{h}_{1}$ and $\iota_2: \mathfrak{h}_{12} \to \mathfrak{h}_{2}$, such that $\pi_1(\mathfrak{h})$ is isomorphic to the sum of $\pi_1(\mathfrak{h}_1)$ and $\pi_1(\mathfrak{h}_2)$, amalgamated along $\pi_1(\mathfrak{h}_{12})$ according to $\iota_1$ and $\iota_2$. In our context, we conclude that $\mathrm{H}_D$ is isomorphic to the sum of $\mathcal{P}_\sigma$, for all $\sigma$, and $H$, amalgamated along their common subgroups $\mathcal{B}_\sigma$ according to the above injections. Since $\mathfrak{r}(\sigma) \subseteq T$, each $\mathcal{P}_{\sigma}$ coincides with the direct limit of the vertex stabilizers defined by all $v \in \mathrm{V}(\mathfrak{r}(\sigma))$. In all that follows, we exploit this property of the groups $\mathcal{P}_{\sigma}$, in order to describe them in more detail.

We start with some comments on the previous choice that simplify our work. Note that, in our context, each vertex stabilizer is finite, since it is the intersection of a compact set with a discrete subgroup of $\mathrm{GL}_2(k)$. So, replacing $\mathfrak{r}(\sigma)$ by another equivalent cuspidal ray has no effect in the statements in Theorem \ref{teo grup}. Hence, to prove the aforementioned theorem, we only need to describe the groups $\mathcal{P}_{\sigma}$ for a suitable set of cusp rays. We describe a convenient choice in what follows.

\subsection{On vertex stabilizers} For any closed point $Q$ in $\mathcal{C}$, any $s \in k$ and any $n \in \mathbb{Z}$, let $\mathfrak{D}(s,n,Q)$ be the $\mathcal{O}_{Q}$-maximal order defined by
\begin{equation}
\mathfrak{D}(s,n,Q)=  \sbmattrix {1}{0}{s}{\pi_{Q}^n} \mathbb{M}_2(\mathcal{O}_{Q}) \sbmattrix {1}{0}{s}{\pi_{Q}^n}^{-1}.
\end{equation}
We denote by $\mathcal{O}$ the ring of local integers at $P_{\infty}$, and we fix a uniformizing parameter $\pi \in \mathcal{O}$. We define $v_n(s) \in \text{V}(\mathfrak{t})$ as the vertex corresponding to the $\mathcal{O}$-maximal order $\mathfrak{D}(s,n,P_{\infty})$. 
For any $s \in \mathbb{P}^1(k)$, we define the $R$-ideal $\mathcal{Q}_s$ by $\mathcal{Q}_s= R \cap s^{-1}R \cap s^{-2}R$, when $s \in k^{*}$, and by $\mathcal{Q}_s=R$, when $s \in \lbrace 0, \infty \rbrace$. Let us write $R(n)=\lbrace a \in R: \nu(a)\geq -n \rbrace$. Recall that $I_D$ denotes the $R$-ideal $\mathfrak{L}^{-D}(U_0)$, where $U_0= \mathcal{C} \smallsetminus \lbrace P_{\infty} \rbrace$ (cf.~Equation \eqref{eq eichler}). 
So, the next lemma follows immediately from \cite[Lemma 3.2]{M} and \cite[Lemma 3.4]{M}.

\begin{lemma}\label{lema mason}
Assume first that $s =0$ and $n<0$. Then
\begin{equation}\label{eq 1 stab H_D}
\mathrm{Stab}_{\mathrm{H}_D}(v_n(0))= \left\lbrace \sbmattrix {\alpha}{c}{0}{\beta}: \alpha, \beta \in \mathbb{F}^{*} \text{ and } c \in R(-n)\right\rbrace. 
\end{equation}
On the other hand, if $s=0$ and $n\geq 1$,  then
\begin{equation}\label{eq 2 stab H_D}
\mathrm{Stab}_{\mathrm{H}_D}(v_n(0))= \left\lbrace \sbmattrix {\alpha}{0}{c}{\beta}: \alpha, \beta \in \mathbb{F}^{*} \text{ and } c \in R(n)\cap I_D \right\rbrace. 
\end{equation}
Finally, if $s \in k \smallsetminus \mathbb{F}$ and $n \deg(P_{\infty})> \deg(\mathcal{Q}_s)$, then the element $g \in \mathrm{GL}_2(k)$ belongs to $\mathrm{Stab}_{\mathrm{H}_D}(v_n(s))$ if and only if it has the form
\begin{equation}\label{eq 3 stab H_D}
g=A(\alpha, \beta, c)= \sbmattrix {\beta-sc}{(\alpha-\beta)s+s^2c}{-c}{\alpha+sc} = \sbmattrix {0}{-1}{1}{-s} ^{-1} \sbmattrix {\alpha}{c}{0}{\beta} \sbmattrix {0}{-1}{1}{-s},
\end{equation}
where
\begin{itemize}
\item[(a)] $\alpha, \beta \in \mathbb{F}^{*}$, and
\item[(b)] $ c \in R(n) \cap I_D \cap Rs^{-1} \cap ((\beta-\alpha)s^{-1}+Rs^{-2})$.
\end{itemize}
In all cases, the stabilizer group $\mathrm{Stab}_{\mathrm{H}_D}(v_n(s))$ contains triangularizable matrices only. Moreover, in \eqref{eq 2 stab H_D} and \eqref{eq 3 stab H_D}, the matrix group $$\text{Sb}_s(n):=\lbrace A(\alpha, \alpha,c): \alpha \in \mathbb{F}^{*}, c\in I_D \cap R(n) \cap \mathcal{Q}_s \rbrace,$$
which is always isomorphic to $\mathbb{F}^{*} \times ( R(n) \cap \mathcal{Q}_s \cap I_D)$, is contained in $\mathrm{Stab}_{\mathrm{H}_D}(v_n(s))$.
\end{lemma}

Let $\mathbf{C}$ be a set indexing all the cusps in $\mathfrak{t}_D$. For each $\sigma \in \mathbf{C}$, let $\mathfrak{r}(\sigma)$ be a cuspidal ray in $\mathfrak{t}_D$ representing the corresponding cusp, and let $j(\mathfrak{r}(\sigma))$ be its lift to $\mathfrak{t}$. We can assume that its visual limit $\xi=\xi_{\sigma}$ is in $\mathbb{P}^1(k)$ by Corollary \ref{Cor comb fin}, and we assume moreover that one of them is $\infty$. It follows from the previous lemma that there exists a ray $\mathfrak{r}'(\sigma)$, equivalent to $\mathfrak{r}(\sigma)$, with a vertex set $\lbrace \overline{v_i} : i=1, \cdots, \infty \rbrace$, where any pair of neighboring vertices $(\overline{v_i}, \overline{v_{i+1}})$ has the property $\mathrm{Stab}_{\text{H}_D}(j(\overline{v_i})) \subset \mathrm{Stab}_{\text{H}_D}(j(\overline{v_{i+1}}))$. Then, up to changing the representing cuspidal ray for each class, we can assume that the previous inclusion holds for each $\mathfrak{r}(\sigma)$ in $\mathfrak{t}_D$. In particular, in this case, the direct limit $\mathcal{P}_{\sigma}$ coincides with the increasing union $\bigcup_{i=1}^{\infty}\mathrm{Stab}_{\text{H}_D}(v_i)$. Thus, in order to describe $\mathcal{P}_{\sigma}$, we only need to study the stabilizers of some unbounded subset $\lbrace v_{\alpha(i)} \rbrace_{i=1}^{\infty} \subseteq \text{V}(j(\mathfrak{r}(\sigma)))$. So, let $\mathfrak{r}({\infty}) \subseteq \mathfrak{t}_D$ be the cuspidal ray at infinity, i.e., the projection on $\mathfrak{t}_D$ of a ray whose vertex set is $\lbrace v_{-n}(0): n \geq N_0\rbrace$, for certain suitable integer $N_0$. Then, Lemma \ref{lema mason} directly shows that $\mathcal{P}_{\infty}$ is isomorphic to $ (\mathbb{F}^{*} \times \mathbb{F}^{*} )\ltimes R$.

When $\mathfrak{r}(\sigma)$ is different from $\mathfrak{r}(\infty)$ we have to work with other tools. We do this in what follows. Let us fix a cuspidal ray $\mathfrak{r}( \sigma)$ different from $\mathfrak{r}(\infty)$. Then, the vertex set of $j(\mathfrak{r}(\sigma))$ equals $\lbrace v_n(\xi) : n \geq N_0 \rbrace $, where $\xi$ and $N_0$ depend on $\sigma$. It follows from Lemma \ref{lema mason} that the maximal unipotent subgroup of each $\mathrm{Stab}_{\mathrm{H}_D}(v_n(\xi))$ is isomorphic to the intersection of $R(n)$ with the $R$-ideal $\mathcal{Q}_{\xi} \cap I_D$. Therefore, since $\bigcup_{n>N_0} R(n)=R$, it follows that the maximal unipotent subgroup of $\mathcal{P}_{\sigma}$ is isomorphic to the $R$-ideal $\mathcal{Q}_{\xi} \cap I_D$. Thus, by Equation \eqref{eq 3 stab H_D}, in order to describe $\mathcal{P}_{\sigma}$, we only need to characterize the eigenvalues of some elements in $\mathrm{Stab}_{\mathrm{H}_D}(v_n(\xi))$.

We start by recalling some relevant definitions from \S \ref{Section grids}. As before, $\Gamma \subset \text{PGL}_2(k)$ denotes the stabilizer of $\mathfrak{E}_D(U_0)$, where $\mathfrak{E}_D$ is the Eichler $\mathcal{C}$-order defined in \eqref{eq de eD}. So, for each $v \in \text{V}(\mathfrak{t})$, its image $\overline{v} \in \mathrm{V}(\Gamma \backslash \mathfrak{t})$ represents one and only one $\mathcal{C}$-Eichler order of level $D$, or equivalently one abstract $D$-grid. Next, we make explicit this correspondence for any vertex $v=v_n(\xi)$. It follows from condition (c) in \S \ref{Section Spinor} that, given the family of local orders $\lbrace \mathfrak{E}(P) : P \in |\mathcal{C}| \rbrace$ defined by
\begin{itemize}
\item[($\mathrm{E}_1$)] $\mathfrak{E}({P_i})= \mathfrak{D}(0,0,P_i) \cap \mathfrak{D}(0, n_i, P_i) $, for any $i \in \lbrace 1, \cdots, n \rbrace$,
\item[($\mathrm{E}_2$)] $\mathfrak{E}({P_{\infty}})= \mathfrak{D}(\xi,n, P_{\infty})$, and
\item[($\mathrm{E}_3$)] $\mathfrak{E}({Q})=\mathfrak{D}(0,0,Q)$, for any $Q\neq P_1, \cdots, P_n, P_{\infty}$,
\end{itemize}
there exists an Eichler $\mathcal{C}$-order $\mathfrak{E}=\mathfrak{E}[v]$ such that $\mathfrak{E}_P=\mathfrak{E}(P)$, for all $P \in |\mathcal{C}|$. Then, the abstract $D$-grid corresponding to $v$ is $\mathbb{S}(v)=\big[S(\mathfrak{E}[v])\big]$. Observe that the level of $\mathfrak{E}$ is equal to $D=\sum_{i=1}^n n_i P_i$.


Now, by definition, $g \in \mathrm{Stab}_{\text{H}_D}(v_n(\xi))$ if and only if $g \in \mathrm{H}_D=\mathfrak{E}_D(U_0)^{*}$ and $g \in \mathrm{Stab}_{\text{GL}_2(k)}(v_n(\xi))$. Observe that $g \in \mathfrak{E}_D(U_0)^{*}$ is equivalent to $g \in \mathfrak{E}_D(U_0)$ and $\det(g) \in R^{*}=\mathbb{F}^{*}$. The following result describes the normalizer of local maximal orders. See \cite[\S 3]{A2} for details.

\begin{lemma}\label{lemma normalizer of maximal order}
 For any closed point $Q$, the normalizer in $\mathrm{GL}_2(k_Q)$ of a local maximal order $\mathfrak{D}(s,n,Q)$ is $k_Q^{*} \mathfrak{D}(s,n,Q)^{*}$.
\end{lemma}

\begin{proof}
Since any $\mathcal{O}_Q$-maximal order is the ring of endomorphisms of an $\mathcal{O}_Q$-lattice, we can write $\mathfrak{D}(s,n,Q)=\mathrm{End}_{\mathcal{O}_Q}(\Lambda)$, for some lattice $\Lambda=\Lambda(s,n,Q)$ of $k_Q \times k_Q$. Then $g \in \mathrm{GL}_2(k_Q)$ normalizes $\mathfrak{D}(s,n,Q)$ if and only if $\mathrm{End}_{\mathcal{O}_Q}(\Lambda)= \mathrm{End}_{\mathcal{O}_Q}(g(\Lambda))$. So, since two lattices have the same endomorphism rings precisely when they belong to the same homothety class, we have that $\mathrm{End}_{\mathcal{O}_Q}(\Lambda)= \mathrm{End}_{\mathcal{O}_Q}(g(\Lambda))$ precisely when $g(\Lambda)=  \lambda \Lambda$, for some $\lambda \in k_Q^{*}$, i.e. $\lambda^{-1}g \in \mathrm{End}_{\mathcal{O}_Q}(\Lambda)^{*}=\mathfrak{D}(s,n,Q)^{*}$.
\end{proof}

 We deduce from the preceding lemma that $g \in \mathfrak{E}_D(U_0)^{*}$ if and only if the following conditions hold:
\begin{itemize}
    \item $\det(g) \in \mathbb{F}^{*}$,
    \item $g$ normalizes $\mathfrak{D}(0,m, P_i)$, for each $P_i \in \mathrm{Supp}(D)$ and $m \in \lbrace 0, \cdots, n_i \rbrace$, and
    \item $g$ normalizes $\mathfrak{D}(0,0,Q)$, for every $Q\neq P_1, \cdots, P_r, P_{\infty}$.
\end{itemize}
Thus, we conclude that $g$ belongs to $\mathrm{Stab}_{\text{H}_D}(v_n(\xi))$ precisely when it satisfies conditions $(\mathrm{A}_1)-(\mathrm{A}_4)$ below:
\begin{itemize}
\item[($\mathrm{A}_1$)] $\text{det}(g) \in \mathbb{F}^{*}$,
\item[($\mathrm{A}_2$)] $g$ normalizes $\mathfrak{D}(0,m,P_i)$, for each $P_i \in \mathrm{Supp}(D)$ and $m \in \lbrace 0, \cdots, n_i \rbrace$,
\item[($\mathrm{A}_3$)] $g$ normalizes $\mathfrak{D}(0,0,Q)$, for every $Q\neq P_1, \cdots, P_r, P_{\infty}$, and
\item[($\mathrm{A}_4$)] $g$ normalizes $\mathfrak{D}(\xi ,n,P_{\infty})$.
\end{itemize}
Recall that we only have to describe the vertex stabilizers for an arbitrary unbounded set of vertex of $j(\mathfrak{r}(\sigma))$. Then,
by changing $v_n(\xi)$ to $v_{n+1}(\xi)$ if needed, we can assume without loss of generality that the type of $v_n(\xi)$ coincides with the type of $v_0(0)$.
Thus, there exists $h_{P_\infty} \in \mathrm{SL}_2(k_{P_{\infty}})$ such that $h_{P_\infty} \cdot v_n(\xi)=v_0(0)$, i.e., $h_{P_{\infty}} \mathfrak{D}(\sigma,n,P_{\infty}) h_{P_{\infty}}^{-1}= \mathfrak{D}(0,0,P_{\infty})$. Now, it follow from Lemma \ref{lemma normalizer of maximal order} that the $\mathrm{GL}_2(k_Q)$-normalizer of local maximal orders $\mathfrak{D}_Q$ are open. So, by the Strong Approximation Theorem applied on the open set $\mathcal{C} \smallsetminus \mathrm{Supp}(D)$, there exists $h=h(v) \in \text{SL}_2(k)$ satisfying $h \mathfrak{D}(\sigma,n,P_{\infty}) h^{-1}= \mathfrak{D}(0,0,P_{\infty})$ and normalizing each $\mathfrak{D}(0,0,Q)$, for $Q\neq P_1, \cdots, P_r, P_{\infty}$. For each $P_i \in \mathrm{Supp}(D)$, let $\mathfrak{s}_i$ be the finite line in $\mathfrak{t}(k_{P_i})$ whose vertex set is $ \lbrace h^{-1}\mathfrak{D}(0,m,P_i)h : m \in \lbrace 0, \cdots, n_i\rbrace \rbrace$. We define $S=S(v)$ as the concrete $D$-grid $S= \prod_{i=1}^n \mathfrak{s}_i$. Note that $S= h S(\mathfrak{E}[v]) h^{-1}$ is another representative of $\mathbb{S}=\mathbb{S}(v)$. Thus, we deduce from conditions $(\mathrm{A}_1)-(\mathrm{A}_4)$ that, $g$ belongs to $\mathrm{Stab}_{\text{H}_D}(v_n(\xi))$ if and only if $\widetilde{g}=hgh^{-1} \in \text{GL}_2(k)$ satisfies the following:
\begin{itemize}
\item[($\mathrm{B}_1$)] $\text{det}(\widetilde{g}) \in \mathbb{F}^{*}$,
\item[($\mathrm{B}_2$)] $\widetilde{g}$ normalizes each maximal $\mathcal{C}$-order in the vertex set of $S$, and
\item[($\mathrm{B}_3$)] $\widetilde{g}$ normalizes each local maximal order $\mathfrak{D}(0,0,Q)$, where $Q\neq P_1, \cdots, P_r$.
\end{itemize}

Let $v_{n+2}(\xi)$ be the vertex at distance two from $v_n(\xi)$ towards $\xi$. Since $C_{P_{\infty}}(\mathbb{O}_D)$ is combinatorially finite, for $n \gg 1$, we have $  \overline{v_n(\xi)} \neq \overline{v_{n+2}(\xi)}$. Moreover, it follows from Proposition \ref{max semi-desc} and Corollary \ref{coro semi-desc of neighbor} that there exists a suitable integer $N_{\sigma}$ such that for all $n> N_{\sigma}$ the actual $D$-grid $S=S(v_{n}(\sigma))$ has a semi-decomposition datum with positive degree.

First, assume that $S$ has a total-decomposition datum $(\beta, B, (n_i)_{i=1}^r)$. 
Let $A=A(\beta)$ be be the base change matrix from the canonical basis to $\beta$, and let $\mathfrak{E}= A^{-1} \mathfrak{E}[B,B+D] A$ be the split Eichler $\mathcal{C}$-order, defined as the intersection of all maximal orders in the vertex set of $S$. Then, it follows from Proposition \ref{prop des}, that there exists a global idempotent matrix $\epsilon_1 \in \mathfrak{E}(\mathcal{C})$. Since $\mathcal{T}:= \mathbb{F}^{*} \epsilon_1 + \mathbb{F}^{*}(\text{id}-\epsilon_1)$ is contined in $\mathfrak{E}(\mathcal{C})^{*}$, any matrix in $\mathcal{T}$ normalizes the Eichler $\mathcal{C}$-order $\mathfrak{E}$. In other words, any matrix $g \in \mathcal{T}$ satisfies the properties $(\mathrm{B}_1),(\mathrm{B}_2)$ and $(\mathrm{B}_3)$. Thus, for any pair of elements $a,b \in \mathbb{F}^{*}$ there is a matrix $g_{a,b} \in \mathrm{Stab}_{\mathrm{H}_D}(v_n(\xi))$ whose eigenvalues are $a$ and $b$. Then, since the group generated by $\lbrace g_{a,b}: a,b \in \mathbb{F}^{*} \rbrace$ and the maximal unipotent subgroup of $\mathrm{Stab}_{\mathrm{H}_D}(v_n(\xi))$ equals
$$\lbrace A(\alpha, \beta,c): \alpha, \beta \in \mathbb{F}^{*}, c \in R(n) \cap \mathcal{Q}_{\xi} \cap I_D \rbrace,$$
we conclude from Lemma \ref{lema mason} that
\begin{equation}
\mathrm{Stab}_{\mathrm{H}_D}(v_n(\xi)) \cong (\mathbb{F}^{*} \times \mathbb{F}^{*}) \ltimes (R(n) \cap Q_{\xi} \cap I_D).
\end{equation}
Now, assume that $S$ has a non total semi-decomposition datum  $(\beta, B, (s_i)_{i=1}^r)$, and let $A=A(\beta)$ as before. 
Then $A^{-1} \widetilde{g}A$ normalizes each maximal $\mathcal{C}$-order in the vertex set of $A^{-1} S A$. In particular, the matrix $A^{-1} \widetilde{g}A$ fixes $\mathfrak{D}_B$, with $\deg(B)>0$, whence $A^{-1} \widetilde{g}A= \sbmattrix {x}{y}{0}{z}$ with $x^{-1}z \in \mathbb{F}^{*}$ (cf.~\cite[\S 4, Proposition 4.1]{A1}). So, by taking $S^{\circ}=S$ in the proof of Proposition \ref{lema 3}, we deduce that $x=z$. Thus, the eigenvalues of $\widetilde{g}$ are equal. The same holds for $g$. Therefore, we conclude from Lemma \ref{lema mason} that $\mathrm{Stab}_{\mathrm{H}_D}(v_n(\xi))= \text{Sb}_{\xi}(n)$ in this case. In particular, we have that
\begin{equation}
\mathrm{Stab}_{\mathrm{H}_D}(v_n(\xi)) \cong \mathbb{F}^{*} \times (R(n) \cap Q_{\xi} \cap I_D).
\end{equation}

\subsection{End of proof of Theorem \ref{teo grup} and \ref{teo grup ab}} Note that it follows from Corollary \ref{coro semi-desc of neighbor} that we can assume that the semi-decomposition vectors of $\mathbb{S}(v_n(\xi))$ and $\mathbb{S}(v_{n+2}(\xi))$ are equal.
Then, by the arguments presented above, we deduce that:
\begin{itemize}
    \item $\mathrm{Stab}_{\mathrm{H}_D}(v_{n+2}(\xi)) \cong (\mathbb{F}^{*} \times \mathbb{F}^{*}) \ltimes (R(n+2) \cap Q_{\xi})$, if $\mathbb{S}(v_{n+2}(\xi))$ is split, and
    \item $\mathrm{Stab}_{\mathrm{H}_D}(v_{n+2}(\xi)) \cong \mathbb{F}^{*} \times (R(n+2) \cap Q_{\xi})$, if not.
\end{itemize}
An inductive argument shows that, for each $t \in \mathbb{Z}_{\geq 0}$, $\mathrm{Stab}_{\mathrm{H}_D}(v_{n+2t}(\xi))$ is isomorphic to
\begin{itemize}
	\item $(\mathbb{F}^{*} \times \mathbb{F}^{*}) \ltimes (R(n+2t) \cap Q_{\xi} \cap I_D)$, if  $\mathbb{S}(v_{n+2t}(\xi))$ is split, and 
	\item $\mathbb{F}^{*} \times (R(n+2t) \cap Q_{\xi} \cap I_D)$, if not.
\end{itemize}
Now, we say that $\mathfrak{r}(\sigma)$ is split when it only contains vertices corresponding to split abstract grids. Since we can assume that $\mathbb{S}(v_{n}(\xi))$ and $\mathbb{S}(v_{n+t}(\xi))$ correspond to vertices at distance $t>0$ in the same cuspidal ray of $C_{P_{\infty}}(\mathbb{O}_D)$, we get from Corollary \ref{coro semi-desc of neighbor} that they have the same semi-decomposition vectors. In particular, if $\mathfrak{r}(\sigma)$ is not split, then every vertex in $\mathfrak{r}(\sigma)$ corresponds to a nonsplit abstract grid. Thus, we conclude
\begin{equation}
\mathcal{P}_{\sigma} \cong \left\{
	       \begin{array}{ll}
		 (\mathbb{F}^{*} \times \mathbb{F}^{*}) \ltimes (Q_{\xi} \cap I_D )    & \text{if } \mathfrak{r}(\sigma) \text{ is split}, \\
		 \mathbb{F}^{*} \times ( Q_{\xi} \cap I_D) & \text{if not}.\\
	       \end{array}
	     \right.
\end{equation}
Since each $n_i$ is odd by hypothesis, it follows from Proposition \ref{lema4} that there are 
$$c(D)/\alpha(D)=[2\mathrm{Pic}(\mathcal{C})+\left\langle \overline{P_{1}}, \cdots ,\overline{P_{r}}, \overline{P_{\infty}}  \right\rangle : \langle \overline{P_{\infty}} \rangle]$$ 
split cusps in $\Gamma \backslash \mathfrak{t}$. Moreover, it follows from Lemma \ref{lemma cubrimiento de cuspides} that there are exactly $[\Gamma: \text{PH}_D]$ cusps of $\mathfrak{t}_D$ with the same image in $\Gamma \backslash \mathfrak{t}$. Thus, we conclude from Equation \eqref{eq gamma/phd} that there are $2^r [2\mathrm{Pic}(\mathcal{C})+ \langle \overline{P_{\infty}} \rangle: \langle \overline{P_{\infty}} \rangle]$ elements $\sigma \in \mathbf{C}$ such that $\mathfrak{r}(\sigma)$ is split. Then, Theorem \ref{teo grup} follows.

Now, it directly follows \cite[Chapter II, S 2.8, Exercise  2]{S} that $\mathrm{H}_D^{\mathrm{ab}}$ is finitely generated if and only if each group $\mathcal{P}_{\sigma}^{\mathrm{ab}}$ is finite, and, in any other case, $\mathrm{H}_D^{\mathrm{ab}}$ is the product of a finitely generated group with a $\mathbb{F}_p$-vector space of infinite dimension, where $p=\mathrm{Char}(\mathbb{F})$. Moreover, it follows from a direct computation that $\mathcal{P}_{\sigma}^{\mathrm{ab}}=\mathbb{F}^*\times \mathbb{F}^*$, if $\mathfrak{r}(\sigma)$ is split, while $\mathcal{P}_{\sigma}^{\mathrm{ab}}=\mathcal{P}_{\sigma}$, otherwise. Thus, the group $\mathrm{H}_D^{\mathrm{ab}}$ is finitely generated if and only if $\mathfrak{t}_D$ has only split cusps or, equivalently, if $\mathrm{Card}(\mathbf{D})=c(\mathrm{H}_D)$ (cf.~Theorem \ref{teo grup}). Then, it follows from Theorem \ref{teo cusp} that $\mathrm{H}_D^{\mathrm{ab}}$ is finitely generated if and only if each $n_i$ equals one.

\subsection{An example where \texorpdfstring{$\mathcal{C}=\mathbb{P}^1_{\mathbb{F}}$}{CP}}
Let $D$ be a square free divisor supported at a closed point $P$ of degree one. Let $P_{\infty}$ be another closed point of degree one. Without loss of generality, we assume that $P_{\infty}$ is the point at infinity of $\mathbb{P}^1_{\mathbb{F}}$ and $D= \mathrm{div}(t)|_{U_0}$, whence $R=\mathbb{F}[t]$. Then $\mathcal{O}=\mathcal{O}_{P_{\infty}}=\mathbb{F}[[t^{-1}]]$ and $\pi=1/t$ is a uniformizing parameter. Since $r=1$, $\deg(P)=1$ and $\deg(P_{\infty})=1$, Theorem \ref{teo cusp} implies that $\mathfrak{t}_D=\mathrm{H}_D\backslash \mathfrak{t}$ is the union of a doubly infinity ray with a certain finite graph.

In order to perform some explicit computations, in all that follows, we interpret the Bruhat-Tits tree $\mathfrak{t}$ in terms of closed balls in $K$, as we describe in \S \ref{subsection BTT}. Using this interpretation, it is not hard to see that the set of vertices in the doubly infinity ray $\mathfrak{p}_{a,\infty}$, joining $a \in k$ with $\infty$, corresponds to the set of closed balls $B_a^{|r|}$ whose center is $a$ and its radius is $|\pi^r|$. Then, it is immediately that the union of all these doubly infinity rays covers $\mathfrak{t}$.

\begin{lemma}\label{lemma quot t} One has
$\mathrm{H}_D \cdot \mathrm{V}(\mathfrak{p}_{0,\infty})=\mathrm{V}(\mathfrak{t})$ and $\mathrm{H}_D \cdot \mathrm{E}(\mathfrak{p}_{0,\infty})=\mathrm{E}(\mathfrak{t})$.
\end{lemma}

\begin{proof}
Firstly, we prove $\mathrm{H}_D \cdot \mathrm{V}(\mathfrak{p}_{0,\infty})=\mathrm{V}(\mathfrak{t})$. Assume, by contradiction, that there exists a vertex in $\mathfrak{t}$ that fails to belong to the same $\mathrm{H}_D$-orbit as any vertex in $\mathfrak{p}_{0,\infty}$. Let $B=B_s^{|r|}$ be a such vertex at the minimal distance from $\mathfrak{p}_{0,\infty}$. Since $B'=B_s^{|r-1|}$ is closer than $B$ to $\mathfrak{p}_{0,\infty}$, we have that $h \cdot B' \in \mathfrak{p}_{0,\infty}$, for some $h \in \mathrm{H}_D$. Thus, up to replace $B$ by $h \cdot B$, we assume that $B$ is at distance one from $\mathfrak{p}_{0,\infty}$. In other words, we assume that $B=B_s^{|\nu(s)+1|}$. For each $f \in R$, let us write:
$$\tau_f= \sbmattrix {1}{-f}{0}{1} , \quad I= \sbmattrix {0}{1}{1}{0}, \, \, \text{and} \quad \sigma_f= \sbmattrix {1}{0}{-tf}{1}= I \cdot \tau_{tf}  \cdot  I .$$
Note that $\tau_f, \sigma_f \in \mathrm{H}_D$, for any $f \in R$. 
In all that follows we extensively use the fact that, for any $p \in \mathbb{Z}$ and any pair $x,s \in k$, we have the following relations:
$$ \tau_x\cdot B_s^{|p|} = B_{s-x}^{|p|}, \quad I \cdot B_0^{|p|}= B_0^{|-p|}, \quad \mathrm{and} \quad I \cdot B_s^{|p|}= B_{1/s}^{|p-2\nu(s)|}, \, \, \mathrm{if} \, \, 0 \notin B_s^{|p|}.$$ 
Assume that $\nu(s) \leq 0$. Then, we can write $s=f_0+\epsilon_0$, where $f_0 \in R=\mathbb{F}[t]$ and $\nu(\epsilon_0) \geq 1 $. 
 Note that $\tau_{f_0} \cdot B_s^{|p|}=B_{\epsilon_0}^{|p|}$ belongs to $\mathrm{V}(\mathfrak{p}_{0,\infty})$ if and only if $\nu(\epsilon_0) \geq p$. In particular, we obtain $\tau_{f_0} \cdot B \in \mathrm{V}(\mathfrak{p}_{0,\infty})$. Thus, this case is impossible. 
Now, assume that $\nu(s) \geq 1$. Then, $1/s= tf_0 + \epsilon_0$, where $f_0 \in R$ and $\nu(\epsilon_0) \geq 1 $. 
Assuming $p > \nu(s)$, we get
$$\sigma_{f_0} \cdot B_s^{|p|}= (I \cdot \tau_{t f_0} \cdot I )\cdot B_s^{|p|}= (I \cdot \tau_{t f_0}) \cdot B_{1/s}^{|p-2\nu(s)|} = I \cdot B_{\epsilon_0}^{|p-2\nu(s)|}. $$
So, since $I \cdot \mathfrak{p}_{0,\infty}=\mathfrak{p}_{0,\infty}$, we conclude that $\sigma_{f_0} \cdot B_s^{|p|} \in \mathrm{V}(\mathfrak{p}_{0,\infty})$ precisely when $B_{\epsilon_0}^{|p-2\nu(s)|} \in \mathrm{V}(\mathfrak{p}_{0,\infty})$, or equivalently, when $\nu(\epsilon_0)+2\nu(s) \geq p $. Thus, we get that $\sigma_{f_0} \cdot B \in \mathrm{V}(\mathfrak{p}_{0, \infty})$, which is also absurd. Therefore, we conclude $\mathrm{H}_D \cdot \mathrm{V}(\mathfrak{p}_{0,\infty})=\mathrm{V}(\mathfrak{t})$. Moreover, since $\mathrm{H}_D$ acts simplicially on $\mathfrak{t}$, we also conclude $\mathrm{H}_D \cdot \mathrm{E}(\mathfrak{p}_{0,\infty})=\mathrm{E}(\mathfrak{t})$.
\end{proof}

For each $m \in \mathbb{Z}_{\geq 0}$, let us denote by $\mathbb{F}[t]_m$ the set of polynomials whose degree is less of equal than $m$. 

\begin{lemma}\label{lemma stab t}
The stabilizers in $\mathrm{H}_D$ of the vertex $v_n \in \mathrm{V}(\mathfrak{p}_{0, \infty})$ corresponding to the ball $B_{0}^{|n|}$ is 
$\sbmattrix {\mathbb{F}^*}{\mathbb{F}[t]_{-n}}{0}{\mathbb{F}^*}$, if $n <0$, $\sbmattrix {\mathbb{F}^*}{0}{0}{\mathbb{F}^*}$, if $n=0$, and $\sbmattrix {\mathbb{F}^*}{0}{t\mathbb{F}[t]_{n-1}}{\mathbb{F}^*}$, if $n>0$.
\end{lemma}

\begin{proof} 
By definition, we have $\mathrm{Stab}_{\mathrm{H}_D}(v_n)= \mathrm{Stab}_{\mathrm{GL}_2(K)}(v_n) \cap \mathrm{H}_D$, where the stabilizer in $\mathrm{GL}_2(K)$ of $v_n$ equals to $K^{*}\sbmattrix {\oink}{\pi^{n} \oink}{\pi^{-n} \oink}{\oink} ^{*}$. Since $\det(\mathrm{H}_D)=\mathbb{F}^{*}$, we have $\mathrm{Stab}_{\mathrm{H}_D}(v_n)= \sbmattrix {\oink}{\pi^{n} \oink}{\pi^{-n} \oink}{\oink} ^{*} \cap \mathrm{H}_D$. Then, the result follows from a straightforward computation. 
\end{proof}

\begin{prop}\label{teo quot t} If $\mathcal{C}=\mathbb{P}^1_\mathbb{F}$, $D=P$ is a closed point and $\deg(P_{\infty})=\deg(P)=1$, then
the doubly infinity ray $\mathfrak{p}_{0,\infty}$ is isomorphic to the quotient graph $\mathfrak{t}_D=\mathrm{H}_D\backslash \mathfrak{t}$.
\end{prop}

\begin{proof}
Firstly, we claim that $\mathrm{V}(\mathfrak{p}_{0,\infty})$ contains one and only one representative of each $\mathrm{H}_D$-orbit in $\mathrm{V}(\mathfrak{t})$. It follows from Lemma \ref{lemma quot t} that, in order to prove this claim, we just have to prove that any two different vertices in $\mathfrak{p}_{0,\infty}$ belong to different $\mathrm{H}_D$-orbits. Let $v_n=B_0^{|n|}$ and $v_m=B_0^{|m|}$ be two vertices in $\mathrm{V}(\mathfrak{p}_{0,\infty})$, which belong in the same $\mathrm{H}_D$-orbit. Then, the stabilizers $S_n, S_m \subseteq \mathrm{H}_D$ of $v_n$ and $v_m$ are conjugates, say $S_n= h S_m h^{-1}$, for some $h \in \mathrm{H}_D$. In particular, they have the same number of elements. Thus, it follows from Lemma \ref{lemma stab t} that $n=m$ or $n=-m+1$. If the second alternative holds, then $h\tau \in \mathrm{Stab}_{\mathrm{GL}_2(K)}(v_n)$, where $\tau=\sbmattrix {\pi^{-2m+1}}{0}{0}{1}$. Since Lemma \ref{lemma normalizer of maximal order} implies that $\mathrm{Stab}_{\mathrm{GL}_2(K)}(v_n)=K^{*}\mathfrak{D}_{n}^*$, where $\mathfrak{D}_n$ is the maximal order corresponding to $v_n$, we have $\det(h\tau)$ belongs to $K^{*2}\mathcal{O}^{*}$, which is impossible since $\det(\tau) \in \pi K^{*2}\mathcal{O}^{*}$ and $\det(h) \in \mathbb{F}^{*}$. Then $n=m$, which proves the claim.

Now, we claim that $\mathrm{E}(\mathfrak{p}_{0,\infty})$ contains exactly one representative of each $\mathrm{H}_D$-orbit in $\mathrm{E}(\mathfrak{t})$. Note that, since $\mathfrak{p}_{0,\infty}$ does not contain two vertices in the same $\mathrm{H}_D$-orbit, it also fails to have two edges in the same $\mathrm{H}_D$-orbit. Thus, the latter claim follows from Lemma \ref{lemma quot t}.
\end{proof}


\begin{prop}\label{teo group t}
The group $\mathrm{H}_D$ of invertible matrices in $\sbmattrix {\mathbb{F}[t]}{\mathbb{F}[t]}{t\mathbb{F}[t]}{\mathbb{F}[t]}$ is isomorphic to the sum of $\sbmattrix {\mathbb{F}^{*}}{\mathbb{F}[t]}{0}{\mathbb{F}^{*}}$ and $\sbmattrix {\mathbb{F}^{*}}{0}{t\mathbb{F}[t]}{\mathbb{F}^{*}}$ amalgamated along the diagonal group $\mathbb{F}^{*} \times \mathbb{F}^{*}$.
\end{prop}

\begin{proof}
Let $\mathfrak{r}_1$ (resp. $\mathfrak{r}_2)$ be the ray whose vertex set is, in the notations of Lemma \ref{lemma stab t}, the set $\lbrace v_n: n\geq 0\rbrace$ (resp. $\lbrace v_n: n\leq 0\rbrace$). Let $\mathcal{P}_1$ (resp. $\mathcal{P}_2)$ be the fundamental group of the graph of groups defined on $\mathfrak{r}_1$ (resp. $\mathfrak{r}_2$) by restriction. Denote by $\mathcal{B}_0$ the fundamental group of the graph of groups defined on the intersection $\mathfrak{r}_1 \cap \mathfrak{r}_2$. Since $\mathfrak{p}_{0, \infty}=\mathfrak{r}_1 \cup \mathfrak{r}_2$, it is not hard to see that the fundamental group defined on $\mathfrak{p}_{0, \infty}$, i.e. $\mathrm{H}_D$, is the sum of $\mathcal{P}_1$ with $\mathcal{P}_2$ amalgamated along $\mathcal{B}_0$. Moreover, since $\mathfrak{r}_1 \cap \mathfrak{r}_2= \lbrace v_0 \rbrace$, we have that $\mathcal{B}_0$ is the stabilizer of $v_0$, which coincides with a diagonal group isomorphic to $\mathbb{F}^{*} \times \mathbb{F}^{*}$. Since $\mathfrak{r}_1$ is simply connected and Lemma \ref{lemma stab t} implies that $\mathrm{Stab}_{\mathrm{H}_D}(v_n) \subset \mathrm{Stab}_{\mathrm{H}_D}(v_{n+1})$, for any $n \geq 0$, we have that $\mathcal{P}_1$ the union of the preceding vertex stabilizers. Thus, Lemma \ref{lemma stab t} implies that $\mathcal{P}_1$ equals $\sbmattrix {\mathbb{F}^{*}}{\mathbb{F}[t]}{0}{\mathbb{F}^{*}}$ . Analogously, $\mathcal{P}_2$ equals $\sbmattrix {\mathbb{F}^{*}}{0}{t\mathbb{F}[t]}{\mathbb{F}^{*}}$, whence the result follows.
\end{proof}

\begin{prop}\label{teo group ab t}
With the same notation and hypotheses of Theorem \ref{teo group t}, the group $\mathrm{H}_D^{\mathrm{ab}}$ is isomorphic to $\mathbb{F}^* \times \mathbb{F}^*$.
\end{prop}

\begin{proof}
Let us write $\mathrm{H}_D$ as the sum of $\mathcal{P}_1= \sbmattrix {\mathbb{F}^{*}}{\mathbb{F}[t]}{0}{\mathbb{F}^{*}}$ and $\mathcal{P}_2
=\sbmattrix {\mathbb{F}^{*}}{0}{t\mathbb{F}[t]}{\mathbb{F}^{*}}$ amalgamated along the diagonal group $\mathcal{B}_0 \cong \mathbb{F}^* \times \mathbb{F}^*$. It follows from \cite[\S 3]{Hu} that, naturally defined from the preceding amalgamated sum, we have the following Mayer-Vietoris exact sequence
$$
  \cdots \xrightarrow{\delta} H_i(\mathcal{B}_0, \mathbb{Z}) \xrightarrow{\iota_*} H_i(\mathcal{P}_1, \mathbb{Z}) \oplus H_i(\mathcal{P}_2, \mathbb{Z}) \xrightarrow{d_*}  H_{i}(\mathrm{H}_D, \mathbb{Z}) \xrightarrow{\delta} \cdots,
$$
where any group acts trivially on $\mathbb{Z}$, $\delta$ is the level morphism, and the morphisms $\iota_*$ and $d_*$ are defined by $\iota(x)=(x,x)$ and $d(x,y)=x-y$, respectively. In particular, at level $i=1$, we obtain that $
 H_1(\mathcal{B}_0, \mathbb{Z}) \xrightarrow{\iota_*} H_1(\mathcal{P}_1, \mathbb{Z}) \oplus H_1(\mathcal{P}_2, \mathbb{Z}) \xrightarrow{d_*}  H_{1}(\mathrm{H}_D, \mathbb{Z}) \rightarrow 0
$ is exact. Thus, the morphism  $d_*:\mathcal{P}_1^{\mathrm{ab}} \oplus\mathcal{P}_2^{\mathrm{ab}} \to \mathrm{H}_D^{\mathrm{ab}}$ defined by $d_*(x,y)=x-y$ is surjective. Since the canonical projection restricts to an isomorphism between the group $\mathcal{B}_0$ and either $\mathcal{P}_1^{\mathrm{ab}}$ or $ \mathcal{P}_2^{\mathrm{ab}}$, the result follows.
\end{proof}


\section*{Acknowledgements} I wish to express my hearty thanks to my PhD advisors Luis Arenas-Carmona and Giancarlo Lucchini Arteche for several helpful remarks. I also want to express my gratitude with Anid-Conicyt by the Doctoral fellowship No $21180544$.

\scriptsize

$ $

{\sc Claudio Bravo}\\
  Universidad de Chile, Facultad de Ciencias, 
  \\Casilla 653, Santiago, Chile\\
  \email{claudio.bravo.c@ug.uchile.cl}


\begin{thebibliography}{99}

\bibitem[AbB08]{Brown}
{\sc P.~Abramenko and K.~S.~Brown}, {\em Buildings: Theory and applications}, Graduate Texts in Mathematics \textbf{248}, Springer, New York, 2008.


\bibitem[A12]{abelianos}
{\sc L. Arenas-Carmona}, {\em Representation fields for
commutative orders}, Ann. Inst. Fourier (Grenoble) \textbf{62} (2012), 807-819. 




\bibitem[A13]{A13}
{\sc L. Arenas-Carmona}, {\em Eichler orders, trees and representation fields}, Int. J. Number Theory \textbf{9} (2013), no. 7, 1725-1741.


\bibitem[A14]{A1}
{\sc L. Arenas-Carmona}, {\em Computing quaternion quotient graphs via representation of orders}, J. of Algebra \textbf{402} (2014), 258-279.



\bibitem[A16]{A2}
{\sc  L. Arenas-Carmona}, {\em Spinor class fields for generalized Eichler orders},  
 J. Th\'eor. Nombres Bordeaux \textbf{28} (2016), 679-698.
 

\bibitem[AAC18]{AAC}
{\sc M. Arenas}, {\sc L. Arenas-Carmona}, and {\sc J. Contreras}, {\em On optimal embeddings and trees}, J.
Number Theor. \textbf{91} (2018), 91-117.



\bibitem[ABp]{A5}
{\sc L. Arenas-Carmona} and {\sc C. Bravo}, {\em On genera containing non-split Eichler orders over function fields}, 	arXiv:1905.08244 [math.NT].


\bibitem[ABLL]{ABLL}
{\sc L. Arenas-Carmona}, {\sc C. Bravo}, {\sc B. Loisel} and {\sc G. Lucchini-Arteche} {\em Quotients of the Bruhat-Tits tree by arithmetic subgroups of special unitary groups}, J. of Pure and Applied Algebra \textbf{266} (2021). 





\bibitem[BH91]{Bri}
{\sc M. Bridson, A. Haefliger}, {\em Metric Spaces of Non-Positive Curvature}, Springer, Berlin (1991).



\bibitem[BKW13]{B}
{\sc K. Bux, R. K{\"o}hl, S. Witzel}, {\em Higher Finiteness Properties of Reductive Arithmetic Groups in Positive Characteristic: the Rank Theorem}, Annals of Mathematics \textbf{177} (2013), 311–366.


\bibitem[Hu15]{Hu}
{\sc K. Hutchinson}, {\em On the low-dimensional homology of $\mathrm{SL}_2(k[t,t^{-1}])$}, J. of Algebra \textbf{425} (2015) 324-366.

\bibitem[Ih66]{Ihara}
{\sc Y. Ihara}, {\em On discrete subgroups of the two by two projective linear group over $p$-adic fields}, J. Math. Soc. Japan \textbf{18} (1966), 219-235.




\bibitem[Mar09]{Margaux}
{\sc B. Margaux}, {\em The structure of the group $G(k[t])$: Variations on a theme of Soul\'e}, Algebra and Number Theory \textbf{3} (2009), 393-409.

\bibitem[Ma01]{M}
{\sc A. W. Mason}, {\em Serre's generalization of Nagao's theorem: an elementary approach}, Trans. Amer. Math. Soc. \textbf{353} (2001), 749-767.



\bibitem[MS03]{M3}
{\sc A. W. Mason and A. Schweizer}, {\em The minimum index of a non-congruence
subgroup of $\mathrm{SL}_2$ over an arithmetic domain}.    Israel J. Math. \textbf{133} (2003), 29-44.


\bibitem[MS13]{M2}
{\sc A. W. Mason and A. Schweizer}, {\em The stabilizers in a Drinfeld modular group of the vertices of its Bruhat-Tits tree: an elementary approach}. Internat. J. Algebra Comput. \textbf{23} (2013), 1653-1683.



\bibitem[Na59]{N}
{\sc H. Nagao}, {\em On $\mathrm{GL}(2,K[x])$}, J. Inst. Polytech. Osaka City Univ. Ser. A \textbf{10} (1959), 117-121.

\bibitem[Ne99]{Neukirch}
{\sc J. Neukirch}, {\em Algebraic number theory}, Springer Verlag, Berlin, 1999.


\bibitem[Pau02]{Paulin} 
{\sc F. Paulin}, {\em Groupe modulaire, fractions continues et approximation diophantinne en 
carat\'eristique $p$},  Geom. Dedi. \textbf{95} (2002), 65-85.

\bibitem[Se70]{Se2}
{\sc  J.-P. Serre}, {\em Le Probl\`eme des Groupes de Congruence Pour $\text{SL}_2$}, Annals of Mathematics \textbf{92} (1970), 489-527.

\bibitem[Se80]{S}
{\sc J.-P. Serre},  {\em Trees}, Springer Verlag, Berlin, 1980.





\bibitem[So77]{So}
{\sc C. Soul\'e}, {\em Chevalley groups over polynomial rings}, Homological group theory (Durham, 1977), edited by C.T.C. Wall, London Math. Soc. Lectures Note Ser., Cambridge Univ. Press, \textbf{36} (1979), 359-361.




\bibitem[St93]{Stichlenoth}
{\sc H.~Stichtenoth}, {\em Algebraic Function Fields and Codes}, Springer-Verlag, Berlin, 1993.


\end{thebibliography}
\end{document}